\documentclass[11pt, reqno]{amsart}

\usepackage[T1]{fontenc}
\usepackage[latin1]{inputenc}
\usepackage[french, english]{babel}
\usepackage{amsthm}
\usepackage{amscd}
\usepackage{amssymb,amsmath,bbm}
\usepackage[all]{xy}
\usepackage{mathrsfs}
\usepackage[left=0.0cm,right=0.0cm,top=2.7cm,bottom=3cm]{geometry}
\usepackage{hyperref}
\usepackage{verbatim}

\theoremstyle{plain}
\newtheorem{theorem}{Theorem}[section]

\newtheorem{lemma}[theorem]{Lemma}
\newtheorem{proposition}[theorem]{Proposition}

\newtheorem{corollary}[theorem]{Corollary}

\newtheorem*{theorem*}{Theorem}

\newtheorem{definition}[theorem]{Definition}
\theoremstyle{definition}
\newtheorem{example}[theorem]{Example}
\newtheorem{notation}[theorem]{Notation}

\theoremstyle{remark}
\newtheorem{remark}[theorem]{Remark}

\usepackage[usenames,dvipsnames]{color}

\newcommand{\mycomment}[1]{%
}%

\usepackage{tikz}
\usetikzlibrary{calc}

\def\Q{{\bf Q}}

\def\Z{{\bf Z}}
\def\C{{\bf C}}
\def\N{{\bf N}}
\def\R{{\bf R}}

\def\O{{\mathcal{O}}}
\def\H{{H}}

\def\Qbar{{\overline{{\bf Q}}}}

\def\zp{{\Z_p}}

\def\zle{{\mathbf{Z}_\ell^\times}}
\def\qle{{\mathbf{Q}_\ell^\times}}

\def\qp{{\Q_p}}

\def\Hom{\mathrm{Hom}}

\def\A{\mathbf{A}}
\def\Af{{\bf A}_{f}}

\def\Sc{\mathcal{S}}
\def\epsilon{\varepsilon}
\def\det{\mathrm{det}}
\def\Sh{{\operatorname{Sh}}}

\def\GSp{{\mathrm{GSp}}}
\def\GSpin{{\mathrm{GSpin}}}
\def\Sp{{\mathrm{Sp}}}
\def\GL{\mathrm{GL}}

\def\G{\mathbf{G}}
\def\H{\mathbf{H}}

\def\B{\mathbf{B}}
\def\l{\ell}
\def\matrix#1#2#3#4{{\big(\begin{smallmatrix}#1&#2\\ #3&#4\end{smallmatrix}\big)}}

\def\Alog{\mathscr{A}_{\overline{X}}^*({\mathrm{log}} \, D)}

\def\Ard{\mathscr{A}_{rd}^*}

\def\Asi{\mathscr{A}_{si}^*}

\def\Dist{\mathscr{D}_{si}}

\title[Tempered currents and Deligne cohomology of Shimura varieties]{Tempered currents and Deligne cohomology of Shimura varieties, with an application to $\GSp_6$}
\author[J.I. Burgos Gil]{Jos\'e Ignacio Burgos Gil}
\address{J.I.B.G.: Nicol\'as Cabrera 13-15 28049 Madrid, Spain}
\email{burgos@icmat.es }
\author[A. Cauchi]{Antonio Cauchi}
\address{A.C.: Dept. of Mathematics, School of Science, Tokyo Institute of Technology, 2-12-1 Ookayama, Meguro-ku, Tokyo, 152-8551 Japan}
\email{cauchi.a.aa@m.titech.ac.jp}
\author[F. Lemma]{Francesco Lemma}
\address{F.L.: Universit\'e Paris Cit\'e, CNRS, IMJ--PRG, b\^atiment Sophie Germain, case 7012, 75205 Paris Cedex 13, France}
\email{francesco.lemma@imj-prg.fr}
\author[J. Rodrigues Jacinto]{Joaqu\'in Rodrigues Jacinto}
\address{J.R.J.: Aix--Marseille Universit\'e, CNRS, I2M, Campus de Luminy, Avenue de Luminy, Case 930, 13288 Marseille Cedex 9, France}
\email{joaquin.rodrigues-jacinto@univ-amu.fr}

\makeatletter
\def\@tocline#1#2#3#4#5#6#7{\relax
  \ifnum #1>\c@tocdepth 
  \else
    \par \addpenalty\@secpenalty\addvspace{#2}%
    \begingroup \hyphenpenalty\@M
    \@ifempty{#4}{%
      \@tempdima\csname r@tocindent\number#1\endcsname\relax
    }{%
      \@tempdima#4\relax
    }%
    \parindent\z@ \leftskip#3\relax \advance\leftskip\@tempdima\relax
    \rightskip\@pnumwidth plus4em \parfillskip-\@pnumwidth
    #5\leavevmode\hskip-\@tempdima
      \ifcase #1
       \or\or \hskip 1em \or \hskip 2em \else \hskip 3em \fi%
      #6\nobreak\relax
    \dotfill\hbox to\@pnumwidth{\@tocpagenum{#7}}\par
    \nobreak
    \endgroup
  \fi}
\makeatother
\usepackage{hyperref}

 \thanks{J.~I.~Burgos was partially supported by MINISTERIO
DE CIENCIA E INNOVACION research projects PID2019-108936GB-C21 and 
ICMAT Severo Ochoa project CEX2019-000904-S. A. Cauchi was supported
by the European Research Council (ERC) under the European Union's
Horizon 2020 research and innovation programme (grant agreement
No. 682152) as well as by the NSERC grant RGPIN-2018-04392, Concordia Horizon postdoc fellowship n.8009 and the JSPS Postdoctoral Fellowship for Research in Japan. F. Lemma was supported by the ANR ClapClap. J. Rodrigues
Jacinto was financially supported by the ERC-2018-COG-818856-HiCoShiVa and by the project ANR-19-CE40-0015
COLOSS}

\begin{document}

\begin{abstract}
We provide a new description of Deligne--Beilinson cohomology for any Shimura variety in terms of tempered currents. This is particularly useful for computations of regulators of motivic classes and hence to the study of Beilinson conjectures. As an application, we construct classes in the middle degree plus one motivic cohomology of Siegel sixfolds and we compute their image by Beilinson higher regulator in terms of Rankin-Selberg type automorphic integrals. Using results of Pollack and Shah, we relate the integrals to noncritical special values of the degree $8$ Spin $L$-functions, as predicted by Beilinson conjectures.
\end{abstract}

\maketitle
\selectlanguage{english}

\setcounter{tocdepth}{2}

\tableofcontents

\section{Introduction}

Beilinson conjectures are vast generalizations of the analytic class number formula. They relate special values of $L$-functions of motives to motivic cohomology classes via higher regulators. Very few cases of these conjectures are known (e.g. \cite{Beilinson}, \cite{Beilinson2}, \cite{Deninger1}, \cite{Deninger2}, \cite{Ramakrishnan}, \cite{brunault-chida}, \cite{Kings}, \cite{lemmarf}, \cite{PollackShahU21}), and they remain one of the main open problems in arithmetic geometry. \\

In this article, we give a new explicit description of Deligne--Beilinson cohomology for an arbitrary Shimura variety that provides a natural setting for studying Beilinson conjectures. In particular, this solves a gap appearing in several works on them.
As an application, we construct classes in the middle degree plus one motivic cohomology of Siegel sixfolds and, relying on  our new description of Deligne--Beilinson cohomology, we prove a formula relating their image under Beilinson regulator to certain adelic integrals of Rankin-Selberg type. These integrals are known to compute noncritical special values of $L$-functions of automorphic forms for $\GSp_{6}$, and hence our results give further evidence for Beilinson conjectures. \\

\subsection{Tempered currents and Deligne cohomology}

One successful approach to Beilinson conjectures for automorphic motives is the interpretation of Rankin--Selberg integrals as a Poincar\'e pairing between certain class in  Deligne cohomology coming from motivic cohomology (via Beilinson regulator) and the differential form associated with a cuspidal automorphic form. While applying this approach (cf. \cite[\S 7.4]{nekovar}), one encounters three difficulties:

\begin{enumerate}
\item The construction of potentially interesting classes in motivic cohomology.
\item The calculation of Beilinson regulator of these classes in terms of certain adelic integrals.
\item The calculation of these integrals as automorphic $L$-functions.
\end{enumerate}

In every case known by the authors, points $(1)$ and $(3)$ are closely related and get inspiration from the automorphic side: one starts from an integral representation of the $L$-function of a cuspidal automorphic form that admits a geometric incarnation, such as certain Rankin--Selberg type integrals for cohomological cuspidal  forms of a group admitting a Shimura variety. Concerning $(2)$, one way to compute the Poincar\'e duality pairing is to express Deligne cohomology of a Shimura variety $X$ in terms of currents that can be naturally evaluated against the cuspidal differential form. When the differential form can be taken to be holomorphic (as in the cases of $\GL_2$ or $\GL_2 \times \GL_2$), this is not particularly problematic as it extends to a toroidal compactification $\overline{X}$ and one can use Jannsen's work \cite{Jannsen} which describes Deligne cohomology of $X$ by currents on $\overline{X}$. In higher dimensional cases, however, one is forced to work with cuspidal differential forms which are not holomorphic and it is not clear whether one can extend them or not to the boundary (to the best of our knowledge this might not be case), but they are always rapidly decreasing along the boundary $D := \overline{X} - X$.\\

The main contribution of this article is the development of theoretical tools to solve step $(2)$ in a general setting. Our novel idea is to give a description of Deligne--Beilinson cohomology in terms of \textit{tempered currents} on $\overline{X}$, i.e. sheaves of continuous linear forms on smooth differential forms on $X$ which are rapidly decreasing along $D$. The main advantage of such a description is that these complexes of currents will not only naturally pair against rapidly decreasing differential forms by construction, but will also be covariant with respect to pushforward by proper maps.  These two properties are crucial for applications to the calculation of complex regulators in many well known constructions. \\



Let us now describe these results in more detail. Let $\overline{X}$ be a proper smooth complex variety, $D \subseteq \overline{X}$ be a normal crossing divisor, $X = \overline{X} - D$ and denote by $j : X \to \overline{X}$ the natural inclusion. In his seminal work \cite{DeligneII}, Deligne constructed a complex $\Omega_{\overline{X}}^*(\log \, D)$ of differential forms with logarithmic singularities along $D$, which is equipped with Hodge and weight filtrations $F$ and $W$. Moreover, if $\Omega^*_X$ denotes the sheaf of holomorphic differential forms on $X$, then there are natural quasi-isomorphisms of complexes $Rj_*\C \xrightarrow{\sim} Rj_*\Omega^*_X \xleftarrow{\sim} \Omega^*_{\overline{X}}(\log D)$ and $(\Omega^*_{\overline{X}}(\log \, D), W, F)$ is a mixed Hodge complex inducing the canonical mixed Hodge structure on the cohomology of $X$. The results of Deligne have been consequently extended by Navarro \cite{Navarro} and the first named author \cite{Burgos}, who showed that the mixed Hodge structure on the cohomology can be calculated by means of complexes of analytic and smooth differential forms with logarithmic singularities at the boundary. We will denote by $\mathscr{A}_{\overline{X}}^*(\log \, D)$ the complex of smooth differential forms with logarithmic singularities along $D$ constructed in \cite{Burgos}. This complex has also the advantage of capturing the underlying real structure that is not seen in $\Omega_{\overline{X}}(\log \, D)$, as well as to provide a resolution by fine sheaves, as they admit partitions of unity. \\

We now introduce some bigger sheaves. Locally around any point, one can find a coordinate system $(z_1, \hdots, z_k, z_{k+1}, \hdots z_{d})$ such that $\overline{X}$ is isomorphic to a polydisc $\Delta^d_r$ of dimension $d$ and some radius $r > 0$ and  that the normal crossing divisor $D$ is given by the equations $z_1 \hdots z_k = 0$. Slowly increasing (resp. rapidly decreasing) functions on $X$ are then defined locally by asking that for some $N \geq 0$ (resp. for all $N \leq 0$), there exists some constant $C$ (possibly depending on $N$) such that
\begin{equation} \label{EqIntrosi} |f(z) | \leq C \left( \prod_{i = 1}^k |\log|z_i|| \right)^N,
\end{equation}
and similar conditions for the derivatives of $f$ (cf. Definition \ref{Deflog}). Then one defines in an analogous way (cf. Definition \ref{Deflogdifferentialforms}) complexes $\mathscr{A}_{si}^*$, resp. $\mathscr{A}_{rd}^*$, of sheaves on $\overline{X}$ of slowly increasing and rapidly decreasing differential forms. These are complexes of fine sheaves equipped with a Hodge structure (given as usual by the type of a differential form) and with a real structure (given by real valued smooth differential forms). A new object introduced in this article is the complex $\Dist^*$ of sheaves on $\overline{X}$ of tempered (or slowly increasing) currents. Rapidly decreasing differential forms are naturally equipped with a Fr\'echet topology and we define $\Dist^{p, q}$ as the sheaves $U \mapsto \Gamma_c(U, \mathscr{A}^{d-p, d-q}_{rd})^*$ of continuous linear forms on compactly supported sections on $U$ of rapidly decreasing differential forms, where $U \subseteq \overline{X}$ is an open set. It is also a complex of fine sheaves and it is equipped with a Hodge filtration as well as with a natural real structure. Our first main result is the following.

\begin{theorem} [Theorem \ref{Propsicurrents}] \label{IntroTheoC}
The natural inclusions
\[ (\Omega^*_{\overline{X}}(\log \, D), F) \to (\mathscr{A}^*_{\overline{X}}(\log \, D), F) \to (\mathscr{A}^*_{si}, F) \to (\Dist^*, F) \]
are filtered quasi-isomorphisms. Moreover, the last two quasi-isomorphisms are compatible for the corresponding real structures.
\end{theorem}

The first quasi-isomorphism is one of the main results of \cite{Burgos}. The proof of the quasi-isomorphism $ (\Omega_{\overline{X}}^*(\log \, D), F) \to (\mathscr{A}_{si}^*, F)$ follows from \cite{HarrisZuckerIII} and \cite{KatoMatsubaraNakayama}. The last one does not seem to have been yet considered in the literature and is the one that will be most useful for applications to complex regulators. We prove this result by using a variant of the Bochner-Martinelli homotopy operator on the space of rapidly decreasing differential forms. As an application, we also give a description of the compactly supported cohomology of $X$ and its Hodge filtration in terms of rapidly decreasing differential forms on $\overline{X}$ (cf. Proposition \ref{PropOnCSC}), fixing a gap in \cite{Harrisdbarcohomology}, \cite{HarrisZuckerIII}. \\

We now explain the applications of Theorem \ref{IntroTheoC} to Deligne cohomology. For $p \in \Z$, let $\R(p)$ denote the subgroup $(2\pi i)^p \R$ of $\C$ and denote by the same symbol the constant sheaf with value $\R(p)$ on $X$. We recall the definition of Deligne--Beilinson cohomology of $X$. For any $p \in \Z$ and $n \in \N$, the cohomology groups $H^n_\mathcal{D}(X, \R(p))$ of $X$ with coefficients in $\R(p)$ are defined as the hypercohomology groups of the complex
\begin{equation}
\R(p)_{\mathcal{D}} := \text{cone}(Rj_*\R(p) \oplus F^p\Omega^*_{\overline{X}}(\log D) \rightarrow Rj_* \Omega^*_X)[-1],
\end{equation}
where the arrow is given by the difference of the natural maps. As notation suggests, $H^n_\mathcal{D}(X, \R(p))$ only depends only on $X$ and not on the choice of $\overline{X}$. Denote by $\Dist^*(\overline{X})$ the complex of global sections of $\Dist^*$, let $\mathscr{D}^*_{si, \R(p-1)}$ be the complex of sheaves of $\R(p-1)$-valued currents and denote by $\mathscr{D}^*_{si, \R(p-1)}(\overline{X})$ the corresponding complex of global sections. Immediately from Theorem \ref{IntroTheoC} we obtain the following description of Deligne--Beilinson cohomology.
 
\begin{corollary} [Theorem \ref{TheoremAHCcurrents}]  \label{IntroCoroAHC}
We have
\[ \R(p)_{\mathcal{D}} \simeq \mathrm{cone} \left( F^p \Dist^* \to \mathscr{D}^*_{si, \R(p-1)} \right)[-1]. \]
In particular, we have
\[ H^n_{\mathcal{D}}(X, \R(p)) \simeq \frac{\{ (S, T) : dS = 0, dT = \pi_{p-1}(S) \}}{d(\widetilde{S}, \widetilde{T})}, \]
where $(S, T) \in F^p \Dist^n(\overline{X}) \oplus \mathscr{D}^{n-1}_{si, \R(p - 1)}(\overline{X})$ and $d(\widetilde{S}, 
\widetilde{T}) = (d S, dT - \pi_{p - 1}(S))$.
\end{corollary}

Technically speaking, the above complex of currents calculates Deligne homology, that we implicitly identify with Deligne cohomology via Poincar\'e duality. The explicit description given by Corollary \ref{IntroCoroAHC} has the pleasant property of being covariant with respect to closed embeddings, which is key for applications as we now explain.

\subsection{Regulators of Siegel sixfolds and special values of Spin
  \texorpdfstring{$L$}{L}-functions}

Let $\G = \GSp_{6}$ be the symplectic similitude group of rank $3$. Denote by $\Sh_{\G}$ the six dimensional Shimura variety associated with $\G$. These Shimura varieties and their cohomology play a prominent role in the study of arithmetic aspects of cuspidal automorphic representations of $\G(\A)$ and their associated Galois representations. \\

Fundamental objects used in most of the approaches to Beilinson conjectures are modular units. These are elements in the motivic cohomology groups $H^1_\mathcal{M}(\Sh_{\GL_2}, \overline{\Q}(1)) \cong \mathcal{O}(\Sh_{\GL_2})^\times \otimes_\Z \overline{\Q}$ of the modular curves $\Sh_{\GL_2}$, which can be seen as motivic incarnations of Eisenstein series. Indeed, by the second Kronecker limit formula, their logarithm is related to limiting values of some real analytic Eisenstein series. 
Using these modular units, and inspired by the work of Pollack and Shah \cite{PollackShah}, we construct natural cohomology classes
\[ \mathrm{Eis}_{\mathcal{M}}(\Phi_f) \in H^{7}_\mathcal{M}(\Sh_{\G}(U), \overline{\Q}(4)), \]
attached to suitable Schwartz-Bruhat functions $\Phi_f$ on $\A_f^2$ and where $U \subseteq \G(\Af)$ denotes an appropriate level structure. The construction goes as follows. Let $F$ denote a real \'etale quadratic $\Q$-algebra, i.e. $F$ is either a totally real quadratic extension of $\Q$ or $\Q \times \Q$.
Denote by $\GL_{2,F}^* /\Q$ the subgroup scheme of ${\rm Res}_{F/\Q} \GL_{2,F}$ sitting in the Cartesian diagram 
\[ \xymatrix{ 
\GL_{2,F}^* \ar@{^{(}->}[r] \ar[d] & {\rm Res}_{F/\Q} \GL_{2,F} \ar[d]^{{\rm det}} \\ 
\mathbf{G}_m  \ar@{^{(}->}[r] & {\rm Res}_{F/\Q} \mathbf{G}_{{\rm m},F}.
}
\]
Let $\H$ denote the group 
$$
\H := \GL_2 \boxtimes  \GL_{2,F}^* = \{(g_1, g_2) \in \GL_2 \times  \GL_{2,F}^* \; | \; {\rm det}(g_1) = {\rm det}(g_2) \}.
$$
Then one has an embedding $\iota: \H \hookrightarrow \G$. For small enough level $U \subseteq \G(\A_f)$, this embedding induces a closed embedding
\[ \iota : \Sh_\H(U \cap \H(\A_f)) \to \Sh_\G(U) \]
of codimension $3$. Letting $\mathrm{pr}_1 : \H \to \GL_2$ denote the projection to the first factor, and  $u(\Phi_f) \in H^1_{\mathcal{M}}(\Sh_{\GL_2}, \overline{\Q}(1))$ be a modular unit associated with $\Phi_f$, one defines
$ \mathrm{Eis}_{\mathcal{M}}(\Phi_f) := \iota_*(\mathrm{pr}_1^* (u(\Phi_f)))$.    \\

Now let $\pi = \pi_\infty \otimes \pi_f$ be a cuspidal automorphic representation of $\G(\A)$  with trivial central character, for which $\pi_\infty$ is a discrete series of Hodge type $(3,3)$  and such that $\pi_f$ has a nonzero vector fixed by $U$. associated with a cusp form $\Psi = \Psi_\infty \otimes \Psi_f \in \pi_\infty \otimes \pi_f^U$ such that $\Psi_\infty$ is a highest weight vector of one minimal $K_\infty$-type of $\pi_\infty$, there is a harmonic differential form $\omega_\Psi$ on $\Sh_{\G}(U)$ of type $(3,3)$.
Via a careful study of Deligne--Beilinson cohomology, we define a natural pairing
\[ \langle \,\, , \omega_\Psi \rangle : H^{7}_{\mathcal{D}}(\Sh_{\G}(U)/\R, \R(4)) \otimes_{\Q} \overline{\Q} \rightarrow \C \otimes_{\Q} \Qbar{}, \]
where $H^*_{\mathcal{D}}(\Sh_{\G}(U)/\R, \R(-))$ denotes the Deligne--Beilinson cohomology group of $\Sh_{\G}(U)$.

Recall the existence of Beilinson regulator
\[  r_\mathcal{D}: H^{7}_{\mathcal{M}}(\Sh_{\G}(U), \overline{\Q}(4)) \to  H^{7}_{\mathcal{D}}(\Sh_{\G}(U)/\R, \R(4)) \otimes_{\Q} \overline{\Q}. \]
Let $\pi$ be a cohomological cuspidal automorphic representations of $\G(\A)$. In the generic sitaution, the $\pi_f$-isotypic component of $H^{7}_{\mathcal{D}}(\Sh_{\G}(U)/\R, \R(4)) \otimes_{\Q} \overline{\Q}$ has rank one (see section \ref{subsec:archimedeanLfunctionsandDelignecoho} below for more details). According to Beilinson conjectures, if the Spin $L$-function of $\pi$ is holomorphic at $s = 1$, then one expects to be able to construct nonzero motivic cohomology classes whose image under Beilinson regulator are related to special values of this $L$-function. One of the main goals of this article is to show that the classes $\mathrm{Eis}_\mathcal{M}(\Phi_f)$ satisfy this. We let $\Phi_\infty(x,y)=e^{-\pi(x^2+y^2)}$, $\Phi=\Phi_\infty \otimes \Phi_f$ and $E(h, \Phi,s)$ be the real analytic Eisenstein series on $\GL_2$ attached to $\Phi$ (cf. Equation \eqref{eisensteinseriesgl2}).

\begin{theorem}[Theorem \ref{adelicintegral}] \label{IntroTheoA} We  have
\begin{equation*}
\langle r_\mathcal{D}({\rm Eis}_{\mathcal{M}}(\Phi_f)), \omega_\Psi \rangle = C_{V}   \int_{\H(\Q) \Z_{\G}(\A) \backslash \H(\A)} E(h_1, \Phi, 0) (A \cdot\Psi)(h)dh,
\end{equation*}
where the operator $A$ is an element of $\mathcal{U}(\mathfrak{g}_\C)$ (defined precisely in \S \ref{adelicintegralsec}), the constant $C_{V}$ is an explicit volume factor depending only on $V:= U \cap \H(\A_f)$ and $dh$ is the Haar measure on $\H(\A)$ fixed in \S \ref{adelicintegralsec}.
\end{theorem}

\begin{remark}
 Our methods for proving Theorem \ref{IntroTheoA} as described below are flexible enough and translate verbatim to correct the aforementioned gap in the literature (\cite{Kings}, \cite[Proposition 5.4.6]{LLZ2} for Hilbert automorphic forms over a quadratic field, \cite{lemmarf} for $\GSp_4$ automorphic forms, and \cite{PollackShahU21} in the ${\rm GU}(2,1)$ case).
\end{remark}
The main theorem of \cite{PollackShah} gives a Rankin-Selberg formula for the Spin $L$-function of certain automorphic representations of $\G$. As a consequence of this and Theorem \ref{IntroTheoA}, we get the following result.

\begin{theorem} [Theorem \ref{regulator3}] \label{theo1intro}
Let $\pi$ be a cuspidal automorphic representation of $\G(\A)$ with trivial central character such that $\pi_\infty$ is a discrete series of Hodge type $(3,3)$. Let $\Sigma$ be a finite set of primes containing $\infty$ and the bad primes for $\pi$. Let $\Psi=\Psi_\infty \otimes \Psi_f$ be a factorizable cusp form in $\pi$ which is unramified outside $\Sigma$ and which supports a certain Fourier coefficient of type $(4 \, 2)$. Then
\[ \langle r_\mathcal{D}({\rm Eis}_{\mathcal{M}}(\Phi_f)), \omega_{\Psi} \rangle = C_V \lim_{s \to 0} \left( I_\Sigma(\Phi_\Sigma, A \cdot\Psi,s)  { L}^\Sigma(\pi, {\rm Spin} ,s) \right),\]
where ${L}^\Sigma(\pi, {\rm Spin} ,s) = \prod_{p \not\in \Sigma} {L}(\pi_p, {\rm Spin} ,s) $ and $I_\Sigma(\Phi_\Sigma,A \cdot\Psi,s)$ is the integral over the finite set of places $\Sigma$ as defined in Equation \eqref{semilcoaint}.
\end{theorem}

\begin{remark} \leavevmode
\begin{enumerate}
\item One can show (cf. Corollary \ref{regulator3coro}) that there exist a cusp form $\widetilde{\Psi} = \Psi_\infty \otimes \widetilde{\Psi}_f \in \pi$, with $\Psi$ and $\widetilde{\Psi}$ coinciding outside $\Sigma$, and a Schwartz function defining the Eisenstein series, such that the finite integral $I_\Sigma(\Phi_\Sigma,A \cdot\widetilde{\Psi},s)$ is equal to the archimedean integral $ I_\infty(\Phi_\infty, A \cdot\Psi, s)$, implying that \[ \langle r_\mathcal{D}({\rm Eis}_{\mathcal{M}}(\Phi_f)), \omega_{\widetilde{\Psi}} \rangle = C_U \lim_{s \to 0} \left(I_\infty(\Phi_\infty, A \cdot\Psi, s)  { L}^\Sigma(\pi, {\rm Spin} ,s)\right).\]
In the case where $L^\Sigma(\pi, {\rm Spin} ,s)$ does not have a pole at $s=1$, we expect the value $\lim_{s \to 0} \left(I_\infty(\Phi_\infty, A \cdot\Psi, s)  { L}^\Sigma(\pi, {\rm Spin} ,s)\right)$ to be equal, up to some nonzero constant, to the leading term of the Taylor expansion of ${ L}^\Sigma(\pi, {\rm Spin} ,s)$ at $s=0$ (cf. Remark \ref{remarkontaylorexpansion} for a more detailed discussion). We would like to point out that, thanks to \cite[Proposition 12.1]{GanGurevich},  this archimedean integral can be made nonzero at arbitrary $s=s_0$ if one has some freedom on the choice of $\Phi_\infty$ and $\Psi_\infty$. This implies that the archimedean integral does not vanish identically, however the main crux to improve this formula is to calculate it when $\Psi_\infty$ is the highest weight vector in the minimal $K_\infty$-type of $\pi_\infty$. In the case where $L^\Sigma(\pi, {\rm Spin} ,s)$ does have a pole at $s=1$, Beilinson conjectures predict that the pairing $\langle r_\mathcal{D}({\rm Eis}_{\mathcal{M}}(\Phi_f)), \omega_{\Psi} \rangle$ is zero and that there exists an algebraic cycle computing the residue of $L^\Sigma(\pi, {\rm Spin} ,s)$ at $s=1$. This has been investigated in \cite{CLRG2}.

\item The hypothesis in Theorem \ref{theo1intro} on the existence of a Fourier coefficient of type $(4\, 2)$ for $\pi$ does not always hold. However, if none of the irreducible cuspidal automorphic representations of ${\rm Sp}_6$ inside $\pi$ has cuspidal theta lift to the split ${\rm SO}_{12}$, $\pi$ supports such a Fourier coefficient (cf.  Lemma \ref{Fouriercoeff}).

\item The classes ${\rm Eis}_{\mathcal{M}}$ were used in \cite{GSp6paper1}, where the second and fourth authors showed that their \'etale realizations could be assembled into a norm-compatible tower of cohomology classes for some $p$-level subgroups, giving rise to an element of the Iwasawa cohomology of the local $p$-adic Galois representation associated with the automorphic representation of $\GSp_6$. The above theorem shows that the $p$-adic $L$-functions constructed in \cite{GSp6paper1} are indeed related to special values of complex $L$-functions in the spirit of Perrin-Riou conjectures.
\item As a corollary, Theorem \ref{theo1intro} gives a criterion for the nonvanishing of the motivic cohomology group in question, cf. Corollary \ref{nonvanish}.
\end{enumerate}
\end{remark}

Let us now explain the strategy for proving Theorem \ref{IntroTheoA}. One would like to calculate explicitly the pairing on the left hand side of Theorem \ref{IntroTheoA} provided by Poincar\'e duality between the class $r_\mathcal{D}(\mathrm{Eis}_\mathcal{M}(\Phi_f))$ and the class in compactly supported Betti cohomology defined by the cuspidal form $\omega_\Psi$. By the work of Jannsen, one can represent a class in Deligne cohomology as a pair of logarithmic currents on the toroidal compactification of the Shimura variety, i.e. linear forms of differential forms which extend to the boundary. Moreover, Poincar\'e duality is given by the natural evaluation of a current against a differential form on the open variety that can be extended to the whole compactification. In particular, one can evaluate these currents against differential forms which are compactly supported in the open variety or, by a well known result \cite[\S 4]{amrt}, top degree differential forms attached to holomorphic cuspidal automorphic forms. When the differential form $\omega_\Psi$ is not holomorphic, it is not clear whether we can extend it or not to the boundary, and to the best of our knowledge this might not be case. As in \cite{Kings} and \cite{lemmarf}, one could try to approximate the form $\omega_\Psi$ by closed differential forms with compact support at which we can evaluate these currents. However the main problem of this approach is that there is not enough control on the Hodge types of the compactly supported differential forms that approximate $\omega_\Psi$. Hence, for some technical reasons, it is not possible to show that the evaluation of the currents representing the class in Deligne cohomology at the compactly supported differential forms converges to the expected value $\langle r_\mathcal{D}(\mathrm{Eis}_\mathcal{M}(\Phi_f)), \omega_\Psi \rangle$. The explicit description of Deligne cohomology given by Corollary \ref{IntroCoroAHC} has the pleasant property of being covariant with respect to closed embeddings, which allows us to explicitly describe the image under the Gysin morphism $\iota_*$ of modular units as tempered currents, which can naturally be paired against the rapidly decreasing differential form $\omega_\Psi$. This, together with the second Kronecker's limit formula, allow us to conclude the proof of Theorem \ref{IntroTheoA}.

\subsection{Organization of the article} Section $2$ is the technical heart of the article and gives an explicit description of the Deligne--Beilinson cohomology of a smooth variety in terms of tempered currents. We also define natural pairings between elements in Deligne--Beilinson cohomology and certain closed rapidly decreasing differential forms. In Section $3$ we introduce the relevant reductive groups and the attached Shimura varieties and we explain the construction of the motivic cohomology classes. In Section $4$ we recall some classical results on the cohomology of the Siegel sixfold and we describe the construction of the harmonic differential form associated with a cusp form on $\GSp_6(\A)$.  Finally, in Section $5$, we prove Theorem \ref{IntroTheoA}, expressing the value of this linear form on the archimedean realization of our motivic classes in terms of Rankin-Selberg type integrals and, using the main result of \cite{PollackShah}, we relate these integrals to noncritical values of the Spin $L$-function of the fixed cuspidal automorphic representation of $\GSp_6(\A)$.

\subsection{Acknowledgements} It is a pleasure to thank the following people: Aaron Pollack and Shrenik Shah for pointing out a mistake in our arguments in a first version of this article as well as for various instructive discussions; Colette Moeglin for many useful conversations concerning discrete series and Fourier coefficients. Guido Kings and David Loeffler for their useful remarks on some of our calculations. The third named author would like to thank Pierre Colmez for his interest and suggestions on this article and for his constant support. Finally, we thank Michael Harris, R\'egis De la Bret\`eche, Taku Ishii, Tadashi Ochiai, Juan Esteban Rodr\'iguez Camargo, Giovanni Rosso, Benoit Stroh, Jun Su and Sarah Zerbes for useful correspondence or conversations related to this article. Finally, we would like to thank the anonymous referee for the valuable corrections and comments that helped us to considerably improve this article.

\section{Deligne--Beilinson cohomology and tempered currents} \label{sectiondelignecohomology}

The purpose of this chapter is to give a new explicit description of Deligne--Beilinson cohomology that will be suitable for the computations of the next chapter. In order to do this, we will explain how the Hodge filtration and the real structure on the cohomology groups of the Shimura variety are induced from Hodge complexes of differential forms with growth conditions and tempered currents. Since the results in this section might be of independent interest, we will work in a general setting that encompasses the case of all Shimura varieties.

Let $X$ denote a complex analytic variety which is smooth, quasi-projective and of pure dimension $d$. Let $\overline{X}$ be a smooth compactification of $X$ such that $D = \overline{X}-X$ is a simple normal crossing divisor. We denote by $j: X \rightarrow \overline{X}$ the open embedding. We will assume that $X$ is defined as the analytification of the base change to $\C$ of a smooth, quasi-projective $\R$-scheme. The complex conjugation $F_\infty$ is an antiholomorphic involution on $X$. For $p \in \Z$, let $\R(p)$ denote the subgroup $(2\pi i)^p \R$ of $\C$. We will denote by the same symbol the constant sheaf with value $\R(p)$ on $X$.

We recall the definition of Deligne--Beilinson cohomology ($DB$-cohomology for short) of $X$. Let $\Omega^*_X$ be the sheaf of holomorphic differential forms on $X$ and let $\Omega^*_{\overline{X}}(\log D)$ be the sheaf of holomorphic differential forms on ${X}$ with logarithmic poles along $D$ (see \cite[\S 3.1]{DeligneII}). The Hodge filtration on $\Omega^*_{\overline{X}}(\log D)$ is defined as $F^p\Omega^*_{\overline{X}}(\log D)=\bigoplus_{p' \geq p} \Omega^{p'}_{\overline{X}}(\log D)$. There are natural quasi-isomorphisms of complexes $Rj_*\C \rightarrow Rj_*\Omega^*_X$ and $\Omega^*_{\overline{X}}(\log D) \rightarrow Rj_* \Omega^*_X$ (see \cite{DeligneII} or \cite{Jannsen} for the basic facts used here). For any $p \in \Z$, the $DB$-cohomology groups $H^n_\mathcal{D}(X, \R(p))$ of $X$ with coefficients in $\R(p)$ are defined as the hypercohomology groups of the complex
\begin{equation} \label{defDBcohom}
\R(p)_{\mathcal{D}} := \text{cone}(Rj_*\R(p) \oplus F^p\Omega^*_{\overline{X}}(\log D) \rightarrow Rj_* \Omega^*_X)[-1],
\end{equation}
where the arrow is given by the difference of the natural maps. Let $\overline{F_\infty^*} = F_\infty^* \otimes c$ be the de Rham involution given by the action of the complex conjugation on $X$ and on the coefficients. The real $DB$-cohomology groups are defined as
\begin{equation} \label{DefRDB} H^n_\mathcal{D}(X / \R, \R(p)) = H^n_\mathcal{D}(X, \R(p))^{\overline{F^*_\infty} = 1}.
\end{equation}
The main idea is that one can replace the complexes appearing in Equation \eqref{defDBcohom} by any complexes that are quasi-isomorphic to them, respecting the different structures involved, namely the Hodge filtration and the real structure. We will start in \S \ref{SectionFunctionSpaces} by defining local spaces of slowly increasing and rapidly decreasing differential forms and in \S \ref{SubSectionBMop} we introduce the Bochner-Martinelly homotopy operator for compactly supported rapidly decreasing differential forms (Theorem \ref{LemmaBM}). In the final parts of this section, we apply these results to the global situation. We define complexes of fine sheaves of rapidly decreasing and slowly increasing differential forms, and of tempered currents. We show (Proposition \ref{PropOnCSC}) that the complex of rapidly decreasing differential forms calculates the compactly supported cohomology of $X$, that the complex of sheaves of tempered currents calculates the cohomology of $X$ (Theorem \ref{Propsicurrents}) and we finish with the description of Deligne cohomology of $X$ in terms of tempered currents (Theorem \ref{TheoremAHCcurrents}).

\subsection{Local functional spaces} \label{SectionFunctionSpaces}
If $0 < r < 1$, we denote by $\Delta_r^* = \{ z \in \C \; : \; 0 < |z| < r \}$ the open punctured disc of of radius $r$ and $\Delta_r = \{ z \in \C \; : \; |z| < r \}$ the open disc of radius $r$. We will write $\Delta := \Delta_{1/2}$. If $m, n \in \N$, we will denote by $z = (z_1, \hdots, z_m, z_{m+1}, \hdots, z_{m+n})$ the coordinates of $(\Delta^*_r)^m \times (\Delta_r)^n$. Let $d=n+m$ be the dimension of this last space. We start by fixing some useful notations.

\begin{notation} \label{notationmulti} Once coordinates as before have been chosen we will use the following notation.
\[
  \xi_{i} =
  \begin{cases}\displaystyle
  \frac{dz_{i}}{z_{i}}&\text{ for }i\le m,\\
  dz_i&\text{ for }i> m, 
  \end{cases} \label{xii}
\]
Note that, formally we have $\xi_i = d \log z_i$ for $i \leq m$. We will denote by $\overline {\xi}_i$ the complex conjugate form. Given a multi-index $\alpha=(\alpha_1,\dots,\alpha_{d}) \in \N^{d}$ we will denote
  \begin{displaymath}
    |\alpha|=\sum_{i=1}^{d}\alpha_i, \quad z^{\alpha}=\prod_{i=1}^{d}z_i^{\alpha_i}.
  \end{displaymath}
  and similar notation for $\overline {z}^{\alpha}$. 
  If $\beta \in \N^d$ is another multi-index, we will write
  \begin{displaymath}
  \partial^{\alpha,\beta} = \prod_{j = 1}^d \left(\frac{\partial}{\partial z_j}\right)^{\alpha_j} \left(\frac{\partial}{\partial \overline{z}_j}\right)^{\beta_j}. 
  \end{displaymath} 
  The symbol $\pmb{0}$ will denote the multi-index $(0,\dots,0)$, while the symbol $\pmb{1}$ will denote the multi-index with all entries equal to $1$. Given two multi-indices the addition is given componentwise. As an example of how this notation works, we have
  \begin{displaymath}
  z^{\alpha + \pmb{1}}=  \prod_{i=1}^{d}z_i^{\alpha_i+1}.
  \end{displaymath}
 By a slight abuse, we might sometimes note $|z|^\alpha = |z^\alpha| = \prod_{i =1}^d |z_i|^{\alpha_i}$.
  We will also use the following shorthands  $dz=dz_1\wedge\dots\wedge dz_d$ and $d\overline {z}=d\overline{z}_1\wedge\dots\wedge d\overline{z}_d$ and
  \[
 \big|\log|z|\big|^{N}=\prod_{i=1}^{d}\big|\log|z_i|\big|^N.
 \]
 We next add subsets to the notation. 
  The symbol $[a,b]$ will denote the set of integers between $a$ and $b$. 
  For any subset $I\subset [1,d]$ and any $d$-tuple $z=(z_1,\dots,z_d)$, we will denote by $z_I$ the object of dimension $|I|$ that agrees with $z$ for all entries in $I$. Similarly for a multi-index $\alpha$ the symbol $\alpha_{I}$ will denote the multi-index that agrees with $\alpha$ for all entries in $I$ and is zero otherwise. 
 
  We will adapt all the previous notations to take care of subsets. For instance, if $\alpha=(\alpha_1,\dots,\alpha_d)$, then 
  \[
  z_I^{\alpha}=z^{\alpha_I}=\prod_{i\in I}z_i^{\alpha_i}.
  \]
 Or, if $I=\{i_1<\dots<i_k\}$, we write 
 \begin{displaymath}
   dz_{I}=dz_{i_1}\wedge\dots \wedge dz_{i_k},\qquad 
      d\overline{z}_{I}=d\overline{z}_{i_1}\wedge\dots \wedge d\overline{z}_{i_k}.
 \end{displaymath} 
 Similarly we will write 
 \begin{displaymath}
    \xi_{I}=\xi_{i_1}\wedge\dots \wedge \xi_{i_k},\qquad
    \overline{\xi}_{I}=\overline{\xi}_{i_1}\wedge\dots \wedge \overline{\xi}_{i_k}.
 \end{displaymath}
 Also, if $N$ is an integer and $I\subset [1,d]$ is a subset, we will write 
 \[
 \big|\log|z_I|\big|^{N}=\prod_{i\in I}\big|\log|z_i|\big|^N.
 \]
 Finally, we will use the shorthand
  \begin{displaymath}
    z^{I}=z_{I}^{\pmb{1}}=\prod_{i\in I}z_i.
  \end{displaymath}
\end{notation}

\begin{definition} \label{Deflog} Let $f$ be a $\mathcal{C}^\infty$-function on $(\Delta^*)^m \times (\Delta)^n$. 
\begin{enumerate}
    \item The function $f$ is said to be \emph{slowly increasing}  if there exists an $N\in \N$  such that for all $\alpha, \beta \in \N^d$, there exists a constant $C$ such that
\[ |(\partial^{\alpha,\beta}f)(z)| \leq C \frac{| \log |z_{[1,m]}||^{N}}{\big|z_{[1,m]}^{\alpha+\beta}\big|}. \]
\item The function $f$ is said to be \emph{rapidly decreasing} if for all $N\in \N$ and all $\alpha, \beta \in \N^d$,  there exists a constant $C$ such that
\[ |(\partial^{\alpha,\beta}f)(z)| \leq C \frac{| \log |z_{[1,m]}||^{-N}}{\big|z_{[1,m]}^{\alpha+\beta}\big|}. \]
\end{enumerate}
\end{definition}

\begin{definition}
We denote by $A^0_{rd}$ and $A^0_{si}$ the spaces of rapidly decreasing and slowly increasing functions on $(\Delta^*)^m \times \Delta^n$, respectively.  We define the space of rapidly decreasing (resp. slowly increasing) differential forms $A^*_{rd}$ (resp. $A^*_{si}$) to be the $A^0_{rd}$ (reps. $A^0_{si}$)-subalgebra of the algebra of smooth differential forms on $(\Delta^*)^m \times \Delta^n$ freely generated by the differential forms $\xi_i$ and $\overline{\xi}_i$ for $1 \leq i \leq d$. A differential form, either rapidly decreasing or slowly increasing, is said to have relatively compact support if its support is relatively compact in $\Delta^d$, or equivalently if it is contained $\Delta_r^d \subseteq \Delta^d$ for some $r < 1/2$. The subspace of compactly supported rapidly decreasing functions will be denoted by $A^0_{rd,c}$.
We analogously define the spaces $A_{rd}^{p,q}$, $A^{p,q}_{si}$ and $A^{p,q}_{rd, c}$ for any $0 \leq p,q \leq d$.
\end{definition}

Explicitly, any $\eta(z) \in A^{p, q}_{rd}$ (resp. $A^{p, q}_{rd, c}$, resp. $A^{p, q}_{si}$) can be written as
\[ \eta = \sum_{|I| = p, |J|=q} f_{I, J}(z) \xi_I \wedge \overline{\xi}_J, \]
with $f_{I, J}(z) \in A^0_{rd}$ (resp. $A^0_{rd, c}$, resp. $A^0_{si}$) and $I, J \subseteq [1,d]$.

\begin{remark}
The growth condition on the derivatives in Definition \ref{Deflog} comes from asking that all the exterior higher derivatives of $f$ will satisfy the same growth conditions. For example, if $m = 1, n = 0$, i.e. working on  $\Delta^*$, and if $f(z)$ is a rapidly decreasing function, one wants $\overline{\partial} f (z) = \frac{\partial f}{\partial \overline{z}} d \overline{z} = \overline{z} \frac{\partial f}{\partial \overline{z}} \frac{d \overline{z}}{\overline{z}}$ of rapid decay, which translates precisely into asking that the function $\overline{z} \frac{\partial f}{\partial \overline{z}}$ is of rapid decay. This will assure in particular that the complexes of differential forms defined below are well defined. We also point out that a slowly increasing function is equivalent to a log smooth function on $\Delta^{d}$ equipped with the log structure given by the normal crossing divisor $\{ z_1 \hdots z_m = 0\}$, cf. \cite[Lemma 3.4(2)]{KatoMatsubaraNakayama}. Finally, we observe that this is the usual condition for defining rapidly decreasing functions taking into account the logarithmic change of variables on polar coordinates.
\end{remark}

Let us recall the definition of the Schwartz topologies on the space of rapidly decreasing functions.

\begin{definition} \label{DefTopologyrd}
We equip the space $A^0_{rd}$, resp. $A_{rd, c}^0$, of rapidly decreasing, resp. rapidly decreasing and compactly supported, functions on $(\Delta^*)^m \times \Delta^n$ with the usual Schwartz topology. More precisely $A^0_{rd}$ has the Fr\'echet topology given by the family of seminorms
\[ p_{N, \alpha,\beta}(f) := \sup_{\alpha',\beta',z} \big|z_{[1,m]}^{\alpha'+\beta'}\big| \big| \log|z_{[1,m]}| \big|^N |\partial^{\alpha',\beta'} f(z)|,\qquad (N \in \N,\quad \alpha,\beta \in \N^d),\] 
where $\alpha'$ and $\beta'$ run over all multi-indices satisfying $\alpha'\le \alpha$ and $\beta'\le \beta$. 
For each compact subset $K\subset \Delta^d$, we denote  by $A^{0}_{rd, K}$ the space of rapidly decreasing functions supported on $K$, equipped with the Fr\'echet topology given by the same family of seminorms.
Then 
\begin{displaymath}
A_{rd, c}^0 =  \varinjlim_{K} A^0_{rd, K},
\end{displaymath}
equipped with the direct limit topology of locally
convex vector spaces.
 We note $p_N$ the norm $p_{N, \pmb{0},\pmb{0}}$. The space $A^{p, q}_{si}$ (resp. $A^{p, q}_{rd, c}$, resp. $A^{p, q}_{si}$) is a free module over $A^0_{rd}$ (resp. $A^0_{rd, c}$, resp. $A^0_{si}$) and it is equipped
with the induced topology.
\end{definition}

\subsection{The Bochner-Martinelli operator on rapidly decreasing differential forms} \label{SubSectionBMop}

The Bochner-Martinelli operator (cf. \cite{GriffithsHarris} or \cite{Laurent}) is classically used to show higher analogues of the Cauchy-Riemann equations in complex analysis and to show the $\overline{\partial}$-Poincar\'e lemma for currents. We will introduce a variant of the Bochner-Martinelli operator on rapidly decreasing differential forms which will be used later to show the $\overline{\partial}$-Poincar\'e lemma for tempered currents. These calculations are inspired by the work of Harris and Phong \cite{HarrisPhong}, and the ideas go back to Borel \cite{borel1}. Throughout this section we will work on $(\Delta^*)^m \times \Delta^n$, with $\Delta = \Delta_{1/2}$, $m,n \in \N$ and we let $d = m+n$ as before.

Let $\eta$ be a smooth differential form on $(\Delta^*)^m \times \Delta^n$ of type $(0, q)$ whose support is relatively compact in $\Delta^d$. The Bochner-Martinelli operator is defined as
\begin{equation} \label{EquationDefBMFormula}
(K \eta)(z) = \int_{\Delta^d} k(z, w) \wedge \eta(w),
\end{equation} 
where, if one sets
\[ \Phi_i(\zeta) = (-1)^{i-1} \zeta_i d \zeta_1 \wedge \hdots \wedge \widehat{d \zeta_i} \wedge \hdots \wedge d \zeta_d, \]
\[ \Phi(\zeta) = d \zeta_1 \wedge \hdots \wedge d \zeta_d, \]
the kernel $k(z, w)$ of the operator is given by
\[ k(z, w) = \frac{(d - 1)!}{(2 \pi i)^d} \frac{ \sum_{i = 1}^d \overline{\Phi_i(z - w)} \wedge \Phi(w)}{\| z- w\|^{2d}}, \]
for $\| z-w \| = \left( \sum_{i=1}^d |z_i - w_i|^2 \right)^{\frac{1}{2}}$ denoting the usual Euclidean norm. The Bochner-Martinelli kernel is a locally integrable differential form on $\C^d \times \C^d \backslash \{ z = w\}$ that decomposes as
\[ k(z, w) = \sum_{q' = 1}^d k^{{q'}}(z, w), \]
where $k^{q'}(z, w)$ is of type $(0, q'-1)$ on the variable $z$ and of type $(d , d - q')$ on the variable $w$. In particular, when $\eta$ has type $(0,q)$ and the integral defining $K\eta$ is convergent, then it defines a form of type $(0,q-1)$.

Let
$$\eta(z) = \sum_{|I|=p}\eta_I(z) \wedge \xi_{I} \in A^{p,q}_{rd, c},$$ where the sum is taken over the subsets $I \subseteq [1,d]$ of cardinality $p$, be a compactly supported rapidly decreasing form of type $(p, q)$ on $(\Delta^*)^m \times \Delta^n $. Then we can write
$$
\eta_I(z) = \sum_{|J|=q} f_{I,J}(z) \overline{\xi}_J
$$
and, for every subsets $I,J \subseteq [1,d]$ and every $N \in \N$, we have
\begin{equation} \label{eqestimaterd}
|f_{I,J}(z)| \leq C \big| \log |z_{[1,m]}|\big|^{-N}
\end{equation}
for some constant $C > 0$ and the corresponding estimates hold (with other constants) for all derivatives of $f_{I,J}$ as in Definition \ref{Deflog}.

\begin{definition} \label{DefinitionBM}
Let $\eta(z) = \sum_{|I|=p}\eta_I(z) \wedge \xi_I \in A^{p,q}_{rd, c}$ be a rapidly decreasing differential form on $(\Delta^*)^m \times \Delta^n$ with relatively compact support. We define
\[ (K' \eta)(z) := \sum_{|I|=p}z^{[1,m]}  K\left(\frac{\eta_I(w) }{w^{[1,m]}}\right)(z)\wedge \xi_I. \]
\end{definition}

Our first purpose is to show that the operator $K'$ is well defined and that it gives a rapidly decreasing differential form. We start by expanding the definition in order to have a more explicit expression for it. We have
\begin{eqnarray*}
(K' \eta_I)(z) &=& \sum_{|J|=q}z^{[1,m]} \int_{\Delta^d} k(z, w) \wedge \frac{f_{I,J}(w)}{w^{[1,m]} \overline{w}^{J\cap[1,m]} } d\overline{w}_J  \\
&=& \frac{(d - 1)!}{(2 \pi i)^d}  \sum_{|J|=q} \sum_{i=1}^d z^{[1,m]}\int_{\Delta^d} \frac{\overline{\Phi_i(z - w)} \wedge \Phi(w)}{\|z - w\|^{2d}} \wedge \frac{f_{I,J}(w)}{w^{[1,m]} \overline{w}^{J \cap [1,m]} }  d\overline{w}_J.
\end{eqnarray*}
Observe that for $i \notin J$ we formally have 
\[
\int_{\Delta^d} \frac{\overline{\Phi_i(z - w)} \wedge \Phi(w)}{\|z - w\|^{2d}} \wedge \frac{f_{I,J}(w)}{w^{[1,m]} \overline{w}^{J \cap [1,m]} }  d\overline{w}_J = 0,
\] 
and for $i \in J$ we get
\begin{multline*}
    \int_{\Delta^d} \frac{\overline{\Phi_i(z - w)} \wedge \Phi(w)}{\|z - w\|^{2d}} \wedge \frac{f_{I,J}(w)}{w^{[1,m]} \overline{w}^{J \cap [1,m]} }  d\overline{w}_J =\\ \pm \left( \int_{\Delta^d} \frac{(\overline{z}_i - \overline{w}_i) f_{I,J}(w)}{\|z - w \|^{2d} w^{[1,m]} \overline{w}^{J \cap [1,m]} } dw \wedge d \overline{w} \right) \wedge d \overline{z}_{J-\{i\}}.  
\end{multline*}
where the $\pm$ sign comes from exchanging the order of the differentials. Putting this together, we obtain that, up to a sign, one can write
\begin{equation} \label{expressionK'}
(K' \eta_I)(z) = \frac{(d - 1)!}{(2 \pi i)^d} \sum_{i=1}^d\sum_{|J|=q} g_{i,I,J}(z) \wedge d \overline{z}_{J-\{i\}},
\end{equation}
where we denote
\begin{equation} \label{EquationDefgIJ}
g_{i,I,J}(z) :=  z^{[1,m]} \int_{\Delta^d} \frac{(\overline{z}_i - \overline{w}_i) f_{I,J}(w)}{\|z - w \|^{2d} w^{[1,m]} \overline{w}^{J \cap [1,m]} } dw \wedge d \overline{w}.
\end{equation}
In the following Lemmas we shall first prove the absolute convergence of the integral defining $g_{i, I,J}(z)$ and second that $g_{i, I,J}(z)$ is rapidly decreasing. Recall that when $i\not \in J$, the function $g_{i,I,J}$ is
zero, so there is nothing to prove. Hence we can assume that $i\in J$.

\begin{lemma} \label{LemmaBMconvergence} Let $f 
\in A^0_{rd, c}$ be any rapidly decreasing function on $(\Delta^*)^m \times \Delta^n$ with relatively compact support and let $i \in J \subseteq [1,d]$. Then, for any $z \in (\Delta^*)^m \times \Delta^n$, the integral
\[ \int_{\Delta^d} \frac{(\overline{z}_i - \overline{w}_i) f(w)}{\|z - w \|^{2d} w^{[1,m]} \overline{w}^{J\cap[1,m]} } dw \wedge d \overline{w}  \]
is absolutely convergent.
\end{lemma}

\begin{proof} 
Let us fix $z \in (\Delta^*)^m\times \Delta^n$. Write $f = f_1  + f_2$, where $f_1$ is a smooth function which is supported in a sufficiently small polydisc $B(z, 2a) := \prod_{j = 1}^d B(z_j, 2a)$ around $z$ and $f_2$ vanishes identically on $\prod_{j = 1}^d B(z_j, a)$, with $a > 0$ small enough so that, for $w\in B(z,3a)$, $w^{[1,m]}\not = 0$. Since $|z_i - w_i| \leq \| z - w \|$, we have
\begin{eqnarray*}
\int_{\Delta^d} \bigg| \frac{(\overline{z}_i - \overline{w}_i) f_1(w)}{\|z - w \|^{2d} w^{[1,m]} \overline{w}^{J\cap [1,m]}} dw \wedge d \overline{w} \bigg| &\leq & C \int_{B(z, 2a)} \frac{1}{\| z - w \|^{2d -1}} |dw \wedge d\overline{w}| < +\infty,
\end{eqnarray*}
for some $C > 0$ and where the last integral converges from the standard fact that $|dw \wedge d\overline{w}|$ is twice the Lebesgue measure and that $\int_{B^n(0, 1)} \frac{dx}{\| x\|^\alpha}$ converges if $\alpha < n$ ($B^n(0, 1) \subseteq \R^n$ is the unit ball). On the other hand, fixing an integer $N \geq 2$, we have
\begin{eqnarray*}
 \int_{\Delta^d}  \bigg|\frac{ (\overline{z}_i - \overline{w}_i) f_2(w)}{\|z - w \|^{2d} w^{[1,m]} \overline{w}^{J\cap[1,m]}} dw \wedge d \overline{w} \bigg| &\leq& C' \int_{\Delta^d}\frac{|\log|w_{[1,m]}||^{-N}}{|w^{[1,m]}|^2} | dw \wedge d\overline{w}| < +\infty,
\end{eqnarray*}
for some constant $C' > 0$. Here we used that the integral $\int_{\Delta^d}\frac{|\log|w_{[1,m]}||^{-N}}{|w^{[1,d]}|^2} | dw \wedge d\overline{w}|$ converges for $N \geq 2$ as it can be written as the product of the integrals
\[ \int_{\Delta}\frac{|\log|w_j||^{-N}}{|w_j|^2} | dw_j \wedge d\overline{w}_j| = 4 \pi \int_{0}^{1 / 2} \frac{|\log(\rho)|^{-N}}{\rho} d\rho = 4 \pi \frac{\log(2)^{-N + 1}}{N - 1} < +\infty. \]
\end{proof}

The following elementary lemmas will be useful in the proof of Lemma \ref{LemmaBMestimation}. We warn the reader that various auxiliary constants will appear and they will sometimes be denoted by the same letter $C$, which should cause no confusion. As one can easily check, they will mostly depend only on $d$ or on $d$ and $N$ (but not on the variables $z$ nor $w$) and they can be made explicit, but this is not relevant for our purposes.

\begin{lemma}\label{lemm:3} Let $z\in \Delta $ and $D_{z}=\{w\in
  \Delta \mid |z-w|\le |z|/3\}$. 
   The following estimations hold.
  \begin{enumerate}
  \item\label{item:1} If $w\in D_{z}$, then
    \begin{equation}\label{eq:1}
      |z|\ge 3|z-w|,\quad 
      |w|\ge 2|z-w|, \quad 
      2|z|\le 3|w|\le 4|z|.
    \end{equation}
    Moreover, if $N\ge 0$, there is a constant $C>0$ independent of $z$ and $w$
    such that
    \begin{equation}
      \label{eq:4}
      \big|\log|w|\big|^{-N}\le C \big|\log|z|\big|^{-N}.
    \end{equation}
  \item\label{item:2} If $w\in \Delta \smallsetminus D_{z}$, then
    \begin{equation}
      \label{eq:5}
      |z|\le 3|z-w| \le 9|z|,\quad
      |w|\le 4|z-w|.
    \end{equation}
    Moreover, if $N\ge 0$ there is a constant $C>0$ independent of $z$
    and $w$ such that
    \begin{equation}
      \label{eq:7}
      \frac{\big|\log|w|\big|^{-N}}{|z-w|}\le C\frac{\big|\log|z|\big|^{-N}}{|z|}.
    \end{equation}
  \end{enumerate}
\end{lemma}
\begin{proof}
  We give only the proof of the estimate \eqref{eq:7} as all the other
  estimates are a consequence of the triangle inequality. We separate two
  cases. If $|w|\le |z|^{3/2}$, then
  \begin{displaymath}
    \frac{\big|\log|w|\big|^{-N}}{|z-w|}\le \frac{(2/3)^{N}\big|\log|z|\big|^{-N}}{|z|/3}.
  \end{displaymath}
  While, if $|w|\ge |z|^{3/2}$, then 
  \begin{displaymath}
    \frac{\big|\log|w|\big|^{-N}}{|z-w|}\le \frac{C}{|w|}\le
    \frac{C}{|z||z|^{1/2}}\le \frac{C'\big|\log|z|\big|^{-N}}{|z|}. 
  \end{displaymath}
\end{proof}

\begin{lemma} \label{lemm:2}
Let $d \in \N$. Then there exists some constant $C_d$ such that, for any $\ell < d$, any $r_{1},\dots,r_{\ell} \le 1/2$ and $a > 0$, 
\begin{equation} \label{eq:9}
\int_{\Delta _{r_{1}}\times \dots\times
    \Delta_{r_\ell}}
    \frac{|dz\wedge
      d\bar z |}{(a+\|z\|^{2})^{d}} \le C_d \frac{1}{a^{d-\ell}}.
\end{equation}
\end{lemma}

\begin{proof}
Straightforward by computing the integral in polar coordinates.
\end{proof}

The following lemma is the first step for showing that the form
$K'\eta$ is actually a rapidly decreasing differential form on $(\Delta^*)^m \times \Delta^n$. Recall that,  if $f \in A_{rd, c}^{0}$ is a rapidly decreasing function on $(\Delta^*)^m \times \Delta^n$ with relatively compact support, then, for any $N \in \N$, we have
\begin{equation} \label{EquationTmpEstimates}
  |f(z)| \leq C_{f,N} \big|\log|z_{[1,m]}|\big|^{-N}=C_{f,N} \prod_{i=1}^m \big|\log|z_i|\big|^{-N},
\end{equation}
for $z \in (\Delta^*)^m \times \Delta^n$, where we denoted $C_{f, N} := p_N(f)$.

\begin{lemma} \label{LemmaBMestimation} 
 Let $f 
\in A^0_{rd, c}$ be any rapidly decreasing function on $(\Delta^*)^m \times \Delta^n$ with relatively compact support and let $J \subseteq [1,d]$. For $i \in J$, let
  \[
  g_i(z)=z^{[1,m]}\int_{\Delta^d} \frac{(\overline{z}_i - \overline{w}_i) f(w)}{\|z - w \|^{2d} w^{[1,m]} \overline{w}^{J\cap[1,m]} } dw \wedge d \overline{w}
  \] 
be the function on $(\Delta^*)^m \times \Delta^n$, which is well defined by Lemma \ref{LemmaBMconvergence}. Then, for any $N \in \N$, there exists a universal constant $C_{d, N}$, depending only on $d$ and $N$, such that for any $z \in (\Delta^*)^m \times \Delta^n$ we have
\[ |g_i(z)| \leq C_{d, N} C_{f, N} \frac{ \big| \log |z_{[1,m]}| \big|^{-N+2}}{|z^{J \cap [1,m] \smallsetminus \{i\}}|}. \]
\end{lemma}

\begin{proof}
By using the estimates of Equation \eqref{EquationTmpEstimates} for $f$, we have
\begin{eqnarray*}
|g_i(z)| &\leq&  |z^{[1,m]}| \int_{\Delta^d} \frac{|\overline{z}_i - \overline{w}_i| |f(w)|}{\|z - w \|^{2d} |{w}^{[1,m]}| |\overline{w}^{J \cap [1,m]}|} |dw \wedge d \overline{w}|\\
&\leq& C_{f, N} |z^{[1,m]}| \int_{\Delta^d} \frac{|\overline{z}_i - \overline{w}_i\big| \log |w_{[1,m]}|\big|^{-N}}{\|z - w \|^{2d} |{w}^{[1,m]}| |\overline{w}^{J \cap [1,m]}|} |dw \wedge d \overline{w}|.
\end{eqnarray*}

We discuss first the case $d=1$. We assume that $m = 1$ and so $J= \{1\}$, since if not the result is trivial. We have
\[ g(z) = z \int_{\Delta} \frac{f(w)}{(z - w) w \overline{w} }
  dw \wedge d \overline{w}. \]
Let $r := |z|$ and decompose $\Delta^*$ into two regions
\[ \Delta^* = D_z \cup D_z ^c\]
where $D_z$ is as in Lemma \ref{lemm:3} and $D_z^c = \Delta \smallsetminus D_z$ denotes its complement. Using Lemma
\ref{lemm:3}(\ref{item:1}) and the change of variables $\zeta=w-z$, we obtain that 
there is a universal constant $C$ such that 
\begin{displaymath}
|z| \int_{D_z} \frac{\big| \log |w| \big|^{-N}}{|z - w| |w|^2}| dw \wedge
d\overline{w}| \leq  C  \frac{\big|\log |z| \big|^{-N}}{|z|} \int_{B(0,|z|/3)}
\frac{1}{|\zeta|} |d\zeta \wedge d\overline{\zeta}|
= \frac{4 \pi}{3} C {\big|\log |z|\big|^{-N}},
\end{displaymath}
where the last integral is calculated using polar coordinates.
On the other hand, using  Lemma \ref{lemm:3}(\ref{item:2}), we deduce
\begin{displaymath}
  |z| \int_{D_z^{c}} \frac{\big| \log |w| \big|^{-N}}{|z - w| |w|^2}| dw \wedge
  d\overline{w}|\le
  C |z|\frac{\big| \log |z| \big|^{-N+2}}{|z |}\int_{\Delta }\frac{| dw \wedge
  d\overline{w}|}{|w|^2 \big|\log|w|\big|^2}\le C' \big| \log |z| \big|^{-N+2},
\end{displaymath}
for some other constants $C$ and $C'$. Putting this together, we obtain the desired bound in dimension $1$.


We next discuss the general case $d > 1$. For simplicity of notation, we will assume that $m=d$, since the function $f$ is smooth on the whole $\Delta_j$ with respect to the variables $w_j$ for $j \in [m+1, d]$ and all the corresponding calculations will be simpler for those coordinates. In particular, in the formulas above we have $|w^{[1,m]}| = |w^{[1,d]}|$, $\overline{w}^{J \cap [1,m]} = \overline{w}^J$ and $\big| \log |w_{[1,m]}|\big| = \big|\log|w|\big|$, etc. We decompose the domain of integration into several subsets.
For $K\subset [1,d]\smallsetminus \{i\}$, we write  $K^c=[1,d]\smallsetminus K$ and $K'=[1,d]\smallsetminus
(\{i\}\cup K)$, and put
\[ U_K = D_K \times D_K^c \times \Delta_i := \prod_{j \in K} D_{z_j} \times \prod_{j \in K'} D_{z_j}^c \times \Delta_i. \]
That is, 
\[
U_K=\{w\in (\Delta^{\ast})^d\mid |z_i-w_i|\le |z_i|/3\text{ for }i\in K,\ |z_i-w_i|> |z_i|/3\text{ for }i\in K'\},
\]
so that we have $(\Delta^*)^d = \bigsqcup_{K \subseteq [1, d] \smallsetminus \{i\}} U_K$.
Using Lemma \ref{lemm:3}(\ref{item:1}), we obtain
  \begin{multline*}
      \int_{U_{K}}\frac{|z_{i}-w_{i}|
      \,\big|\log|w|\big|^{-N}}{\|z-w\|^{2d}|w^{[1,d]}||w^{J}|}|dw\wedge d\overline
    w|\leq \\
    C     \frac{\big|\log|z_{K}|\big|^{-N}}{|z^{K}|\,|z^{K\cap J}|}
    \int_{D^{c}_{K'}\times \Delta _i}\frac{|z_{i}-w_{i}|\big|\log|w_{K^{c}}|\big|^{-N}}{|w^{K^{c}}||w^{K^{c} \cap J}|}\\
    \left(\int_{D_{K}}\frac{|dw_{K}\wedge
    d\overline w_{K}|}{(\|z_{K}-w_{K}\|^{2}+\|z_{K^c} - w_{K^c}\|^{2})^{d}}\right)    |dw_{K^{c}}\wedge
    d\overline w_{K^{c}}|, 
  \end{multline*}
  for some constant $C>0$. In the previous formula, recall that, according to notation \ref{notationmulti}, if $K=\{k_1,\dots,k_\ell\}$,
  \[
  z^K=\prod_{i\in K} z_i,\quad\text{while}\quad  z_K = (z_{k_1},\dots,z_{k_\ell}),
  \]
  Noticing that $|K|<d$ we can bound the
  inner integral using Equation \eqref{eq:9} of Lemma \ref{lemm:2}, so 
    \begin{multline*}
    \int_{U_{K}}\frac{|z_{i}-w_{i}|
      \,\big|\log|w_{[1,d]}|\big|^{-N}}{\|z-w\|^{2d}|w^{[1,d]}||w^{J}|}|dw\wedge d\overline
    w|\le\\
    C     \frac{\big|\log|z_{K}|\big|^{-N}}{|z^{K}|\,|z^{K\cap J}|}
    \int_{D^{c}_{K'}\times \Delta
      _i}\frac{|z_{i}-w_{i}|\big|\log|w_{K^{c}}|\big|^{-N}}
    {\|z_{K^c}-w_{K^c}\|^{2(d-|K|)}|w^{K^{c}}||w^{K^{c}
        \cap J}|}  |dw_{K^{c}}\wedge
    d\overline w_{K^{c}}|, 
    \end{multline*}
  for some new constant $C>0$. Using the AM-GM inequality, the remaining integral can be bounded, up to a
  constant, by a product of the following integrals
  \begin{align*}
    &\int_{D^{c}_{j}}\frac{\big|\log|w_j|\big|^{-N}}
    {|z_{j}-w_{j}|^{2}|w_{j}|^2}  |dw_{j}\wedge
    d\overline w_{j}|, \text{ for }j\in K'\cap J,\\ 
    &\int_{D^{c}_{j}}\frac{\big|\log|w_j|\big|^{-N}}
    {|z_{j}-w_{j}|^{2}|w_{j}|}  |dw_{j}\wedge
    d\overline w_{j}|, \text{ for }j\in K'\smallsetminus J,\\ 
    &\int_{\Delta
      _i}\frac{\big|\log|w_{i}|\big|^{-N}}
    {|z_{i}-w_{i}| |w_{i}|^2}  |dw_{i}\wedge
    d\overline w_{i}|.
  \end{align*}
  
By the case $d=1$, the third integral can be
bounded as
\begin{displaymath}
  \int_{\Delta
      _i}\frac{|z_{i}-w_{i}|\big|\log|w_{i}|\big|^{-N}}
    {|z_{i}-w_{i}|^{2}|w_{i}|^2}  |dw_{i}\wedge
    d\overline w_{i}|\le C \frac{\big|\log |z_i|\big|^{-N+2}}{|z_i|}.
\end{displaymath}
By Lemma \ref{lemm:3}(\ref{item:1}), we have
\[ \int_{D^{c}_{j}}\frac{\big|\log|w_j|\big|^{-N}}
    {|z_{j}-w_{j}|^{2}|w_{j}|}  |dw_{j}\wedge
    d\overline w_{j}| \leq 4 \int_{D^{c}_{j}}\frac{\big|\log|w_j|\big|^{-N}}
    {|z_{j}-w_{j}||w_{j}|^2}  |dw_{j}\wedge
    d\overline w_{j}| \leq 4C \frac{\big|\log |z_j|\big|^{-N+2}}{|z_j|}, \]
where the last inequality follows again from the case $d=1$, providing the bound for the integrals for $j \in K' \smallsetminus J$. Finally, by Lemma \ref{lemm:3}(\ref{item:2}), for $j \in K'\cap J$ we have
\begin{align*}
  \int_{D^{c}_{j}}
  \frac{\big|\log|w_{j}|\big|^{-N}}{|z_{j}-w_{j}|^{2} |w_{j}|^2}|dw_{j}\wedge
    d\overline w_{j}| 
&\le C_1\frac{\big|\log|z_{j}|\big|^{-N+2}}{|z_{j}|^{2}}
    \int_{D^{c}_{j}}
  \frac{1}{ |w_{j}|^2 \big|\log|w_j|\big|^2}|dw_{j}\wedge
                           d\overline w_{j}|\\
  &\le C_2\frac{\big|\log|z_{j}|\big|^{-N+2}}{|z_{j}|^{2}},
\end{align*}
for some constants $C_1$ and $C_2$. Putting all the bounds together we deduce
\begin{displaymath}
  |g_i(z)| \leq  C_{f,N}C_{d,N} |z^{[1,d]}| \frac{|\log
    |z_{[1,d]}||^{-N+2}}{|z^{[1,d]}||z^{J \smallsetminus \{i\}}|}=
  C_{f,N}C_{d,N} \frac{\big|\log |z_{[1,d]}|\big|^{-N+2}}{|z^{J \smallsetminus \{i\}}|}.
\end{displaymath}
This completes the proof.
\end{proof}

We are finally left to estimate the growth of the successive derivatives of $K'\eta$. Recall that, for any $N \in \N$ and any $\alpha, \beta \in \N^d$, setting $C_{f, N, \alpha, \beta} := p_{N, \alpha, \beta}(f)$, one has $|\partial^{\alpha', \beta'} f(z)| \leq C_{f, N, \alpha, \beta} \big|z_{[1,m]}^{-(\alpha' + \beta')} \big| \big| \log|z_{[1,m]}| \big|^{-N}$ for any $\alpha' \leq \alpha$ and $\beta' \leq \beta$. 

\begin{lemma} \label{LemmaDerivatives} Let
  $f \in A_{rd, c}^{0}$ be a
  rapidly decreasing function on $(\Delta^*)^m \times \Delta^n$ with relatively compact support. Let $J \subseteq [1, d]$, $i \in J$ and let $g_i(z)$ be as in Lemma \ref{LemmaBMestimation}. Then, for
  any $\alpha, \beta \in \N^d$ and $ N\ge 2$,
  there exists a constant $C_{d, N}^{\alpha, \beta}$, depending only on $d$, $N$, $\alpha$ and $\beta$, such that for any $z \in (\Delta^*)^m \times \Delta^n$ we have
\[ \big| z_{[1,m]}^{\alpha + \beta} \big| \big|(\partial^{\alpha, \beta} g_i)(z) \big| \leq C_{d, N}^{\alpha, \beta} C_{f, N, \alpha, \beta} \frac{\big|
      \log|z_{[1,m]}|\big|^{-N+2}}{|z^{J \cap [1,m] \smallsetminus \{i\}} | }. \]
\end{lemma}

\begin{proof}
As in the proof of Lemma \ref{LemmaBMestimation}, for simplicity in the notation, we assume $m = d$ (for the other variables the calculation is simpler, as $f$ is smooth). Recall that 
\[ g_i(z) =  z^{[1,d]} \int_{\Delta^d} \frac{(\overline{z}_i - \overline{w}_i) f(w)}{\|z - w \|^{2d} w^{[1,d]} \overline{w}^J } dw \wedge d \overline{w}. \]
Denote $\mu_i(z, w) = \frac{(\overline{z}_i - \overline{w}_i)}{\| z - w \|^{2d}}$. Since we will be deriving with respect to different variables, during this proof we note $\partial_{z, \overline{z}}^{\alpha, \beta} = \partial^{\alpha, \beta}$ so that $\partial_{w, \overline{w}}^{\alpha, \beta}$ will denote the same differential operator but with respect to the variable $w$. Applying Leibniz rule, we can write
\begin{eqnarray*}
(\partial_{z,\overline{z}}^{\alpha, \beta} g_i)(z) &=& \partial_{z,\overline{z}}^{\alpha, \beta} \left( z^{[1,d]} \int_{\Delta^d} \mu_i(z, w) \frac{f(w)}{w^{[1,d]} \overline{w}^J} dw \wedge d\overline{w} \right) \\
&=& \sum_{\alpha_1 + \alpha_2 = \alpha} \binom{\alpha} {\alpha_1} \partial_{z,\overline{z}}^{\alpha_1, 0}(z^{[1,d]}) \partial_{z,\overline{z}}^{\alpha_2, \beta} \left( \int_{\Delta^d} \mu_i(z, w) \frac{f(w)}{w^{[1,d]} \overline{w}^J} dw \wedge d\overline{w} \right).
\end{eqnarray*}
Observing that $\partial_{z,\overline{z}}^{\alpha_1, 0}(z^{[1,d]}) = 0$ when $\alpha_1$ has at least one coordinate that is $>1$, and that if every coordinate of $\alpha_1$ is $\leq 1$ then $\partial_{z,\overline{z}}^{\alpha_1, 0}(z^{[1,d]}) = z^{\pmb{1} - \alpha_1}$, it suffices to show that, for all $\alpha, \beta \in \N^d$,
\begin{equation} \label{EquationGrowthLemma1}
\big| z^{\alpha+\beta + \pmb{1}} \big| \left| \partial_{z,\overline{z}}^{\alpha, \beta} \left( \int_{\Delta^d} \mu_i(z, w) 
\frac{f(w)}{w^{[1, d]} \overline{w}^J} dw \wedge d\overline{w} \right) \right| 
\end{equation}
satisfies the bound of the statement, where, as in notation \ref{notationmulti} $\alpha + \beta + \pmb{1} = (\alpha_1 + \beta_1 + 1 , \ldots, \alpha_d + \beta_d + 1)$. We make a
decomposition of the domain of integration $\Delta ^d$ slightly
different from the one used in Lemma \ref{LemmaBMestimation}. 
For $K\subset [1,d]$, we write  $K^c=[1,d]\smallsetminus K$, and put
  \begin{displaymath}
    U'_{K}=\{w =(w_{1},\dots,w_{d})\mid |z_{j}-w_{j}|\le
    |z_{j}|/3 \text{ if }j\in K,\  |z_{j}-w_{j}|\ge
    |z_{j}|/3 \text{ if }j\in K^c\},
  \end{displaymath}
i.e., $U'_{K} = \prod_{j \in K} D_{z_j} \times \prod_{j \in K^c} D_{z_j}^c$.
Then the expression of
Equation \eqref{EquationGrowthLemma1} will be written as
\begin{equation} \label{EquationGrowthLemma2}
  |z^{\alpha +\beta + \pmb{1}}| \sum_{K \subseteq [1, d]} \partial_{z, \overline{z}}^{\alpha, \beta} \left( \int_{U'_K} \mu_i(z, w) \frac{f(w)}{w^{[1, d]} \overline{w}^J} dw \wedge d\overline{w} \right).
\end{equation}
We claim that, for each $K \subseteq [1, d]$, we have
\begin{multline}
  \partial_{z, \overline{z}}^{\alpha, \beta}
  \left( \int_{U'_K} \mu_i(z, w) \frac{f(w)}
    {w^{[1, d]} \overline{w}^J} dw \wedge d\overline{w} \right) \\
  = \int_{U'_K} \partial_{z, \overline{z}}^{\alpha_{K^c},
    \beta_{K^c}}
  \left( \mu_i(z, w) \right)
  \partial_{w, \overline{w}}^{\alpha_K, \beta_K}
  \left( \frac{f(w)}{w^{[1, d]} \overline{w}^J} \right)
  dw \wedge d\overline{w}. \label{eq:14}
\end{multline}
Recall that, following notation \ref{notationmulti}, $\alpha_K, \beta_K \in \N^d$ have $j$-th coordinates equal to
$\alpha_j$ for $j \in K$ and $0$ for $j \notin K$, and the opposite for
$\alpha_{K^c}, \beta_{K^c}$; Equation \eqref{eq:14} follows basically from
the composition
\begin{displaymath}
\partial_{z, \overline{z}}^{\alpha, \beta}=\partial_{z,
  \overline{z}}^{\alpha_K, \beta_K}\circ
\partial_{z, \overline{z}}^{\alpha_{K^c}, \beta_{K^c}}.
\end{displaymath}
Indeed, in $U'_K$, for $j\in K^c$, the function $\mu _i(z,w-z)$ is
smooth with respect to  $z_j$ and $\overline{z}_j$, so that
\begin{displaymath}
  \partial_{z, \overline{z}}^{\alpha_{K^c}, \beta_{K^c}}
  \left( \int_{U'_K} \mu_i(z, w) \frac{f(w)}
    {w^{[1, d]} \overline{w}^J} dw \wedge d\overline{w} \right)
  = \int_{U'_K} \partial_{z, \overline{z}}^{\alpha_{K^c},
    \beta_{K^c}}
  \left( \mu_i(z, w) \right)
  \frac{f(w)}{w^{[1, d]} \overline{w}^J}
  dw \wedge d\overline{w},
\end{displaymath}
while for $j\in K$, 
the function $\frac{f(w)}{w^{[1, d]} \overline{w}^J}$ and
its derivatives are smooth with respect to the variables $w_j$ and
$\overline{w}_j$. Therefore, we have
\begin{align*}
  \partial_{z, \overline{z}}^{\alpha_{K}, \beta_{K}} \Bigg( \int_{U'_K} \mu_i(z, w) \frac{f(w)}{w^{[1,d]}
  \overline{w}^J} & dw \wedge d\overline{w} \Bigg)\\
  &= \partial_{z, \overline{z}}^{\alpha_{K}, \beta_{K}} \left( \int  \mu_i(z, z-u) \frac{f(z-u)}
    {(z -u)^{[1,d]}  \overline{(z - u)}^J}
    du \wedge d\overline{u} \right)\\
  &= \int \mu_i(z, z-u)  \partial_{z, \overline{z}}^{\alpha_{K}, \beta_{K}} \left( \frac{f(z-u)}
    {(z - u)^{[1,d]}  \overline{(z - u)}^J} \right)
    du \wedge d\overline{u} \\
  &= \int_{U'_K} \mu_i(z, w)  \partial_{w, \overline{w}}^{\alpha_{K}, \beta_{K}} \left( \frac{f(w)}
    {w^{[1,d]}  \overline{w}^J} \right) dw \wedge d\overline{w}, 
\end{align*}
where in the first equality we have made the change of variables $w
\mapsto u:=z - w$. In the second one we used that $\mu_i(z, z-u) =
\frac{\overline{u}_i}{\|u\|^{2d}}$ is integrable and independent of
$z$ and $\overline{z}$ and $\frac{f(z-u)}{(z - u)^{[1,d]}
  \overline{(z - u)}^J}$ is smooth on $\{ z \mid z-u \in U_K \}$ with
respect to the variables in $K$. This proves the claim.

We calculate the derivative $\partial_{w,
  \overline{w}}^{\alpha_{K}, \beta_{K}} \left( \frac{f(w)}{w^{[1, d]} \overline{w}^J} \right)$. Using Leibniz rule,
one sees that this derivative is a linear combination of terms of the
form
\[
  \frac{\partial_{w, \overline{w}}^{\ell, k} f(w)}
  {w^{\alpha_{K}+\pmb{1} - \ell} \overline{w}
    ^{\beta_{K \cap J}+\pmb{1}_J - k_{K \cap J}}}
\]
for $\ell, k \in \N^d$ with $\ell \leq \alpha_{K}$ and
$k \leq \beta_{K}$ and with
$k_{K \smallsetminus J} = \beta_{K\smallsetminus J}$. By the definition of the norm $p_{N,\alpha,\beta}$, we have the bound
\begin{equation} \label{EquationEstimateTmp}
  \left| \frac{\partial_{w, \overline{w}}^{\ell, k} f(w)}
    {w^{\alpha_{K} + \pmb{1} - \ell}
      \overline{w}^{\beta_{K\cap J} + \pmb{1}_J - k_{K \cap J}}} \right|
  \leq C_{f, N, \alpha, \beta} \frac{\big| \log|w_{[1,d]}|\big|^{-N}}{|w|^{\alpha_{K} + \pmb{1}}
    |\overline{w}|^{\beta_{K \cap J} + \pmb{1}_J} |w|^{k_{K
        \smallsetminus J}}}.
\end{equation}
We recall that $w^{\alpha_{K} + \pmb{1}}$ denotes $\prod_{j} w_j^{\alpha_{K}(j) + 1}$.
So, observing that $\beta_{K \cap J} + k_{K \smallsetminus J} \leq
\beta_{K}$, we obtain the bound
\begin{multline*}
 |z^{\alpha + \beta + \pmb{1}}|  \left| \partial_{z,
    \overline{z}}^{\alpha, \beta} \left( \int_{U'_K} \mu_i(z, w)
    \frac{f(w)}{w^{[1, d]} \overline{w}^J} dw \wedge
    d\overline{w} \right) \right|\\
  \le C_{f, N, \alpha, \beta} 
  |z^{\pmb{1} + \alpha_{K^c} + \beta_{K^c}}| \int_{U'_K} \left| \partial_{z, \overline{z}}^{\alpha_{K^c}, \beta_{K^c}}\left( \mu_i(z, w) \right) \right| \left( \frac{\big|\log|w_{[1,d]}|\big|^{-N}}{|w^{[1,d]}| |\overline{w}^J|} \right)  \left| dw \wedge d\overline{w} \right|.
\end{multline*}
We are left to calculate the factors $\partial_{z,
  \overline{z}}^{\alpha_{K^c}, \beta_{K^c}}\left( \mu_i(z, w) \right)$. Using
Leibniz rule again, one sees that $\partial_{z,
  \overline{z}}^{\alpha_{K^c}, \beta_{K^c}}\left( \mu_i(z, w) \right)$ can be
written as an integral linear combination of terms of the form 
\[ \frac{ (z - w)^\ell (\overline{z} - \overline{w})
    ^{\alpha_{K^c} - \beta_{K^c} + e_i + \ell}}
  {\| z - w \|^{2(d + |\alpha_{K^c}| + |\ell|)}} \]
where $\ell \in \N^d$ is such that $\beta_{K^c} - (\alpha_{K^c}+ e_i)
\leq \ell \leq \beta_{K^c}$, and where $e_i$ is the basis vector with
$i$th coordinate $1$ and zero elsewhere. So it suffices to bound
integrals of the form 
\begin{equation*}
|z^{\pmb{1} + \alpha_{K^c} + \beta_{K^c}}|
\int_{U'_K}  \frac{ |z - w|^{\alpha_{K^c} - \beta_{K^c} + e_i +
    2 \ell}}{\| z - w \|^{2(d + |\alpha_{K^c}| + |\ell|)}} \left(
  \frac{|\log\big|w_{[1,d]}|\big|^{-N}}{|w^{[1,d]}| |\overline{w}^J|} \right)  \left|
  dw \wedge d\overline{w} \right| .
\end{equation*}
But for $w \in U'_K$ and $j \in K^c$ we have that $|z_j - w_j| \geq |z_j|/3$ and also $|z_j - w_j| \leq 3|z_j|$ so this integral is bounded, up to a constant, by
\begin{equation} \label{EquationGrowthLemma3}
|z^{\pmb{1} + 2 \alpha_{K^c}}| \int_{U'_K}  \frac{ |z - w|^{2 \ell} |\overline{z}_i - \overline{w}_i|}{\| z - w \|^{2(d + |\alpha_{K^c}| + |\ell|)}} \left( \frac{\big|\log|w_{[1,d]}|\big|^{-N}}{|w^{[1,d]}| |\overline{w}^J|} \right)  \left| dw \wedge d\overline{w} \right|.
\end{equation}
By the weighted AM-GM inequality we have
\begin{equation*}
\| z - w \|^{2 |\ell|} \geq C_1 |z - w|^{2 \ell}, \quad \| z - w \|^{2 |\alpha_{K^c}|} \geq C_2 |z - w|^{2 \alpha_{K^c}},
\end{equation*}
for some constants $C_1$ and $C_2$ (depending only on $\ell$ and $\alpha_{K^c}$, respectively), which shows that, up to a constant, the expression of Equation \eqref{EquationGrowthLemma3} is bounded above by
\begin{equation} \label{EquationGrowthLemma4}
|z^{[1,d]}| \int_{U'_K}  \frac{|\overline{z}_i - \overline{w}_i|}{\| z - w \|^{2d}} \frac{\big|\log|w_{[1,d]}|\big|^{-N}}{|w^{[1,d]}| |\overline{w}^J|}  \left| dw \wedge d\overline{w} \right|.
\end{equation}
Then the result follows from Lemma \ref{LemmaBMestimation}.
\end{proof}

We can finally state and prove one of our key results.

\begin{theorem}
 \label{LemmaBM}
The Bochner-Martinelli operator of Definition \ref{DefinitionBM} induces well defined continuous operators
\[ K' : A_{rd, c}^{p, q} \to A_{rd}^{p, q-1} \]
for any $p \geq 0$ and $q \geq 0$. Moreover, these operators satisfy $\overline{\partial} K'  + K' \overline{\partial} = id$. 
\end{theorem}

\begin{proof}

To see that the operator is well defined, we need to prove that if $\eta \in A_{rd, c}^{p, q}$ is a rapidly decreasing form on $(\Delta^*)^m \times \Delta^n$ with relatively compact support, then the integral defining $K' \eta$ converges and gives a rapidly decreasing differential form. We have $K'\eta=\sum_{|I|=p} K'\eta_I \wedge \xi_I$ where $\eta=\sum_{|I|=p} \eta_I \wedge \xi_I$ and $\eta_I = \sum_{|J| = q} f_{I, J} \overline{\xi}_J$. According to equations \eqref{expressionK'} and \eqref{EquationDefgIJ}, we have
$$
(K' \eta_I)(z) = \frac{(d - 1)!}{(2 \pi i)^d} \sum_{i=1}^d\sum_{|J|=q} g_{i,I,J}(z) \wedge d \overline{z}_{J-\{i\}}
$$
where 
$$
g_{i,I,J}(z) =  z^{[1,m]} \int_{\Delta^d} \frac{(\overline{z}_i - \overline{w}_i) f_{I,J}(w)}{\|z - w \|^{2d} w^{[1,m]} \overline{w}^{J \cap [1,m]} } dw \wedge d \overline{w}
$$
and $f_{I,J}$ is a rapidly decreasing function on $(\Delta^*)^m \times \Delta^n$ with relatively compact support. Hence, by Lemma \ref{LemmaBMconvergence}, Lemma \ref{LemmaBMestimation} and Lemma \ref{LemmaDerivatives} $K' \eta$ is well defined and rapidly decreasing, showing the first claim.
Moreover, Lemma \ref{LemmaBMestimation} and Lemma \ref{LemmaDerivatives} also show that the operator $K'$ is bounded and hence continuous for every seminorm $p_{N, \alpha, \beta}$ and thus is continuous as an operator from the space of compactly supported rapidly decreasing differential forms equipped with the Schwartz topology of Definition \ref{DefTopologyrd} to the space of rapidly decreasing differential forms. As the operator $\overline{\partial}$ is continuous by definition, one deduces that the operator $\overline{\partial} K' + K' \overline{\partial}$ is also continuous.
It follows from the classical Bochner-Martinelli-Koppelman formula (\cite[\S III Theorem 1.7]{Laurent}) that the identity $\overline{\partial} K' + K'\overline{\partial} = {\rm id} $ holds on the dense subspace of $A_{rd, c}^{0, q}$ of differential forms  on $(\Delta^*)^m \times \Delta^n$ with relatively compact support and hence the result follows by continuity of the operator.
\end{proof}

\subsection{Sheaves of differential forms with growth conditions} \label{SubsectionGlobal1}

Let $\overline{X}$ be a smooth compact complex manifold of dimension $d$, $D \subseteq \overline{X}$ a normal crossing divisor and $X = \overline{X} - D$. This means that each point $x \in \overline{X}$ has an open neighbourhood $U$ isomorphic to $\Delta^d$ with coordinates $(z_1, \hdots, z_d)$ for which $x = (0, \hdots, 0)$ and such that there exist some integers $m, n$, with $0 \leq n, m \leq d$ and $m + n = d$, such that
\[ X \cap U = (\Delta^*)^m \times (\Delta)^{n} = \{ (z_1, \hdots, z_d) \in \Delta^d \; | \; z_1 \hdots z_m \neq 0 \}. \] 
We will call such a subset $U$ an open coordinate neighbourhood (of $x$).

We denote by $\mathcal{O}_{\overline{X}}$ the structural sheaf of holomorphic functions on $\overline{X}$ and we denote by $\Omega_{\overline{X}}^*$ the holomorphic de Rham complex on $\overline{X}$. This is a complex of locally free $\mathcal{O}_{\overline{X}}$-module of finite type. Recall also that, if $j : X \to \overline{X}$ denotes the natural inclusion, then $\Omega^*_{\overline{X}}(\log \, D)$ is defined to be the sub-$\mathcal{O}_{\overline{X}}$-algebra of $j_* \Omega_{X}^*$ locally generated by the sections $\xi_i$, $1 \leq i \leq d$, where we recall the reader from Notation \ref{notationmulti} that $\xi_i = \frac{dz_i}{z_i}$ for  $1 \leq i \leq m$, and $\xi_i = d z_i$,  $m < i \leq d$. We shall denote by $\mathscr{A}^0_{\overline{X}}$ the sheaf of smooth functions on $\overline{X}$ and $\mathscr{A}_{\overline{X}}^*$ the complex of sheaves of smooth differential forms. The complex $\Alog$ of smooth differential forms on $X$ with logarithmic growth at $D$ is defined (cf. \cite[\S 2]{Burgos}) to be the $\mathscr{A}_{\overline{X}}$-algebra subsheaf of $j_* \mathscr{A}^*_{X}$ locally generated by the sections
\[ \log | z_i | \;\;\; (1 \leq i \leq m), \]
\[  \xi_i, \; \overline{\xi}_i \;\;\;\; (1 \leq i \leq d). \]

\begin{definition} \label{Deflogdifferentialforms}
We denote by $\mathscr{A}_{si}^0$ (resp. $\mathscr{A}_{rd}^0$) the sheaf on $\overline{X}$ whose sections on $U \subseteq \overline{X}$ are given by complex valued functions on $U \cap X$ which are locally at each point of $U$ slowly increasing (resp. rapidly decreasing). We define the graded sheaf $\Asi$ of slowly increasing differential forms to be the $\mathscr{A}^0_{si}$-subalgebra of $j_* \mathscr{A}^*_{X}$ locally generated by $\xi_i$, $\overline{\xi}_i $  for $1 \leq i \leq d$.
Similarly, the graded sheaf $\Ard$ of rapidly decreasing differential forms is defined as the $\mathscr{A}^0_{rd}$-algebra of $j_* \mathscr{A}^*_{X}$ locally generated by $\xi_i$, $\overline{\xi}_i $  for $1 \leq i \leq d$.
\end{definition}

\begin{remark} \leavevmode
\begin{enumerate}
\item Observe that, even if it is not reflected in the notation, the sheaves $\mathscr{A}^*_{si}$ and $\mathscr{A}^*_{rd}$ depend on the divisor $D$.  

\item More precisely, for any open subset $U \subseteq \overline{X}$, a differential form $\omega \in j_* \mathscr{A}_X^*(U)$ lies in $\Asi(U)$ (resp. $\Ard(U)$) if it can locally be written as
\[ \omega = \sum_{I, J} \alpha_{I, J} \xi_I  \wedge \overline{\xi}_{J}, \]
where $I,J \subseteq [1,d]$ and $\alpha_{I, J}$ is a function in $\mathscr{A}_{ si}^0(U) \cong A^0_{si}$ (resp. $\mathscr{A}_{rd}^0(U) \cong A^0_{rd}$).
\item By the definition of slowly increasing and rapidly decreasing functions, if a form $\omega$ belongs to $\Asi$ (resp. $\Ard$) then $\partial \omega$ and $\overline{\partial} \omega$ also belong to $\Asi$ (resp. $\Ard$). Therefore $\Asi$ and $\Ard$ are differential graded algebras.
\end{enumerate}
\end{remark}

There are natural inclusions
\[ \Ard \subseteq \Alog \subseteq \Asi \subseteq j_* \mathscr{A}_{X}^*. \] Moreover, these are all complexes of fine sheaves since their terms are modules over $\mathscr{A}_{\overline{X}}^0$ and hence admit partitions of unity. We denote by $\mathscr{A}_{rd}^*(\overline{X})$ and $\mathscr{A}_{si}^*(\overline{X})$ the corresponding complexes of global sections. Moreover, the complex structure on $\overline{X}$ induces compatible bigradings
\[ \mathscr{A}_{rd}^n = \bigoplus_{p + q = n} \mathscr{A}_{rd}^{p, q},
  \;\;\;  \mathscr{A}^n_{\overline{X}}(\log \, D) = \bigoplus_{p + q =
    n} \mathscr{A}^{p, q}_{\overline{X}}(\log \, D), \;\;\;
  \mathscr{A}_{si}^n = \bigoplus_{p + q = n} \mathscr{A}_{si}^{p,
    q}, \] with corresponding Hodge filtrations $F^p
\mathscr{A}_{rd}^n = \bigoplus_{p' \geq p} \mathscr{A}^{p', q}_{rd}$,
etc. Finally, we denote by $\mathscr{A}_{si, \R}^* \subseteq \Asi$,
$\mathscr{A}_{rd, \R}^* \subseteq \Ard$, $\mathscr{A}^*_{\overline{X},
  \R}(\log \, D) \subseteq \mathscr{A}^*_{\overline{X}}(\log \, D)$
the subcomplexes of sheaves of $\R$-valued differential forms.

\begin{example}\label{exm:1}
  Using the Taylor expansion of smooth functions, one can check that a
  differential form $\omega \in \mathscr{A}^*_{\overline{X}}$ 
  belongs to $\mathscr{A}^*_{rd}$ if and only if the restriction
  $\omega |_{D_{sm}}$ is zero. Here $D_{sm}$ is the submanifold of
  smooth points of $D$. 
\end{example}

For the record, we have the following comparison results.

\begin{proposition}[{\cite[Theorem 2.1, Corollary 2.2]{Burgos}}, {\cite[Proposition 5.8]{HarrisZuckerIII}}, {\cite[\S 3.7]{KatoMatsubaraNakayama}}] \label{PropBurgos}
The natural maps $(\Omega^*_{\overline{X}}(\log \, D), F) \to (\mathscr{A}^*_{\overline{X}}(\log \, D), F) \to (\mathscr{A}^*_{si}, F)$ are filtered quasi-isomorphisms and $R j_* \R \to \mathscr{A}^*_{\overline{X}, \R}(\log \, D) \to \mathscr{A}^*_{si, \R}$ is a quasi-isomorphism.
\end{proposition}

\subsection{Compactly supported cohomology}

As a first application of the Bochner-Martinelli operator, we show that the complexes or rapidly decreasing smooth differential forms calculate the cohomology of $X$ with compact support. We point out that this result was stated in \cite{Harrisdbarcohomology} and \cite[Proposition 2.2.4]{HarrisZuckerIII} but the proof in these references does not seem to hold as the usual Bochner-Martinelli operator does not preserve rapidly decreasing differential forms. Indeed, if $g(z)$ is any smooth compactly supported differential form on $\Delta$ such that $g(0) \neq 0$, then $\eta:= \overline{\partial} g(z) = \frac{\partial}{\partial \overline{z}} g(z) d \overline{z} = \frac{\partial}{\partial \overline{z}} g(z) \overline{z} \frac{d \overline{z}}{\overline{z}}$ is a rapidly decreasing differential form of type $(0,1)$ but $(K\eta)(z) = g(z)$ is clearly not rapidly decreasing because $g(0)\not = 0$.

\begin{proposition}\label{PropOnCSC}
For any $p \geq 0$, the complex $\mathscr{A}^{p, *}_{rd}$ is a fine resolution of the sheaf $\Omega^p_{\overline{X}}(\log \, D)(-D)$, i.e. there is a long exact sequence
\[ 0 \to \Omega^p_{\overline{X}}(\log \, D)(- D) \to \mathscr{A}_{rd}^{p, 0} \xrightarrow{\overline{\partial}} \mathscr{A}_{rd}^{p, 1} \xrightarrow{\overline{\partial}} \ldots \xrightarrow{\overline{\partial}} \mathscr{A}_{rd}^{p, d} \to 0. \]
\end{proposition}

\begin{proof}
It suffices to prove exactness at the level of stalks so we can
restrict to an open coordinate neighborhood $V$ of a point $x \in
\overline{X}$  isomorphic to $\Delta ^{d}$ with  
$U=X\cap V = (\Delta^*)^m \times (\Delta ) ^{n}$ with $m+n = d$. We first prove exactness at
the left. Let $\varphi = \sum_{|I| = p} f_{I}(z)
\xi_I \in \mathscr{A}_{rd}^{p, 0}$ be such that
$\overline{\partial} \varphi = 0$. Then each $f_{I}(z)$ satisfies
$\overline{\partial} f_{I}(z) = 0$ and is hence a holomorphic
function. In particular, $f_{I}(z)$ is holomorphic on $V$ and
vanishes at $D$ (since it is rapidly decreasing), hence it is divisible by
$z_{[1,m]}$. Writing $f_{I}(z) = z_{[1,m]} g_{I}(z)$ for some holomorphic
function $g_{I}(z)$ on $V$, this shows that $\varphi = \sum_{|I|
  = p} z_{[1,m]} g_{I}(z) \xi_I$ with
$g_{I}$ holomorphic, i.e. $\varphi \in 
\Omega^p(\log \, D)(-D)$ as wished. 

We now show that the exact sequence is exact in the middle.
Let $\varphi \in \mathscr{A}_{rd}^{p, q}(V)$ with $q \geq 1$ be such
that $\overline{\partial} \varphi = 0$. Let $\eta$ be a bump function
such that $\eta \equiv 1$ in a strictly smaller open polydisc $V' \subseteq V$ and vanishing outside a strictly smaller polydisc $V'' \subseteq V$. Since $\eta \varphi \in \mathscr{A}_{rd, c}^{p, q}(V)$, By Theorem \ref{LemmaBM} we have $K' \overline{\partial}(\eta \varphi) + \overline{\partial} K' (\eta \varphi) = \eta \varphi$ on $V$. Since $\overline{\partial}(\eta \varphi) = \overline{\partial}\eta \wedge \varphi$, we deduce that
\[ K' (\overline{\partial} \eta \wedge \varphi) + \overline{\partial} K'(\eta \varphi) = \eta \varphi. \]
The differential form $\overline{\partial} \eta \wedge \varphi$ is
compactly supported (since $\overline{\partial} \eta$ is compactly
supported), smooth and closed. Hence $\kappa :=
K'(\overline{\partial} \eta \wedge \varphi)$ is a smooth differential
form on $V$ which is moreover rapidly decreasing by Theorem
\ref{LemmaBM}.

Since $\eta \equiv 1$ on $V'$, restricting to $V'$ we obtain
\[ \kappa + \overline{\partial} K'(\eta \varphi) = \varphi. \] In
particular, the smooth differential form $\kappa$ is closed (on $V'$)
since it is cohomologous to $\varphi$ which is closed. Hence by the
usual Poincar\'e lemma for smooth differential forms, we have $\kappa
= \overline{\partial} \alpha$ for some smooth differential form
$\alpha$ on $V'$. We would be done if we knew that $\alpha$ is rapidly
decreasing. We will now construct a rapidly decreasing  differential form
$\alpha'$ such that $\overline{\partial} \alpha' = \overline{\partial}
\alpha = \kappa$.

Let $D\cap V=\bigcup_{i=1}^k D_i$ be the decomposition of $D\cap V$ into
smooth components. Each $D_i$ is a coordinate hyperplane for the
coordinates of $V$ and there are holomorphic projections $\pi _i\colon
V\to D_i$. By example \ref{exm:1}, for each $i$, $\kappa
|_{D_i}=0$. Therefore $\alpha |_{D_i}$ is $\overline{ \partial}$-closed. We
write inductively
\begin{displaymath}
  \alpha _0=\alpha , \qquad \alpha _i=\alpha _{i-1}-\pi
  _{i}^\ast\alpha _{i-1}|_{D_i}. 
\end{displaymath}
These forms satisfy
\begin{displaymath}
  \overline{ \partial} \alpha_i= \overline{ \partial} \alpha_{i-1} 
\end{displaymath}
because $\alpha_{i-1}|_{D_i}$ is $\overline{ \partial}$-closed, and
\begin{displaymath}
  \alpha _i|_{D_j}=0,\quad \text{for }j\le i.
\end{displaymath}
In particular $\alpha _k|_{D_{sm}}=0$. Therefore by example
\ref{exm:1}, the form $\alpha _k$ is rapidly decreasing. Since
$\overline{ \partial}\alpha _k=\overline{ \partial}\alpha =\kappa $ we
deduce that $\kappa $ is exact in the complex $\mathscr{A}^*_{rd}$
completing the proof of the proposition.
\end{proof}

\begin{remark}
 Proposition \ref{PropOnCSC} can also be proved using the usual inductive argument to reduce the proof of the $\overline{\partial}$-Poincar\'e lemma to the case of dimension $1$ (cf. \cite[p. 25]{GriffithsHarris}). Observe that, for the induction step to be valid, it is crucial to impose growth conditions on all derivatives in the definition of rapidly decreasing differential form. In the case of dimension $1$ the above proof becomes much simpler. Indeed, let $\varphi \in \mathscr{A}^{0,1}_{rd}(\Delta)$ be a rapidly decreasing differential form with $\overline{\partial} \varphi = 0$. Let $\eta$ denote the bump function as in the proof of Proposition \ref{PropOnCSC}. By Theorem \ref{LemmaBM}, we have that $\overline{\partial} K' (\eta \varphi) + K' \overline{\partial} (\eta \varphi) = \eta \varphi$. But $\overline{\partial} (\eta \varphi) = \overline{\partial}\eta \wedge \varphi = 0$ for dimension reasons. This shows that, after restricting to a smaller disc, $\varphi$ is $\overline{\partial}$-exact, as wished.
\end{remark}

\subsection{Tempered currents and cohomology}

Let $\overline{X}, X$ and $D$ be as in \S \ref{SubsectionGlobal1}. In this section we define sheaves of tempered currents by considering continuous linear forms on compactly supported rapidly decreasing forms. Our main purpose is to show that the Hodge complex of tempered currents calculates cohomology of $X$ with its Hodge filtration and real structure.

\begin{definition}
For any $0 \leq p, q \leq d$, we define the sheaf $\Dist^{p,q}$ of tempered currents to be the sheaf on $\overline{X}$ assigning to any open coordinate neighbourhood $U \subseteq \overline{X}$ the complex vector space $\Dist^{p,q}(U)$ of continuous complex linear forms on the compactly supported sections $\Gamma_c(U, \mathscr{A}^{d-p, d-q}_{rd})$. Similarly we denote by $\mathscr{D}_{si, \R}^{p, q}$ the sheaf on $\overline{X}$ assigning to any open coordinate neighbourhood $U \subseteq \overline{X}$ the real vector space $\mathscr{D}_{si, \R}^{p,q}(U)$ of continuous real linear forms on the compactly supported sections $\Gamma_c(U, \mathscr{A}^{d-p, d-q}_{rd, \R})$.
\end{definition}

\begin{remark} \label{remtoprd} \leavevmode
\begin{enumerate}

\item The space $\Gamma_c(U, \mathscr{A}^0_{rd})$ of compactly supported rapidly decreasing functions on $\overline{X}$ is equipped with the usual Schwartz topology as in Definition \ref{DefTopologyrd}.

\item Observe that the boundary $D$ might intersect $U$ and hence an element in  $\Gamma_c(U, \mathscr{A}^0_{rd})$ is a smooth function on $U \cap X$ with growth conditions along $D \cap U$ and whose support is contained in a compact subset $K$ of $U$. For example, if $\overline{X} = \Delta$, $D = \{0\}$, $X = \Delta^*$, and $U \subseteq \Delta$ is any open ball around the origin, then an element in $\Gamma_c(U, \mathscr{A}_{rd}^{d - p, d - q})$ can still have singularity at $z = 0$, but it will vanish outside some closed ball centered at the origin contained in $U$.
\end{enumerate}
\end{remark}

For any open subset $U \subseteq \overline{X}$, let $T \in \Dist^{p, q}(U)$ and $\omega \in \mathscr{A}^{p', q'}_{si}(U)$. Since the product of a slowly increasing differential form against a rapidly decreasing differential form is rapidly decreasing, the formula
\[ (T\omega)(\eta) = T(\omega\wedge \eta) \]
is a well defined element in $\Dist^{p+p', q+q'}(U)$. This induces a $\mathscr{A}_{si}^0$-bilinear map
\[  \Dist^{p', q'} \otimes_{\mathscr{A}_{si}^0}\mathscr{A}_{si}^{p, q} \to \Dist^{p+p', q+q'}, \]
equipping $\Dist^{*, *}$ with a bigraded right $\mathscr{A}_{si}^{*,*}$-module structure.

\begin{lemma} \label{LemmatempDist} Let $0 \leq p, q \leq d$. The map
\[ \Dist^{p - p', q - q'} \otimes_{\mathscr{A}_{\overline{X}}^0(\log \, D)} \mathscr{A}_{\overline{X}}^{p', q'}(\log \, D) \to \Dist^{p, q} \]
is an isomorphism for any $p' \in \{0, p\}$ and any $q' \in \{0, q\}$. In particular
\[\mu\colon \Dist^{0, q}  \otimes_{\mathcal{O}_{\overline{X}}} \Omega_{\overline{X}}^p(\log \, D)\to \Dist^{p, q} \]
is an isomorphism.
\end{lemma}

\begin{proof}
We give the proof for $p' = p$ and $q' =0$, the other cases being similar. It suffices to prove the isomorphism at stalks. Let $x \in \overline{X}$ and let $U \subseteq \overline{X}$ be a sufficiently small coordinate open neighborhood of $x$ such that $U \cong \Delta^d$, $U \cap X \cong (\Delta^*)^m \times \Delta^{n}$ and $x$ is mapped to the origin. Let $\{\xi_I\}_{I}$, where $I\subseteq [1, d]$ with $|I|= p$, be a basis of $\Omega^p_{\overline{X}}(\log \, D)(U)$ over $\mathcal{O}_{\overline{X}}(U)$. If $T \in \Dist^{p, q}(U)$, then, for any $I$ as above, define $T_{I} \in \Dist^{0, q}(U)$ by 
\[ T_{I} \left( \eta(z) \wedge \xi_{[1,d]} \right) = T \left( \eta(z) \wedge \xi_{I^c} \right), \]
where $I^c =[1, d] \smallsetminus I$ and $\eta \in \Gamma_c(U, \mathscr{A}_{rd}^{0, d - q})$. Then we have
\[ T = \sum_{|I| = p} \pm \mu(T_{I} \otimes \xi_I ), \]
showing that the multiplication map is surjective. Moreover, one easily checks that this decomposition is unique, showing that the map is injective. This finishes the proof.
\end{proof}

By Lemma \ref{LemmatempDist}, for any coordinate open neighbourhood $U \subseteq \overline{X}$, we can write a tempered current $T \in \Dist^{p, q}(U)$ as a differential form with coefficients in tempered distributions: 
\[ T = \sum_{|I| = p, |J| = q} T_{I, J} \xi_I \wedge \overline{\xi}_J, \]
where $T_{I, J} \in \Dist^{0, 0}(U)$. The exterior derivative of rapidly decreasing forms induces a differential $d : \Dist^n \to \Dist^{n + 1}$ defined by
\[ d T (\omega) = (-1)^{n + 1} T(d \omega).\]
We have $d = \partial + \overline{\partial}$ with $\partial : \Dist^{p, q} \to \Dist^{p + 1, q}$, $\overline{\partial} : \Dist^{p, q} \to \Dist^{p, q+1}$ defined in the analogous way. This defines complexes of sheaves of tempered currents $\Dist^*$, $\Dist^{*, *}$ (with differentials $d, \partial, \overline{\partial}$). Since currents are modules over $\mathscr{A}_{si}^0$ all these complexes are complexes of fine sheaves. We will denote by $\Dist^{*,*}(\overline{X})$ the corresponding complexes of global sections. Moreover $\Dist^*$ is equipped with a Hodge filtration given by
\[ F^p \Dist^* = \bigoplus_{p' \geq p} \Dist^{p', q}. \]
For any open subset $U \subseteq \overline{X}$, there is a natural way to associate to any form $\omega \in \mathscr{A}^{p, q}_{si}(U)$ a current $T_\omega \in \Dist^{p, q}(U)$ given by
\[ T_\omega(\eta) = \frac{1}{(2 \pi i)^d}\int_U \omega \wedge \eta, \;\;\; (\eta \in \Gamma_c(U, \mathscr{A}_{rd}^{d - p, d - q})) \] Using Lemma \ref{LemmatempDist}, the current $T_\omega$ can also be described as the product of $\omega$ with the normalized trace distribution in $\Dist^{0,0}(U)$ defined by $\eta \in \Gamma_c(U, \mathscr{A}^{d,d}_{rd}) \mapsto \frac{1}{(2\pi i)^d} \int_U \eta$ (recall that every rapidly decreasing differential form is integrable). By Stokes formula \cite{borel1} we have $d T_\omega = T_{d \omega}$, $\partial T_\omega =  T_{\partial \omega}$, etc., which implies that there is a filtered morphism of complexes
\[ \iota : \mathscr{A}^*_{si} \to \Dist^*. \]
We now show a Poincar\'e lemma for tempered currents as our main application of Theorem \ref{LemmaBM}. The proof is adapted from \cite[Theorem 4.5]{BurgosLitcanu}.

\begin{theorem} \label{Propsicurrents}
The map
\[ \iota : (\mathscr{A}_{\overline{X}}^*(\log \, D), F) \to (\Dist^*, F) \]
is a filtered quasi-isomorphism compatible with the underlying real structures.
\end{theorem}

\begin{proof}
The statement of the Theorem is equivalent to showing that the inclusion
\begin{equation} \label{EquationQisom} \iota : \mathscr{A}_{\overline{X}}^{p, *}(\log D) \to \Dist^{p, *}
\end{equation}
is a quasi-isomorphism for any $p \geq 0$.
Since exactness can be checked at the level of stalks, it suffices to show that
\[ \iota_x : \mathscr{A}_{\overline{X}}^{p, *}(\log D)_x \to (\Dist^{p, *})_x \]
is a quasi-isomorphism for any $x \in \overline{X}$ and any $p \geq 0$.

Let $V = \Delta^d$ be an open neighborhood of $x$ in $\overline{X}$. One would like to use the Bochner-Martinelli operator $K'$ to construct an homotopy to show that this inclusion is a quasi-isomorph\-ism. The problem is that the formula $(K' T)(\omega) = T(K' \omega)$, for $T \in \Dist^{p, q}(V)$ and $\omega \in \Gamma_c(V, \mathscr{A}_{rd}^{d - p, d - q})$, is not well defined since $K' \omega$ is not necessarily a compactly supported differential form on $V$. Indeed $K' \omega$ might never vanish on $V$. To contour this problem, we will define a homotopy operator for the inductive systems defining the stalks at $x$ by using auxiliary truncating bump functions.

We start by fixing some notations. Let $(\epsilon_k)_{k \in \N}$, $1/2 \geq \epsilon_k > 0$ (e.g. $\epsilon_k = 1/2^{k+1}$), be a decreasing sequence of positive real numbers converging to zero. Then we can consider a system of open neighborhoods of $x$ in $\overline{X}$ given by $(V_k)_{k \in \N}$ with $V_k = \Delta_{\epsilon_k}^d$ and such that $U_k := V_k \cap X =(\Delta^*_{\epsilon_k})^m \times \Delta_{\epsilon_k}^n$ for some $m, n \geq 0$, $n + m = d$. Observe that, if $x \in X$, then $m = 0$ and the following also gives a proof of the quasi-isomorphism for the usual complex of currents with no growth conditions. We also let $\eta_k$ be a compactly supported smooth real function on $V_k$ such that $\eta_k \equiv 1$ on $\Delta_{\delta_k}^d$ and $\equiv 0$ outside $\Delta_{\delta'_k}^d$ for some $\epsilon_{k+1} < \delta_k < \delta'_k < \epsilon_n$, i.e. $\eta_n$ is a bump function which is constant equal to $1$ on some polydisc slightly bigger than $V_{n+1}$ and whose support is contained in some polydisc slightly smaller than $V_n$.

For any $k \in \N$, we define an operator
\[ K'_k : \Dist^{p, q}(V_k) \to \Dist^{p, q - 1}(V_{k+1}) \]
by the formula
\[ (K'_k T)(\varphi) := (-1)^{p+q} T(\eta_k K' \varphi), \;\;\; \varphi \in \Gamma_c(V_{k + 1}, \mathscr{A}_{rd}^{d-p, d - q + 1}). \]
Note that this formula is well defined as $\eta_k K' \varphi \in \Gamma_c(V_{k}, \mathscr{A}_{rd}^{d - p, d - q})$ by Theorem \ref{LemmaBM} (for the rapidly decreasing condition) and by construction of $\eta_k$ (for the compact support condition).
Let $T \in \Dist^{p, q}(V_k)$ and 
 let $\varphi \in \Gamma_c(V_{k+1}, \mathscr{A}_{rd}^{d - p, d - q})$ be a test form.
By Theorem \ref{LemmaBM} we have $\overline{\partial} K' \varphi + K' \overline{\partial} \varphi = \varphi$. As $\eta_n \varphi = \varphi$ (because $\eta_k \equiv 1$ on $V_{k+1}$), we have that 
\begin{align*} (\overline{\partial} K'_k T + K'_k \overline{\partial} T) (\varphi) &=  T (\eta_k K' \overline{\partial} \varphi) + T(\overline{\partial} (\eta_k K' \varphi)) \\
&= T(\eta_k (K' \overline{\partial} \varphi + \overline{\partial} K' \varphi)) + T (\overline{\partial} \eta_k \wedge K' \varphi) \\
&= T (\varphi) + T (\overline{\partial} \eta_k \wedge K'\varphi).
\end{align*}
Write $\varphi(z) = \sum_{I, J} f_{I, J}(z) \overline{\xi}_J \wedge \xi_I$, where $I, J$ run over subsets of $[1,d]$ of size $d - p$ and $d-q$ respectively. By
definition we have
\[ (K' \varphi)(z) = C_d \sum_{I, J, i} \left( z^{[1, m]} \int_{\Delta^d} \mu_i(z, w) \frac{f_{I, J}(w)}{w^{[1, m]} \overline{w}^{J \cap [1,m]}} dw \wedge d \overline{w} \right) \wedge d\overline{z}_{J - \{i\}} \wedge \xi_I , \]
where $C_d = \frac{(d - 1)!}{(2 i \pi)^d}$, $i$ runs over all elements of $J$ and where as before we noted $\mu_i(z, w) = \frac{\overline{z}_i - \overline{w}_i}{\| z - w\|^{2d}}$. We have then
\begin{multline*}
    T(\overline{\partial} \eta_k(z) \wedge K' \varphi) =\\ C_d T \left( \overline{\partial} \eta_k(z) \wedge \sum_{I, J, i} \left( z^{I^c \cap [1,m]} \int_{\Delta^m} \mu_i(z, w) \frac{f_{I, J}(w)}{w^{[1, m]} \overline{w}^{J \cap [1,m]}} d \overline{w} \wedge d w  \right) \wedge d\overline{z}_{J - \{i\}} \wedge dz_I \right) 
\end{multline*}
which, reordering the terms, can be written as
\[ C_d  \sum_{I, J, i} \pm \int_{\Delta^m} \big( A(T,\eta,I,J,i)(w) d \overline{w}_{J^c} \wedge  \xi_{w, I^c} \big) \wedge \big( f_{I, J}(w) \overline{\xi}_{w, J} \wedge \xi_{w, I} \big),
\]
where
\[
A(T,\eta,I,J,i)(w) := \pm T \left( z^{I^c \cap [1,m]} \mu_i(z, w) \, \overline{\partial} \eta_k(z) \wedge d\overline{z}_{J - \{i\}} \wedge dz_I \right)
\]
and we denoted as $\xi_{w, I^c} = \frac{1}{w^{I^c \cap [1,m]}} d w_{I^c}$, $\overline{\xi}_{w, J}$ and $\xi_{w, I}$ the analogous forms to $\xi_{I^c}$, $\overline{\xi}_J$ and $\xi_I$ but with the variable $w$, and
the ambiguity about the sign is coming from the exchange of the order of the differentials. Now, as $\overline{\partial} \eta_k(z)$ vanishes identically on the polydisc $\Delta_{\delta_k}^m$, then the differential form $z^{I^c \cap [1,m]} \mu_i(z, w) \, \overline{\partial} \eta_k(z) \wedge d\overline{z}_{J - \{i\}} \wedge dz_I$ is well defined on the domain $\{ (z, w) \in V_k \times V_{k+1}\}$ and is smooth with respect to both variables $z$ and $w$. Indeed, the only possible singularities are on $\mu_i(z, w)$ and occur in the diagonal but they are never reached as $\overline{\partial} \eta_k(z)$ vanishes identically in a strict neighborhood of $V_{n+1}$.
Hence 
$
A(T,\eta,I,J,i)(w)
$ 
is a smooth differential form on $V_{n + 1}$
(c.f. for example \cite[Theorem 2.1.3]{Hormander})  so 
$A(T,\eta,I,J,i)(w)\wedge d \overline{w}_{J^c} \wedge \xi_{w, I^c}$ is
a logarithmic differential form.
Therefore the current on $V_{k+1}$ given by
\[ \varphi \mapsto T( \overline{\partial} \eta_k(z) \wedge K' \varphi) \]
is the current induced by the logarithmic differential form
\[
 \sum_{I, J, i} A(T,\eta,I,J,i)(w) d \overline{w}_{J^c} \wedge \xi_{w, I^c} 
\]
in the variable $w$. We denote this logarithmic differential form by $-j_k(T)$. So we have shown that, for any $T \in \Dist^{p,q}(V_k)$, we have
\begin{equation} \label{EquationHomotopyFormula}
T|_{V_{k+1}} - \iota_{k+1} (j_k(T)) = \overline{\partial} K'_k T + K'_k \overline{\partial} T,
\end{equation}
where $\iota_k : \mathscr{A}_{\overline{X}}^{p, q}(\log \, D)(V_k) \to \Dist^{p, q}(V_k)$ denotes the map induced by $\iota$. This shows that we have defined operators
\[ j_k : \Dist^{p, q}(V_k) \to \Gamma(V_{k+1},\mathscr{A}_{\overline{X}}^{p, q}(\log D)), \]
which, by \eqref{EquationHomotopyFormula} give a quasi-inverse for the the inductive system of maps  of complexes $(\iota_k : \Gamma(V_k, \mathscr{A}_{\overline{X}}^{p, *}(\log \, D)) \to \Dist^{p, *}(V_k))_{k \in \N}$. In other words, we see from Equation \eqref{EquationHomotopyFormula} that if $T \in \Dist^{p,q}(V_k)$ is a closed tempered current on $V_k$, then its restriction to $V_{k+1}$ is cohomologous to the current associated with the differential form $j_k(T)$, which shows that the map induced by $\iota_x$ on cohomology is surjective. Conversely, if $\psi \in \Gamma(V_k, \mathscr{A}_{\overline{X}}^{p, q}(\log \, D))$ is such that $\iota_k(\psi) = \overline{\partial} T$, then by applying $\overline{\partial}$ to Equation \eqref{EquationHomotopyFormula} we obtain $\psi|_{V_{k+1}} = \overline{\partial} T|_{V_{k+1}} = \overline{\partial} K_k \psi + \overline{\partial} j_k(T)$, which shows that $\psi$ is exact, and thus the maps on cohomology induced by $\iota_x$ are injective. This shows that $\iota$ induces isomorphisms on stalks and finishes the proof of the result.
\end{proof}

\subsection{Application to Deligne--Beilinson cohomology}

We now apply the results obtained in the last section to give some useful descriptions of $DB$-cohomology. We recall that, for any $p \in \Z$ the $DB$-cohomology complex is defined as
\begin{equation} \label{defDBcohom2}
\R(p)_{\mathcal{D}} := \text{cone}(Rj_*\R(p) \oplus F^p\Omega^*_{\overline{X}}(\log D) \rightarrow Rj_* \Omega^*_X)[-1].
\end{equation}
Let $\mathscr{A}_{si}^*(\overline{X})$, resp. $\mathscr{A}_{si, \R(p-1)}^*(\overline{X})$, denote the global sections of the complex of sheaves $\mathscr{A}_{si}^*$, resp. $\mathscr{A}_{si, \R(p-1)}^*$.

\begin{proposition} \label{DBcoh-diff-form}
There is a quasi-isomorphism
\[ \R(p)_{\mathcal{D}} \simeq \mathrm{cone}( F^p \mathscr{A}_{si}^* \to \mathscr{A}_{si, \R(p-1)}^*)[-1], \]
where the arrow is induced by the projection $\pi_{p-1} : \C \to \R(p-1)$ defined by $\pi_{p - 1}(z) = \frac{z + (-1)^{p-1} \overline{z}}{2}$.
In particular, we have canonical isomorphisms
\begin{equation}\label{delignecoh0}
H^n_{\mathcal{D}}(X, \R(p)) \simeq \frac{\{ (\phi, \phi') \in F^p \mathscr{A}_{si}^n(\overline{X}) \oplus \mathscr{A}_{si, \R(p-1)}^{n-1}(\overline{X}) \,|\, d\phi = 0, d \phi' = \pi_{p-1}(\phi) \}}{\{d(\widetilde{\phi}, \widetilde{\phi}')\}},
\end{equation}
where $d(\widetilde{\phi}, 
\widetilde{\phi'}) = (d \widetilde{\phi}, d \widetilde{\phi'} - \pi_{p - 1}(\widetilde{\phi}))$.
\end{proposition}

\begin{proof}
The first assertion follows from Proposition \ref{PropBurgos}. The final one follows from the fact that the hypercohomology of the complex
\[ \mathrm{cone}( F^p \mathscr{A}_{si}^* \to \mathscr{A}_{si, \R(p-1)}^*)[-1] \]
is given by the cohomology of the corresponding complex of global sections since the sheaves are fine.
\end{proof}

\begin{remark} \label{unitDeligne} Let
$r_{\mathcal{D}}: H^1_{\mathcal{M}}(X, \Q(1)) \rightarrow H^1_{\mathcal{D}}(X, \R(1))$ be Beilinson regulator (cf. \S \ref{subsecintegralexpression}). Recall the canonical isomorphism $\mathcal{O}(X)^\times \otimes \Q \simeq H^1_{\mathcal{M}}(X, \Q(1))$. Then for $u \in \mathcal{O}(X)^\times$, the Deligne cohomology class $r_{\mathcal{D}}(u \otimes 1)$ is represented by $(d \log(u), \log|u|) \in F^1 \mathscr{A}_{si}^1(\overline{X}) \oplus \mathscr{A}_{si, \R(0)}^0(\overline{X})$.
\end{remark}

Consider the map $F^p \Dist^* \to \mathscr{D}^*_{si, \R(p-1)}$ given by the composition of the inclusion of $F^p \Dist^*$ into $\Dist^*$ and the map $\Dist^* \to F^p \mathscr{D}^*_{si, \R(p-1)}$ induced by $\pi_{p-1}: \C \to \R(p-1)$. With an abuse of notation, we will also denote this map by $\pi_{p-1}$. The following result is the key to the calculations of this article.

\begin{theorem} \label{TheoremAHCcurrents}
We have
\[ \R(p)_{\mathcal{D}} \simeq \mathrm{cone} \left( F^p \Dist^* \to \mathscr{D}^*_{si, \R(p-1)} \right)[-1]. \]
In particular,
\[ H^n_{\mathcal{D}}(X, \R(p)) \simeq \frac{\{ (S, T) \in F^p \Dist^n(\overline{X}) \oplus \mathscr{D}^{n-1}_{si, \R(p - 1)}(\overline{X}) : dS = 0, dT = \pi_{p-1}(S) \}}{\{d(\widetilde{S}, \widetilde{T})\}}, \]
where $d(\widetilde{S}, 
\widetilde{T}) = (d \widetilde{S}, d \widetilde{T} - \pi_{p - 1}(\widetilde{S}))$.
\end{theorem}

\begin{proof}
This is a consequence of  Theorem \ref{Propsicurrents}.
\end{proof}

In what follows, for $(S, T) \in F^p \Dist^n(\overline{X}) \oplus \mathscr{D}^{n-1}_{si, \R(p - 1)}(\overline{X})$ such that $dS=0$ and $dT=\pi_{p-1}(S)$, we will denote by $[(S,T)] \in H^n_{\mathcal{D}}(X, \R(p))$ the cohomology class of the pair $(S,T)$.

\begin{proposition} \label{compatibility} Let $x \in H^n_\mathcal{D}(X, \R(n))$ be a Deligne--Beilinson cohomology class which is represented, via the isomorphism of Proposition \ref{DBcoh-diff-form}, by a pair $(\phi, \phi')$ of smooth slowly increasing differential forms. Then via the isomorphism of Theorem \ref{TheoremAHCcurrents}, the class $x$ is represented by the pair of currents $(T_\phi, T_{\phi'})$.
\end{proposition}

\begin{proof}
This is immediate from Proposition \ref{DBcoh-diff-form} and the fact that the quasi-isomorphisms of Theorem \ref{Propsicurrents} is given by the natural inclusion of slowly increasing differential forms into tempered currents.
\end{proof}

The previous result allows us to decribe explicitly the Gysin morphism in $DB$-cohomology as follows. Let $\iota: X' \hookrightarrow X$ be a closed embedding of pure codimension $c$. Let $\overline{X'}$ denote a smooth compactification of $X'$ such that $D'=\overline{X'}-X'$ is a simple normal crossing divisor. Assume further that $\iota$ extends to a morphism $\overline{X'} \rightarrow \overline{X}$ that we still denote by $\iota$ and such that $\iota^{-1}(D) = D'$. Then we define the Gysin morphism
\begin{equation} \label{push}
\iota_*: H^n_{\mathcal{D}}(X', \R(p)) \rightarrow H^{n+2c}_{\mathcal{D}}(X, \R(p+c)).
\end{equation}
by $\iota_*[(S,T)]=[(\iota_*S, \iota_*T)]$. Here for a tempered current $T$, we denote by $\iota_*T$ the tempered current defined by $\omega \mapsto (\iota_*T)(\omega)=T(\iota^*\omega)$. This makes sense because the pullback of a rapidly decreasing differential form by a closed embedding is still rapidly decreasing.

We conclude with the construction of a linear form on $DB$-cohomology associated with certain rapidly decreasing differential forms.

\begin{proposition} \label{temperedpairing}
Let $n \in \N, p \in \Z$ and let $\omega \in \mathscr{A}^{2d-n}_{rd}(\overline{X})$ be a smooth closed rapidly decreasing differential form of Hodge type components inside $\{ (a, b) : a, b > d - p \}$. Then the assignment $(S, T) \mapsto T(\omega)$ induces a map
\[ \langle -, \omega \rangle : H^{n+1}_\mathcal{D}(X, \R(p)) \to \C. \]
\end{proposition}

\begin{proof}
By Theorem \ref{TheoremAHCcurrents} we have
\[ H^{n+1}_{\mathcal{D}}(X, \R(p)) = \{ (S, T) \in F^p \Dist^{n+1}(\overline{X}) \oplus \mathscr{D}^n_{si, \R(p-1)}(\overline{X}) \} / \sim. \]
In order to show that the linear form $(S, T) \mapsto T(\omega)$ is well defined at the level of cohomology, we need to see that it vanishes at any coboundary. Let $(\tilde{S}, \tilde{T}) \in F^p \Dist^{n}(\overline{X}) \oplus \mathscr{D}^{n-1}_{si, \R(p-1)}(\overline{X})$. We have $d(\tilde{S}, \tilde{T}) = (d\tilde{S}, d \tilde{T} - \pi_{p-1}(\tilde{S}))$ and we need to check that $(d \tilde{T} - \pi_{p-1}(\tilde{S})) (\omega) = 0$. We have
\[ d \tilde{T}(\omega) = -\tilde{T}(d \omega) = 0 \]
since $\omega$ is closed. Moreover, $\tilde{S} \in F^p \Dist^{n+1}(\overline{X})$, which implies that $\pi_{p-1}(\tilde{S})$ vanishes on forms of type $(a, b)$ with $a, b > d - p$, as $\tilde{S}$ vanishes on forms of type $(a, b)$ with $a > d - p$ and its complex conjugate vanishes on forms of type $(a,b)$ with $b > d - p$. This finishes the proof.
\end{proof}

\section{Classes in motivic cohomology}

This section is devoted to the construction of the classes in the motivic cohomology groups of Siegel sixfolds.

\subsection{Groups} \label{section-groups}

Let $\GSp_{2n}$ be the group scheme over $\Z$ whose $R$-points, for any commutative ring $R$ with identity, are described by
\[ \GSp_{2n}(R) = \{ A \in \GL_{2n}(R) \; | \; {}^t A J_n A = \nu(A) J_n, \; \nu(A) \in \mathbf{G}_m(R) \}, \] where $J_n$ is the matrix ${ \matrix 0 {I_n} {-I_n} 0}$, for $I_n$ denoting the $n \times n$ identity matrix. We will denote $\G:=\GSp_6$.

\subsubsection{Subgroups}

Let $F$ be a totally real \'etale quadratic $\Q$-algebra. Denote by $\GL_{2,F}^* /\Q$ the subgroup scheme of ${\rm Res}_{F/\Q} \GL_{2,F}$ sitting in the Cartesian diagram 
\[ \xymatrix{ 
\GL_{2,F}^* \ar@{^{(}->}[r] \ar[d] & {\rm Res}_{F/\Q} \GL_{2,F} \ar[d]^{\det} \\ 
\mathbf{G}_m  \ar@{^{(}->}[r] & {\rm Res}_{F/\Q} \mathbf{G}_{{\rm m},F}.
}
\]
For instance, when $F=\Q \times \Q$, we have
\[\GL_{2,F}^*=\{ (g_1, g_2) \in \GL_{2} \times \GL_2 \; | \; {\rm det}(g_1) = {\rm det}(g_2) \}.\]

Consider $F^{2}$ with its standard $F$-alternating form $\langle \; ,\; \rangle_F$. We fix the standard symplectic $F$-basis $\{ e_1 , f_1 \}$ and define $\langle \; ,\; \rangle_\Q$ to be ${\rm Tr}_{F/\Q} \circ \langle \; ,\; \rangle_F$. Then, by definition $\GL_{2,F}^* \subset \GSp(\langle \; ,\; \rangle_\Q)$. The extensions $F / \Q$ are parametrized by $a \in \Q^\times_{> 0} / ( \Q^\times _{> 0})^2,$ and we identify $F = \Q \oplus \Q \sqrt{a}$, for a representative $a$ of the corresponding class in $\Q^\times_{> 0} / ( \Q^\times_{> 0} )^2$. Fixing the $\Q$-basis of $F^2$ given by 
\[\{ \tfrac{1}{2\sqrt{a}}e_1 , \tfrac{1}{2}e_1 , \sqrt{a}f_1,f_1 \} \] gives an isomorphism $\GSp(\langle \; ,\; \rangle_\Q) \simeq \GSp_{4}$. Indeed, such a basis represents the alternating form $\langle \; ,\; \rangle_\Q$ as given by $J_2$. Thus we have an embedding \begin{eqnarray}\label{subemb} \GL_{2,F}^* \hookrightarrow \GSp(\langle \; ,\; \rangle_\Q) \simeq \GSp_{4}. \end{eqnarray}

Let $V_3$ be the standard representation of $\G$ with symplectic basis $\{ e_1, e_2, e_3, f_1, f_2, f_3  \}$. Define \[\GL_{2} \boxtimes \GSp_{4} := \{ (g_1,g_2) \in \GL_{2} \times \GSp_{4}\,:\, {\rm det}(g_1) = \nu(g_2)\}\] and consider the embedding \begin{eqnarray}\label{block} \GL_{2} \boxtimes \GSp_{4} \hookrightarrow \G \end{eqnarray} induced by the decomposition $V_3 = \langle e_1, f_1 \rangle \oplus \langle e_2, e_3, f_2, f_3 \rangle$.

By composing the maps of \eqref{subemb} and \eqref{block}, we construct the embedding
\[ \iota: \H := \GL_{2} \boxtimes \GL_{2, F}^* \hookrightarrow \G  \]

\subsection{Shimura varieties}

Keep the notation of the previous section and denote by $\mathbf{S}={\rm Res}_{\C/\R}{\mathbf{G}_m}_{/\C}$ the Deligne torus. Denote by $X_{\H}$ the $\H(\R)$-conjugacy class of \[h:\mathbf{S} \longrightarrow \H_{/\R}, \quad x+iy \mapsto \left( { \matrix {x} {y } {-y} {x} }, { \matrix {x} {y } {-y} {x} }, { \matrix {x} {y } {-y} {x} }  \right). \] The pair $(\H,X_{\H})$ defines a Shimura datum whose reflex field is $\Q$.
Denote by $\Sh_{\H}$ the corresponding Shimura variety of dimension $ 3$. If $U \subseteq \H(\Af)$ is a fiber product (over the similitude characters) $U_1 \times_{\Af^\times} U_2 $ of sufficiently small subgroups, we have
\[ \Sh_{\H}(U) = \Sh_{\GL_2}(U_1) \times_{\mathbf{G}_m} \Sh_{\GL_{2,F}^*}(U_2), \]
where $\times_{\mathbf{G}_m}$ denotes the fiber product over the zero dimensional Shimura variety $\pi_0(\Sh_{\GL_2})(D)$ of level $D = {\rm det}(U_1)  = \det(U_2)$. Recall that $\pi_0(\Sh_{\GL_2})(D)$ is a zero dimensional scheme defined over $\Q$ whose complex points are 
\[\pi_0(\Sh_{\GL_2})(D)(\C) \simeq \hat{\Z}^\times/D.\]
We also recall the reader that the complex points of $\Sh_{\H}(U)$ are given by
\[ \Sh_{\H}(U)(\C) = \H(\Q) \backslash \H(\A) / \Z_{\H}(\R) K_{\H, \infty} U, \]
where $\Z_{\H}$ denotes the center of $\H$ and $K_{\H, \infty} \subseteq \H(\R)$ is the maximal compact defined as the product $\mathrm{U}(1) \times \mathrm{U}(1) \times \mathrm{U}(1)$.

Notice that the embedding $\iota : \H \to \G$ induces another Shimura datum $(\G, X_{\G})$ of reflex field $\Q$. 
For any neat open compact subgroup $U$ of $\G(\Af)$, denote by $\Sh_{\G}(U)$ the associated Shimura variety of dimension $6$.
We also write $\iota: \Sh_{\H}(U \cap \H(\A_f)) \hookrightarrow \Sh_{\G}(U)$ the morphism of Shimura varieties induced by the group homomorphism $\iota: \H \hookrightarrow \G$. When $U$ satisfies the hypotheses of \cite[Lemma 2.1]{GSp6paper1} (e.g. if $U$ is contained in the kernel of reduction $\G(\widehat{\Z}) \rightarrow \G(\Z/d\Z)$ modulo $d$ for $d\geq 3$) then $\iota$ is a closed embedding of codimension $3 $.

\subsection{Motivic cohomology classes for \texorpdfstring{$\GSp_{6}$}{GSp6}}\label{Constructiongeneral}

We now define the cohomology classes we want to study in this article.

\subsubsection{Modular units and Eisenstein series}\label{sectionmodularunits}
The input of our construction are the modular units already considered by Beilinson and Kato, which are related to real analytic Eisenstein series by the second Kronecker limit formula.\\

Let $\mathbf{T}_2$ denote the diagonal maximal torus of $\GL_2$ and let $\mathbf{B}_2$ denote the standard Borel. Define the algebraic character $\lambda: \mathbf{T}_2 \rightarrow \mathbf{G}_m$ by $\lambda(\mathrm{diag}(t_1,t_2))=t_1/t_2$.
Let $\Sc(\A^2, \C)$, resp. $\Sc(\A^2_f, \overline{\Q})$, denote the space of Schwartz-Bruhat functions on $\A^2$, resp. the space of $\overline{\Q}$-valued Schwartz-Bruhat functions on $\A_f^2$. Given $\Phi \in \Sc(\A^2, \C)$, denote by \[f(g,\Phi,s):= |\operatorname{det}(g)|^s \int_{\GL_1(\A)} \Phi((0,t)g)|t|^{2s}d^\times t  \]
the normalized Siegel section in $\operatorname{Ind}_{\B_2(\A)}^{\GL_2( \A)}(|\lambda|^s)$ and define the associated Eisenstein series \begin{eqnarray}\label{eisensteinseriesgl2} E(g,\Phi,s) := \sum_{\gamma \in \mathbf{B}_2(\Q) \backslash \GL_2(\Q)} f(\gamma g,\Phi ,s).  \end{eqnarray}
Fix the Schwartz-Bruhat function $\Phi_\infty$ on $\R^2$ defined by $(x,y) \mapsto e^{- \pi(x^2 + y^2)}$ and, for each $\overline{\Q}$-valued function $\Phi_f \in \Sc(\Af^2,\overline{\Q})$, the smallest positive integer $N_{\Phi_f}$ such that $\Phi_f$ is constant modulo $N_{\Phi_f}\widehat{\Z}^2$. Finally, denote $ \Sc_0(\Af^2,\overline{\Q})\subset \Sc(\Af^2,\overline{\Q})$ the space of elements $\Phi_f$ such that $\Phi_f((0,0))=0$. We now state the following (classical) result, which relates modular units to values of the adelic Eisenstein series defined in \eqref{eisensteinseriesgl2}.
 
\begin{proposition} \label{KLF} Let $\Phi_f \in \Sc_0(\Af^2,\overline{\Q})$ with $N_{\Phi_f} \geq 3$, then there exists $$
u(\Phi_f) \in \mathcal{O}(\Sh_{\GL_2}(K(N_{\Phi_f})))^\times \otimes \overline{\Q}$$ such that for any $g \in \GL_2(\A)$ we have
\[ E(g,\Phi,s)={\rm log}|u(\Phi_f)(g)| + O(s) \;\;\; \text{ as $s \to 0$}, \] where $\Phi=\Phi_{\infty} \otimes \Phi_f$.
\end{proposition}

\begin{proof}
To our knowledge, this adelic interpretation of Siegel units is due to Colmez \cite{ColmezBourbaki}. We refer the reader to the second statement of \cite[Corollary 5.6]{PollackShahU21}, where $\nu_1$ is taken to be the trivial character.
\end{proof}

\begin{example}

When $\Phi_f= {\rm char}((0,1)+N\widehat{\Z}^2)$ for $N \geq 4$, the corresponding $u(\Phi_f) \in \O(\Sh_{\GL_2}(K(N)))^\times \otimes \Q$ is given by $\prod_{b \in ( \Z/N\Z)^\times} g_{0, b /N}^{\varphi(N)}$, where $g_{0, \star/N}$ is the  Siegel unit as in \cite[\S 1.4]{Kato}. Indeed, $\Sh_{\GL_2}(K(N))$ is a disjoint union of connected components all isomorphic to the modular curve $Y(N)$, which are indexed by the class of $[{\rm det}(k)] \in \widehat{\Z}^\times/(1+N \widehat{\Z})$, for $k \in \GL_2(\widehat{\Z})$. Choose a system of representatives given by the elements $k_d=\left( \begin{smallmatrix} 1 & \\ & d \end{smallmatrix} \right)$, as $d$ varies in $\widehat{\Z}^\times/(1+N \widehat{\Z})$. Then, a point in $\Sh_{\GL_2}(K(N))(\C)$ is represented by a pair $(z,k_d)$, with $z \in Y(N)$. By \cite[Corollary 5.6]{PollackShahU21}, as function on $\Sh_{\GL_2}(K(N))(\C)$, \[u(\Phi_f)(z,k_d) = \prod_{b \in ( \Z/N\Z)^\times} g_{(0, b/N) \cdot k_d^{-1}}^{\varphi(N)}(z)=\prod_{b \in ( \Z/N\Z)^\times} g_{0, b r_d/N}^{\varphi(N)}(z), \]
where $r_d$ denotes the inverse of $d$ modulo $N$, $\varphi$ is Euler's totient function, and  $g_{0, \star/N}$ is the  Siegel unit as in \cite[\S 1.4]{Kato}. Thus, $u(\Phi_f)$ descends to an element of $\O(\Sh_{\GL_2}(K_1(N)))^\times \otimes \Q$, as each $g_{0, b/N}$ does.
\end{example}

\subsubsection{The construction}   Let
$$u: \mathcal{S}_0(\A_f^2, \overline{\Q}) \rightarrow H^1_\mathcal{M}(\Sh_{\GL_2}, \overline{\Q}(1)) \simeq \mathcal{O}(\Sh_{\GL_2})^\times \otimes_{\Z} \overline{\Q}$$ be the $\GL_2(\A_f)$-equivariant map defined by $\Phi_f \mapsto u(\Phi_f)$, where $H^1_\mathcal{M}(\Sh_{\GL_2}, \overline{\Q}(1))$ denotes $\underrightarrow{\lim}_{V_1} H^1_\mathcal{M}(\Sh_{\GL_2}(V_1), \overline{\Q}(1))$ and $\mathcal{O}(\Sh_{\GL_2})^\times \otimes_{\Z} \overline{\Q}$ denotes $\underrightarrow{\lim}_{V_1} (\mathcal{O}(\Sh_{\GL_2}(V_1))^\times \otimes_{\Z} \overline{\Q})$, the limits being taken over all neat compact open subgroups $V_1 \subset \GL_2(\A_f)$.
 
Let
$$V_1 \subset \GL_2(\A_f), V_2 \subset \GL_{2, F}^*(\A_f) $$
denote neat compact open subgroups such that the images of $V_1$ and $V_2$ by the similitude characters are the same. Taking the fiber products over the similitude character, we obtain a compact open subgroup $V = V_1 \times_{\A_f^\times} V_2$ of $ \H(\A_f)$. Let $U \subset \G(\A_f)$ be a neat compact open subgroup such that the embedding $\iota$ induces a closed embedding $\Sh_{\H}(V) \hookrightarrow \Sh_{\G}(U)$ of codimension $3$. As a consequence, we have an induced map on motivic cohomology 
\[ \iota_{*}: H^{1}_{\mathcal{M}} \big(\Sh_{\H}(V), \overline{\Q}(1) \big) \to H^{7}_{\mathcal{M}}\big( \Sh_{\G}(U), \overline{\Q}(4) \big). \] 
The projection on the first factor of $\Sh_{\H}(V)$ is a morphism $p_1: \Sh_{\H}(V) \rightarrow \Sh_{\GL_2}(V_1)$. Hence we have the sequence of morphisms
$$
\begin{CD}
\mathcal{S}_0(\A_f^2, \overline{\Q})^{V_1} @>u>> H^1_\mathcal{M}(\Sh_{\GL_2}(V_1), \overline{\Q}(1))\\
@>p_{1}^*>> H^1_{\mathcal{M}} \big(\Sh_{\H}(V), \overline{\Q}(1) \big)\\
@>\iota_{*}>> H^{7}_{\mathcal{M}}\big( \Sh_{\G}(U), \overline{\Q}(4) \big).
\end{CD}
$$

\begin{definition} \label{eis1}
We define
$
\mathrm{Eis}_{\mathcal{M}}: \mathcal{S}_0(\A_f^2, \overline{\Q})^{V_1} \rightarrow H^{7}_{\mathcal{M}}\big( \Sh_{\G}(U), \overline{\Q}(4) \big)
$
to be the composite of these morphisms.
\end{definition}

\begin{remark} The notation $\mathrm{Eis}_{\mathcal{M}}$ is slightly abusive as these morphisms depend also on $U$, $V$ and the data entering in the definition of $\iota$.
\end{remark}

\section{Cohomology of the Siegel sixfold} 

In this chapter we recollect some classical general facts that will be used later or that serve as motivation for our constructions. We also describe the construction of the differential form $\omega_\Psi$ that will be used in the statement of the results.

\subsection{Representation theory}

We set the notations for the representation theory background needed to describe the component at infinity of the automorphic representations under consideration.

\subsubsection{Cartan decomposition}\label{cartandec} The maximal compact subgroup $K_\infty$ of $\Sp_{6}(\R)$ is described as
\[ K_\infty = \{ {\matrix A B {-B} A} \; | \; A A^t + B B^t = 1, A B^t = B A^t \}. \]
It is isomorphic to $U(3)$ via the map ${\matrix A B {-B} A} \mapsto A + iB$ and its Lie algebra is
\[ \mathfrak{k} = \{ {\matrix A B {-B} A} \; | \; A = - A^t, B = B^t \}. \] Letting 
$$
\mathfrak{p}_{\C}^\pm=\left\{ {\matrix A {\pm i A} {\pm i A} {-A} } \in M(6,\C) \;| \; A = A^t \right\},
$$
one has a Cartan decomposition \[ \mathfrak{sp}_{6,\C}=\mathfrak{k}_{\C} \oplus \mathfrak{p}^+_{\C} \oplus \mathfrak{p}^-_{\C}. \]

\subsubsection{Root system} For $1 \leq j \leq 3$, let $D_j$ be the square matrix of size $3$ with entry $1$ at position $(j, j)$ and $0$ elsewhere. Define
\[ T_j = - i {\matrix 0 {D_j} {- D_j} 0}. \] Then $\mathfrak{h} = \oplus_{j} \R \cdot T_j$ is a compact Cartan subalgebra of $\mathfrak{sp}_{6,\C}$. We let $(e_j)_j$ denote the basis of $\mathfrak{h}^*_\C$ dual to $(T_j)_j$. A system of positive roots for $(\mathfrak{sp}_{6,\C}, \mathfrak{h}_\C)$ is then given by
\begin{eqnarray*}
2 e_j, &\;& 1 \leq j \leq 3, \\
e_j + e_k, &\;& 1 \leq j < k \leq 3, \\
e_j - e_k, &\;& 1 \leq j < k \leq 3.
\end{eqnarray*}
The simple roots are $e_1 - e_2, e_2 - e_3, 2 e_3$. We note that $\mathfrak{p}^+_{\C}$ is spanned by the root spaces corresponding to the positive roots of type $2 e_j$ and $e_j+e_k$. We denote $\Delta = \{ \pm 2 e_j, \pm (e_j \pm e_k) \}$ the set of all roots, $\Delta_\mathrm{c} = \{ \pm(e_j - e_k)\}$ the set of compact roots and $\Delta_{\rm nc} = \Delta - \Delta_{\rm c}$ the noncompact roots. Finally, we note $\Delta^+, \Delta_{\rm c}^+$ and $\Delta_{\rm nc}^+$ the set of positive, positive compact and positive noncompact roots, respectively.

The corresponding root vectors for each root space are given as follows:
\begin{itemize}
    \item  For $1 \leq j \leq 3$, the element $X_{\pm 2e_j}={ \matrix {D_j} {\pm i D_j} {\pm i D_j} {- D_j}}$ spans the root space of $\pm 2 e_j$.
    \item For $1 \leq j < k \leq 3$, letting $E_{jk}$ be the matrix with entry $1$ at positions $(j, k)$ and $(k, j)$ and zeroes elsewhere, the elements $X_{\pm(e_j + e_k)}={\matrix {E_{jk}} {\pm i E_{jk}} {\pm i E_{jk}} {- E_{jk}}}$ spans the root space for of $e_j + e_k$.
    \item Finally, for $1 \leq j < k \leq 3$, letting $F_{j, k}$ be the matrix with entry $1$ at position $(j, k)$, $-1$ at position $(k, j)$ and zeroes elsewhere, the element $X_{\pm(e_j - e_k)}={\matrix {\pm F_{jk}} {- iE_{jk}} {i E_{jk}} {\pm F_{jk}}}$ spans the root space of the compact root $\pm (e_j - e_k)$. 
\end{itemize}

\subsubsection{Weyl groups} \label{weylgroups} Recall that the Weyl group of $\Sp_6$ is given by $W_{\Sp_6} = \{ \pm 1\}^3 \rtimes \mathfrak{S}_3$. The reflection $\sigma_j$ in the orthogonal hyperplane of $2 e_j$ simply reverses the sign of $e_j$ while leaving the other $e_k$ fixed. The reflection $\sigma_{jk}$ in the orthogonal hyperplane of $e_j - e_k$ exchanges $e_j$ and $e_k$ and leaves the remaining $e_\ell$ fixed. The Weyl group $W_{K_\infty}$ of $K_\infty \cong U(3)$ is isomorphic to $\mathfrak{S}_3$ and, via the embedding into $\G$, identifies with the subgroup of $W_\G$ generated by the $\sigma_{jk}$. With the identification $W_{\Sp_6} = N(T) / Z(T)$, an explicit description of $W_{\Sp_6}$ and $W_{K_\infty}$ is given as follows. The matrices corresponding to the reflections $\sigma_{jk}$ are
\[ {\matrix {S_{jk}} 0 0 {-S_{jk}} }, \]
where $S_{jk}$ is the matrix with entry $1$ at places $(\ell, \ell)$, $\ell \neq j, k$, $(k, j)$ and $(j, k)$ and zeroes elsewhere. The matrices corresponding to the reflection $\sigma_j$ in the hyperplane orthogonal to $2 e_j$ are of the form
\[ {\matrix 0 {T_j} {-T_j} 0}, \]
where $T_j$ denotes the diagonal matrix with $-1$ at the place $(j, j)$ and ones at the other entries of the diagonal.

\subsubsection{$K_\infty$-types}\label{kinftypes}

We previously defined the maximal compact subgroup $K_{\infty} \simeq U(3)$ of $\Sp_{6}(\R)$, with Lie algebra $\mathfrak{k}$, and we considered the Cartan decomposition $\mathfrak{sp}_{6,\C}=\mathfrak{k}_{\C} \oplus \mathfrak{p}^+_{\C} \oplus \mathfrak{p}^-_{\C}.$ Recall that $\mathfrak{p}^{\pm}_{\C}=\bigoplus_{\alpha \in \Delta_{\rm nc}^+} \C X_{\pm \alpha} $.

Denote by $(k_2, k_2 ,k_3) = k_1 e_1+ k_2 e_2 + k_3 e_3$, with $k_i\in \Z$ the integral weights. Integral weights are dominant for our choice of $\Delta_{\rm c}^+$ if $k_1 \geq k_2 \geq k_3$. 
Recall that there is a bijection between isomorphism classes of finite dimensional irreducible complex representations of $K_{\infty}$ and dominant integral weights, given by assigning to the representation $\tau_{(k_1, k_2, k_3)}$ its highest weight $(k_1, k_2, k_3)$. 

\subsection{Lie algebra cohomology}

Let $A_\G = \R^\times_+$ denote the identity component of the center of $\G(\R)$ and let $K_\G =A_\G K_\infty \subset \G(\R)$. The embedding $\mathfrak{sp}_{6,\C}\subset \mathfrak{g}_\C$ induces an isomorphism $\mathfrak{sp}_{6,\C}/ \mathfrak{k} \simeq  \mathfrak{g}_\C / ({\rm Lie}(K_\G))_\C.$ By \cite[II. Proposition 3.1]{BorelWallach}, for any discrete series $\pi_\infty$ associated with the trivial representation (cf. \S \ref{discreteseries} below), we have
\[ H^6(\mathfrak{g}, K_\G; \pi_\infty) = {\rm Hom}_{K_\infty}(\bigwedge^6 \mathfrak{sp}_{6,\C}/ \mathfrak{k}, \pi_\infty). \]
By using the Cartan decomposition above, we get \[\bigwedge^6 \mathfrak{sp}_{6,\C}/ \mathfrak{k} = \bigoplus_{p+q=6} \bigwedge^p \mathfrak{p}^+_{\C} \otimes_\C  \bigwedge^q \mathfrak{p}^-_{\C}.\] 

One can easily decompose each term of the sum above in its irreducible constituents (if treated as a $K_\infty$-representation via the adjoint action). This will be helpful for writing explicit elements in $H^6(\mathfrak{g}, K_\G ; \pi_\infty)$ according to the minimal $K_\infty$-type of $\pi_\infty$. Indeed, $\mathfrak{p}_\C^+$ (resp. $\mathfrak{p}_\C^-$) is the six dimensional irreducible representation of $K_\infty$ of weight $(2,0,0)$ (resp. $(0,0,-2)$). Using Sage package for Lie groups, one can see the following.

\begin{lemma} \label{Ktypesex}
We have the following decompositions in irreducible components of the following $K_\infty$-representations.
\begin{align*}
    \bigwedge^6 \mathfrak{p}^+_{\C} &= \tau_{(4,4,4)} \\
    \bigwedge^5 \mathfrak{p}^+_{\C} \otimes_\C \mathfrak{p}^-_{\C} &= \tau_{(4,2,2)} \oplus \tau_{(4,3,1)} \oplus \tau_{(4,4,0)} \\
    \bigwedge^4 \mathfrak{p}^+_{\C} \otimes_\C \bigwedge^2 \mathfrak{p}^-_{\C} &= \tau_{(2,1,1)} \oplus \tau_{(2,2,0)} \oplus 2  \tau_{(3,1,0)} \oplus 2   \tau_{(3,2,-1)} \oplus \tau_{(3,3,-2)}  \oplus \tau_{(4,0,0)}\oplus\\
    & \oplus \tau_{(4,1,-1)}\oplus \tau_{(4,2,-2)} \\ 
    \bigwedge^3 \mathfrak{p}^+_{\C} \otimes_\C \bigwedge^3 \mathfrak{p}^-_{\C} &= 2  \tau_{(0,0,0)} \oplus \tau_{(1,1,-2)} \oplus 2 \tau_{(1,0,-1)} \oplus \tau_{(2,-1,-1)} \oplus \tau_{(2,1,-3)} \oplus \tau_{(2,2,-4)}  \oplus \\
    & \oplus 4  \tau_{(2,0,-2)} \oplus \tau_{(3,-1,-2)} \oplus 2  \tau_{(3,0,-3)} \oplus \tau_{(4,-2,-2)}  \\
    \bigwedge^2 \mathfrak{p}^+_{\C} \otimes_\C \bigwedge^4 \mathfrak{p}^-_{\C} &= \tau_{(-1,-1,-2)} \oplus 2   \tau_{(1,-2,-3)} \oplus  \tau_{(1,-1,-4)} \oplus \tau_{(2,-3,-3)} \oplus \tau_{(2,-2,-4)} \oplus \\
    & \oplus \tau_{(0,-2,-2)}  \oplus 2  \tau_{(0,-1,-3)} \oplus \tau_{(0,0,-4)} \\ 
    \mathfrak{p}^+_{\C} \otimes_\C \bigwedge^5 \mathfrak{p}^-_{\C} &= \tau_{(-2,-2,-4)} \oplus \tau_{(-1,-3,-4)} \oplus \tau_{(0,-4,-4)} \\ 
     \bigwedge^6 \mathfrak{p}^-_{\C} &= \tau_{(-4,-4,-4)}.
    \end{align*}
\end{lemma}

It will be useful to have some explicit description of the components $\tau_{(2,2,-4)}$ and $\tau_{(4, -2, -2)}$ of $\bigwedge^3 \mathfrak{p}^+_{\C} \otimes_\C \bigwedge^3 \mathfrak{p}^-_{\C}$.

\begin{lemma} \label{LemmaHighestWeightVector}
The vector
\[ X_{(2,2,-4)} := (X_{2e_1} \wedge X_{2e_2} \wedge X_{e_1 + e_2}) \otimes (X_{-e_1 - e_3} \wedge X_{-e_2 - e_3} \wedge X_{-2e_3}) \]
is a highest weight vector of $\tau_{(2, 2, -4)}$. Analogously, a highest weight vector of $\tau_{(4, -2, -2)} \subseteq \tau_{(4,1, 1)} \otimes \tau_{(0, -3, -3)}$ is given by
\[ X_{(4, -2, -2)} := (X_{2 e_1} \wedge X_{e_1 + e_2} \wedge X_{e_1 + e_3}) \otimes (X_{-e_2 - e_3} \wedge X_{-2e_2} \wedge X_{- 2e_3}). \]
\end{lemma}

\begin{proof}
We have a decomposition of $K_\infty$ representations $\bigwedge^3 \mathfrak{p}^+_{\C} = \tau_{(3,3,0)} \oplus \tau_{(4,1,1)}$,  $\bigwedge^3 \mathfrak{p}^-_{\C} = \tau_{(-1, -1, -4)} \oplus \tau_{(0, -3, -3)}$. Since each of the four summands have multiplicity-free weights (i.e. every weight space has dimension at most one), then one can easily check that the vector
\[ X_{(2,2,-4)} := (X_{2e_1} \wedge X_{2e_2} \wedge X_{e_1 + e_2}) \otimes (X_{-e_1 - e_3} \wedge X_{-e_2 - e_3} \wedge X_{-2e_3}) \]
is a highest weight vector of $\tau_{(2, 2, -4)} \subseteq \tau_{(3,3,0)} \otimes \tau_{(-1, -1, -4)}$. Indeed $\tau_{(2,2,-4)}$ is the Cartan component of the tensor product and each of the terms in the tensor product defining $X_{(2,2,-4)}$ is a highest weight vector of its corresponding representation. Analogously, one shows that $X_{(4, -2, -2)}$ is a highest weight vector of $\tau_{(4, -2, -2)} \subseteq \tau_{(4,1, 1)} \otimes \tau_{(0, -3, -3)}$. Observe finally that we can pass from $\tau_{(2,2,-4)}$ to $\tau_{(4, -2, -2)}$ by the action of the matrix inducing  complex conjugation.
\end{proof}

\subsection{Discrete series \emph{L}-packets and test vectors} \label{discreteseries}

We recall some standard facts on discrete series. For any nonsingular weight $\Lambda \in \Delta$, define
\[ \Delta^+(\Lambda) := \{ \alpha \in \Delta \; : \; \langle \alpha, \Lambda \rangle > 0 \}, \;\;\; \Delta^+_c(\Lambda) = \Delta^+(\Lambda) \cap \Delta_c, \]
where $\langle \;,\;\rangle$ is the standard scalar product on $\R^3$. Let $\lambda$ be a dominant weight for $\Sp_{6}$ (with respect to the complexification $\mathfrak{h}_\C$ of the compact Cartan algebra $\mathfrak{h}$), let $\rho = \frac{1}{2} \sum_{\alpha \in \Delta^+} \alpha = (3, 2, 1),$ and let $\chi_\lambda : Z(U(\mathfrak{g})) \to \C$ be the map obtained by composing the Harish-Chandra isomorphism $Z(U(\mathfrak{g})) \cong U(\mathfrak{h})^{W_{\mathrm{Sp}_6}}$ with $\lambda : U(\mathfrak{h}) \to \C$. As $| W_{\Sp_6} / W_{K_\infty}| = 8$, the set of equivalence classes of irreducible discrete series representations of $\Sp_6(\R)$ with infinitesimal character $\chi_\lambda$ contains $8$ elements. More precisely (cf. \cite[Theorem 9.20]{knapp}), let us choose representatives $\{w_1, \ldots, w_{8}\}$ of $W_{\Sp_6} / W_{K_\infty}$ of increasing length and such that for any $1 \leq i \leq 8$, the weight $w_i (\lambda + \rho)$ is dominant for $K_\infty$. Then, for any $1 \leq i \leq 8$, there exists an irreducible discrete series $\pi_\infty^{\Lambda}$, where $\Lambda = w_i(\lambda + \rho)$, of Harish-Chandra parameter $\Lambda$, with infinitesimal character $\chi_\lambda$ and with minimal $K_\infty$-type of highest weight $ \Lambda + \delta_{\Sp_6} - 2\delta_{K_\infty}$, where $\delta_{\Sp_6}$ (resp. $\delta_{K_\infty}$) is the half-sum of roots (resp. of compact roots), which are positive with respect to the Weyl chamber in which $\Lambda$ lies, i.e.,
\[ 2 \delta_{\Sp_6} := \sum_{\alpha \in \Delta^+(\Lambda)} \alpha, \;\;\; 2 \delta_{K_\infty} := \sum_{\alpha \in \Delta^+_c(\Lambda)} \alpha. \]
Moreover, for $i \neq j$, $\Lambda = w_i (\lambda + \rho)$, $\Lambda' = w_j(\lambda + \rho)$, the representations $\pi_\infty^{\Lambda}$ and $\pi_\infty^{\Lambda'}$ are not equivalent and any discrete series of $\Sp_6(\R)$ is obtained in this way. We define the discrete series $L$-packet $P(V^\lambda)$ associated with $\lambda$ to be the set of isomorphism classes of discrete series of $\Sp_6(\R)$ whose Harish-Chandra parameter is of the form $\Lambda = w_i(\lambda + \rho)$, for some $1 \leq i \leq 8$. From now on, we will only consider discrete series associated with the trivial representation, i.e. $\lambda = 0$.

\begin{lemma} \label{discrete-series}
There exist two (unique) irreducible discrete series representations $\pi_\infty^{3,3}$ and $\overline{\pi}_\infty^{3,3}$ of $\Sp_6(\R)$ with Harish-Chandra parameter $(2, 1, -3)$ and $(3, -1, -2)$. Their central characters are trivial and their minimal $K_\infty$-types are, respectively, $\tau_{(2,2,-4)}$ and $\tau_{(4,-2,-2)}$.
\end{lemma}

\begin{proof}
This is an immediate consequence of \cite[Theorem 9.20]{knapp}, but we give some details that will be useful later. As explained in section \ref{weylgroups}, the reflections $\sigma_i$ in the orthogonal hyperplane of the long roots $2e_i$ for $1 \leq i \leq 3$ generate a system of representatives of $W_{\Sp_6}/W_{K_\infty}$. But these elements do not necessarily meet the condition that $\sigma_i\rho$ is dominant for $K_\infty$. In order to find representatives satisfying this condition, we have to multiply these elements by elements of $W_{K_\infty}$ to put the coordinates of $\sigma_i \rho$ in decreasing order. We find the representatives defined by their action on $\rho$ as follows: $w_1(3,2,1)=(3,2,1)$, $w_2(3,2,1)= (3,2,-1) $, $w_3(3,2,1)=(3,1,-2)$, $w_4(3,2,1)= (2,1,-3)$, $w_5(3,2,1)=(3,-1,-2)$, $w_6(3,2,1)=(2,-1,-3)$, $w_7(3,2,1)= (1,-2,-3)$, $w_8(3,2,1)=(-1,-2,-3)$.

For each $\Lambda = w_i \rho$, $1 \leq i \leq 8$, observe that $\delta_{\Sp_6} = w_i \rho$ and hence the minimal $K_\infty$-type of the discrete series $\pi_\infty^\Lambda$ is given by the formula
\[ \Lambda + \delta_{\Sp_6} - 2 \delta_{K_\infty} = 2 w_i \rho - 2 \delta_{K_\infty}. \]
Using this formula, one easily checks that the minimal $K_\infty$-types corresponding to each of representative $w_i$, $i = 1, \hdots, 8$ described above are, respectively, $\tau_1 = \tau_{(4,4,4)}, \tau_2 = \tau_{(4,4,0)}, \tau_3 = \tau_{(4, 2, -2)}, \tau_4 = \tau_{(2,2,-4)}, \tau_5 = \tau_{(4,-2,-2)}, \tau_6 = \tau_{(2,-2,-4)}, \tau_7 = \tau_{(0,-4,-4)}$ and $\tau_8 = \tau_{(-4,-4,-4)}$. The discrete series of the statement are those corresponding to $i = 4, 5$.
\end{proof}

\begin{lemma} \label{testvector} Let $\pi=\pi_\infty \otimes \pi_f$ be a cuspidal automorphic representation of $\G(\A)$ with archimedean component $\pi_\infty$ the discrete series such that $\pi_\infty|_{\Sp_6(\R)} \simeq \pi_\infty^{3,3} \oplus \overline{\pi}_\infty^{3,3}$ and such that $\pi_f^U \neq 0$, where $U \subseteq \G(\A_f)$ is a neat compact open subgroup. Let $\Psi = \Psi_\infty \otimes \Psi_f$ be a cusp form in the space of $\pi$ such that $\Psi_\infty$ is a highest weight vector of the minimal $K_\infty$-type $\tau_{(2,2,-4)}$ of $\pi_\infty^{3,3}$ and $\Psi_f$ is a nonzero vector in $\pi_f^U$. Let $X_{(2,2,-4)}$ be the highest weight vector in the $K_\infty$-type $\tau_{(2,2,-4)} \subset \bigwedge^{3} \mathfrak{p}^+_{\C} \otimes_\C \bigwedge^{3} \mathfrak{p}^-_{\C}$ of Lemma \ref{LemmaHighestWeightVector}. Then there exists a unique nonzero harmonic $(3, 3)$ differential form 
$$
\omega_\Psi \in \Hom_{K_{\infty}}\left(\bigwedge^{3} \mathfrak{p}^+_{\C} \otimes_\C \bigwedge^{3} \mathfrak{p}^-_{\C}, \pi_\infty^{3,3} \right) \otimes  \pi_f^U
$$
on $\Sh_{\G}(U)$ such that $\omega_\Psi(X_{(2,2,-4)}) = \Psi$. 
\end{lemma}

\begin{proof}
The results follows immediately from the fact that the minimal $K_\infty$-type $\tau_{(2,2,-4)}$ of $\pi_\infty^{3,3}$ appears with multiplicity one, cf. \cite[Theorem II.5.3]{BorelWallach}. 
\end{proof}


\subsection{Archimedean \emph{L}-functions and Deligne cohomology}

We end up this chapter by recalling some classical results on the relation between Deligne cohomology groups and the $L$-function of a motive, which explains when one expects to have nontrivial motivic cohomology. This section has only motivational interest and is not relevant to the main calculations of this text.

\subsubsection{Hodge decomposition}

Let $\pi = \pi_\infty \otimes \pi_f$ be a cuspidal automorphic representation of $\G(\A)$ with trivial central character. Assume that $\pi$ is cohomological for the trivial local system of the Shimura variety $\Sh_\G$. In this section, we describe the Hodge decomposition of the "interior motive" $M(\pi_f)$ defined by the $\pi_f$-isotypical component of the interior cohomology
\[ M(\pi_f) \otimes \pi_f= H_!^{6}(\Sh_{\G}, \Q)[\pi_f], \] 
which will allow us to describe the $\Gamma$-factor of its $L$-function and its simple poles. Recall from \cite{ChaiFaltings}, \cite{MokraneTilouine} that $M(\pi_f)$ is pure of weight $6$. The Hodge weights of $M(\pi_f)$ lie in the set of pairs
\[ (0, 6), (1 , 5), (2,4), (3,3), (4,2),(5,1),(6,0). \]
If there exists a finite place place $p$ such that $\pi_p$ is the Steinberg representation, then, by \cite[Lemma 2.8]{CLRG2}, we know that all of the types $(p,q)$ have positive multiplicity (more precisely, multiplicity $1$ when $p \neq q$ and $2$ when $p = q = 3$).

\subsubsection{Archimedean $L$-functions and Deligne cohomology}\label{subsec:archimedeanLfunctionsandDelignecoho}

We recall now, following \cite{Serre}, the definition of the $\Gamma$-factor of $M(\pi_f)$. The poles of this factor indicate the set of Tate twists of the motive $M(\pi_f)$ for which the Deligne cohomology is of positive dimension. It is precisely for these twists that one expects to construct nontrivial motivic cohomology classes.

The Betti realization $M_B(\pi_f)$ of $M(\pi_f)$ admits a Hodge decomposition
\[ M_B(\pi_f) \otimes \C = \bigoplus_{p + q = 6} H^{p, q}. \] Denoting by $\sigma = 1 \otimes c$, where $c$ is the complex conjugation, the involution on $M_B(\pi_f) \otimes \C$, we have that $\sigma(H^{p, q}) = H^{q, p}$. We also let $h^{p, q} = \dim_\C H^{p, q}$ be the Hodge numbers. We also write $H^{3, 3} = H^{3, +} \oplus H^{3, -}$, where $H^{3, \pm} = \{ x \in H^{3, 3} | \sigma(x) = \pm (-1)^3 x \}$, and let $h^{3, \pm 1} = \dim_\C H^{3, \pm}$. Let
\[ \Gamma_\R(s) = \pi^{-s/2} \Gamma(s / 2), \;\;\; \Gamma_\C(s) = 2 (2 \pi)^{-s} \Gamma(s) \]
be the real and complex Gamma factors, so that we have $\Gamma_\C(s) = \Gamma_\R(s) \Gamma_\R(s + 1)$. The archimedean factor of the $L$-function of $M(\pi_f)$ is then defined as
\begin{align} \label{EquationLinfty}
L_\infty(M(\pi_f), s) &= \Gamma_\R(s - 3)^{h^{3, +}} \Gamma_\R(s - 3 + 1)^{h^{3, -}} \prod_{p < q} \Gamma_\C(s - p)^{h^{p, q}}.
\end{align}
Since the Gamma function $\Gamma(s)$  has simple poles at $s = -n$, $n \in \N$, a simple calculation shows that the order of the pole of $L_\infty(M(\pi_f), s)$ at $s = m$, $m \in \Z$, is given by
\begin{equation} \label{Equationpole}
\sum_{m \leq p <q} h^{p, q} + \begin{cases} 0 & \! \! \text{if } m > 3 \\ h^{3, (-1)^{m - 3}} & \! \! \text{ otherwise} \end{cases}.
\end{equation}


Immediately from this formula, we get:

\begin{lemma} \label{orderpole}
The function $L_\infty(M(\pi_f), s)$ has a pole of order $h^{3, +}$ at $s = 3 $. In particular, if $\pi_p$ is the Steinberg representation for some finite prime $p$, then $L_\infty(M(\pi_f), s)$ has a simple pole at $s = 3$.
\end{lemma}

\begin{proof}
The first statement follows directly from Equation \eqref{Equationpole}. By \cite[Lemma 2.8]{CLRG2}, if $\pi_p$ is the Steinberg representation for some finite prime $p$, then $h^{p, \pm} = 1$, implying the second statement.
\end{proof}

We will also denote by $L(M(\pi_f), s)$ the uncompleted $L$-function of the motive $M(\pi_f)$ and we let $\Lambda(M(\pi_f), s) = L_\infty(M(\pi_f), s) L(M(\pi_f), s)$ the completed $L$-function of the motive. It satifies the functional equation
\begin{equation} \label{EquationFEQ} \Lambda(M(\pi_f), s) = \epsilon(M(\pi_f), s) \Lambda(M(\pi_f), 7 - s),
\end{equation}
where $\epsilon(M(\pi_f), s)$ denotes the $\epsilon$-factor of $M(\pi_f)$ (cf. \cite[(1.4.2)]{nekovar}).

\begin{remark} \label{equality-L-functions}
Let $\Sigma$ be the set of finite primes where the representation $\pi_f$ is ramified and denote by $L^\Sigma(M(\pi_f), s)$ the partial $L$-function associated to $M(\pi_f)$ (defined as the product of the local finite $L$-factors outside $\Sigma$). We note that, by \cite[Proposition 13.1]{KretShin}, the partial $L$-function $L^\Sigma(M(\pi_f)(3), s) = L^\Sigma(M(\pi_f), 3+ s)$ of $M(\pi_f)(3)$ agrees with the (partial) Spin $L$-function $L^{\Sigma}(\pi, {\rm Spin}, s)$ (note that both $L$-functions are centered at $s=1/2$).
\end{remark}

Recall the following result.

\begin{proposition} [{\cite[\S 2, Proposition]{Schneider}} \label{dimcohom}] We have
\[
\dim_\R H^{1}_\mathcal{D}(M(\pi_f)(7 - m)) = \begin{cases} {\rm ord}_{s = m} L(M(\pi_f), s) & \! \! \text{if } m < 3 \\ {\rm ord}_{s = m} L(M(\pi_f), s) - {\rm ord}_{s = m + 1} L(M(\pi_f), s) & \! \! \text{if } m = 3. \end{cases}
\]
\end{proposition}

In this text, we are interested in  the study of these Deligne cohomology groups for $m=3$.  The crucial hypothesis of Lemma \ref{testvector} and \S \ref{subsecintegralexpression} on the existence of a nontrivial test vector of the right Hodge type translates into asking that $h^{3,+} \neq 0$. Using the functional equation of the completed $L$-function for $M(\pi_f)$ and the fact that 
$L_\infty(M(\pi_f), s)$ is nonzero at $s=3+1$, one obtains
\[
h^{3,+} = {\rm ord}_{s = 3} L(M(\pi_f), s) - {\rm ord}_{s = 4} L(M(\pi_f), s).
\]
If $\pi_p$ is the Steinberg representation for some finite prime $p$, then, by Lemma \ref{orderpole} and Proposition \ref{dimcohom}, we have that $H^{1}_\mathcal{D}(M(\pi_f)(4))$ is of dimension one. In this case, Beilinson conjectures \cite{Beilinson} predict that the sum of Beilinson regulator $r_\mathcal{D}$ and of Betti cycle class map $c_B$ induces an isomorphism $$
(H^1_\mathcal{M}(M(\pi_f)(4)) \oplus N(M(\pi_f)(3))) \otimes_\Q \R \xrightarrow[]{(r_\mathcal{D},c_B)} H^1_\mathcal{D}(M(\pi_f)(4)),
$$
where $H^1_\mathcal{M}(M(\pi_f)(4))$ denotes the first motivic cohomology group of $M(\pi_f)(4)$, the group $N(M(\pi_f)(3))$ denotes algebraic cycles in $M(\pi_f)(3)$ up to homological equivalence and $H^1_\mathcal{D}(M(\pi_f)(4))$ denotes the first Deligne-Beilinson cohomology group of $M(\pi_f)(4)$.  In this article, we focus our attention to the cuspidal automorphic forms which contribute to the motivic part of the conjecture, i.e. precisely those for which $ L(M(\pi_f)(3), s)$ does not have a (simple) pole at $s=1$. Our main result, Theorem \ref{regulator3}, gives evidence towards the conjecture for these automorphic forms, by relating an element of $H^1_\mathcal{M}(M(\pi_f)(4))$ to the special value predicted by the conjecture. The forms which contribute to the other term, in the spirit of Tate's conjecture, are instead the ones of Hodge type $(3,3)$, for which $ {\rm ord}_{s = 1} L(M(\pi_f)(3), s) = -1$. These automorphic forms are lifts from the split $G_2$ under the exceptional theta correspondence (cf. \cite[Theorem 1.3]{PollackShah}). In \cite{CLRG2}, we address this part of Beilinson conjecture.

\section{Archimedean regulators and noncritical values of the Spin \texorpdfstring{$L$}{L}-function} \label{Section5}

We now apply the general results of chapter
\ref{sectiondelignecohomology} to the study of the archimedean realization of the elements constructed in Definition \ref{eis1}. From now on, we fix neat levels $U$ and $V = \H(\A_f) \cap U$ and smooth toroidal compactifications $ \overline{\Sh_{\G}(U)}$ and $\overline{\Sh_{\H}(V)}$ such that the boundaries are normal crossing divisors. By  \cite[Proposition 3.4]{harris-functorial}, it is possible to extend $\iota$ to a morphism $\overline{\Sh_{\H}(V)} \rightarrow \overline{\Sh_{\G}(U)}$.

\subsection{Integral expression for the pairing}\label{subsecintegralexpression}
Let
$$
\mathrm{Eis}_{\mathcal{M}}: \mathcal{S}_0(\A_f^2, \overline{\Q})^{V_1} \rightarrow H^{7}_{\mathcal{M}}\big( \Sh_{\G}(U), \overline{\Q}(4) \big)
$$
denote the morphism defined in Definition \ref{eis1}. Recall that it is defined as the composite 

\[\begin{CD}
\mathcal{S}_0(\A_f^2, \overline{\Q})^{V_1} @>u>> H^1_\mathcal{M}(\Sh_{\GL_2}(V_1), \overline{\Q}(1))\\
@>p_{1}^*>> H^1_{\mathcal{M}} \big(\Sh_{\H}(V), \overline{\Q}(1) \big)\\
@>\iota_{*}>> H^{7}_{\mathcal{M}}\big( \Sh_{\G}(U), \overline{\Q}(4) \big).
\end{CD} \]
Let

\[r_\mathcal{D}: H^{7}_{\mathcal{M}}\big( \Sh_{\G}(U), \overline{\Q}(4) \big) \rightarrow H^{7}_{\mathcal{D}}\big( \Sh_{\G}(U) / \R, \R(4) \big) \otimes_{\Q} \overline{\Q}\]
denote Beilinson higher regulator from motivic cohomology to real Deligne--Beilinson cohomology (as defined in Equation \eqref{DefRDB}). We refer to  \cite[\S 2.3]{Beilinson} for the definition of the regulator map, but we point out that we will only use some functorial properties of it and its description on units given by Remark \ref{unitDeligne}.

\begin{lemma}\label{regulator-currents} Let $  \Phi_f \in  \mathcal{S}_0(\A_f^2, \overline{\Q})^{V_1}$. Then, via the isomorphisms given by Theorem \ref{TheoremAHCcurrents}, the cohomology class $r_{\mathcal{D}}(\mathrm{Eis}_{\mathcal{M}}(\Phi_f))$ is represented by the pair of tempered currents $$(\iota_{*} T_{{p}_1^* d\log u(\Phi_f)}, \iota_{*} T_{{p}_1^*\log|u(\Phi_f)|}).$$
\end{lemma}

\begin{proof} According to \cite[\S 3.7]{Jannsen}, the regulator maps are morphisms between twisted Poincar\'e duality theories. As a consequence, we have the commutative diagram
{\small
  \[
\begin{CD}
H^1_{\mathcal{M}}(\Sh_{\GL_2}(V_1), \overline{\Q}(1)) @>p^*_{1}>> H^1_{\mathcal{M}}(\Sh_{\H}(V), \overline{\Q}(1)) @>\iota_{*}>> H^{7}_{\mathcal{M}}(\Sh_{\G}(U), \overline{\Q}(4)) \\
@Vr_{\mathcal{D}}VV                @Vr_{\mathcal{D}}VV             @Vr_{\mathcal{D}}VV\\
H^1_{\mathcal{D}}(\Sh_{\GL_2}(V_1)/\R, \R(1)) \otimes_{\Q} \overline{\Q} @>p^*_{1}>> H^1_{\mathcal{D}}(\Sh_{\H}(V)/\R, \R(1))\otimes_{\Q} \overline{\Q} @>\iota_{*}>> H^{7}_{\mathcal{D}}(\Sh_{\G}(U)/\R, \R(4))\otimes_{\Q} \overline{\Q}.\\
\end{CD}
\]}
Via the isomorphism of Proposition \ref{DBcoh-diff-form}, the morphism $p_{1}^*$ is induced by the pullback of differential forms. The statement of the Lemma follows from Remark \ref{unitDeligne}, Proposition \ref{compatibility} and the explicit description of the Gysin morphism in $DB$-cohomology \eqref{push}. 
\end{proof}

Let $\pi=\pi_\infty \otimes \pi_f$ be a  cuspidal automorphic representation of $\G(\A)$ with trivial central character such that $\pi_f^U \ne 0$. We assume that  $\pi_\infty|_{\mathrm{Sp}_6(\R)} \simeq \pi_\infty^{3,3} \oplus \overline{\pi}_\infty^{3,3}$. Following Lemma \ref{testvector}, we consider a cusp form $\Psi = \Psi_\infty \otimes \Psi_f$ in the space of $\pi$ such that $\Psi_\infty$ is a highest weight vector of the minimal $K_\infty$-type of $\pi_\infty^{3,3}$ and such that $\Psi_f$ is a nonzero vector in $\pi_f^U$ and we let $\omega_\Psi$ be the associated harmonic cuspidal differential form. Analogously we do it for $\overline{\pi}^{3,3}_\infty$ at the place of $\pi_\infty^{3,3}$.

\begin{theorem} \label{integralform}
Let $ \Phi_f \in  \mathcal{S}_0(\A_f, \overline{\Q})^{V_1}$. Then, we have
\[ \langle r_\mathcal{D}({\rm Eis}_{\mathcal{M}}(\Phi_f)), \omega_\Psi \rangle  = \int_{\Sh_\H(V)} p_1^*\log |u(\Phi_f)| \iota^*\omega_\Psi. \]
\end{theorem}

\begin{proof}
According to Lemma \ref{regulator-currents}, the $DB$-cohomology class \[r_\mathcal{D}({\rm Eis}_{\mathcal{M}}(\Phi_f)) \in H^7_{\mathcal{D}}(\Sh_\G(U), \R(4)) \otimes \overline{\Q}\] is represented by the pair of tempered currents $(\iota_{*} T_{{p}_1^* d\log u(\Phi_f)}, \iota_{*} T_{{p}_1^*\log|u(\Phi_f)|})$. It follows from \cite[Proposition 2.6.1]{HarrisZuckerIII} that the differential form $\omega_\Psi$ is rapidly decreasing. As it is moreover of Hodge type $(3,3)$, by Proposition \ref{temperedpairing} we have
\begin{eqnarray*}
\langle r_\mathcal{D}({\rm Eis}_{\mathcal{M}}(\Phi_f)), \omega_\Psi \rangle  &=& \iota_{*} T_{{p}_1^*\log|u(\Phi_f)|}(\omega_\Psi)\\
&=&   T_{{p}_1^*\log|u(\Phi_f)|}(\iota^*\omega_\Psi)\\
&=& \int_{\Sh_\H(V)} p_1^*\log |u(\Phi_f)| \iota^*\omega_\Psi.
\end{eqnarray*}
This finishes the proof.
\end{proof}

\subsection{The adelic integrals}\label{adelicintegralsec}

We now use Kronecker's limit formula to rewrite Theorem \ref{integralform} in terms of Rankin--Selberg automorphic integrals which involve Eisenstein series of $\GL_2$. Throughout the section, we let $\pi$ be a cuspidal automorphic representation as in \S \ref{subsecintegralexpression}.\\

Fix the choice of a measure on $\H(\A)$ as follows. For each finite place $p$ of $\Q$, we take the Haar measure $d h_p$ on $\H(\qp)$ that assigns volume one to $\H(\zp)$. For the archimedean place, we fix $X_0 := (X_{2e_1} \wedge X_{2e_2} \wedge X_{2e_3}) \otimes (X_{-2e_1} \wedge X_{-2e_2} \wedge X_{-2e_3}) \in \bigwedge^6 \mathfrak{h}_{\C}/\mathfrak{k}_{\H,\C} = \wedge^3 \mathfrak{p}_{\H, \C}^+ \otimes \wedge^3 \mathfrak{p}_{\H, \C}^- \subseteq \wedge^3 \mathfrak{p}_\C^+ \otimes \wedge^3 \mathfrak{p}_\C^-$ as the basis of the $1$-dimensional subspace $\wedge^6 \mathfrak{p}_{\H, \C}$ of $\wedge^3 \mathfrak{p}_\C^+ \otimes \wedge^3 \mathfrak{p}_\C^-$ \footnote{by a slight abuse of language, we denote by $X_0$ the vector in $\wedge^6 \mathfrak{p}_{\H, \C}$ as well as its image in $\wedge^3 \mathfrak{p}_\C^+ \otimes \wedge^3 \mathfrak{p}_\C^-$ by $\iota$, this should cause no confusion}. Such a choice of basis induces an equivalence between top differential $\omega$ forms on $X_\H = \H(\R) / K_{\H, \infty}$ and invariant measures $d_\omega h_\infty$ on $\H(\R)$ assigning measure one to $K_{\H, \infty}$ (cf. \cite[p. 83]{Harris97} for details). We then define $d h = d_\omega h_\infty \prod_{p} dh_p$.


Let $A \in \mathcal{U}(\mathfrak{h})$ be the operator such that $\mathrm{pr}_{\tau_{(2,2,-4)}}(X_0) = A \cdot X_{(2,2,-4)}$, where $\mathrm{pr}_{\tau_{(2,2,-4)}}: \bigwedge^{3} \mathfrak{p}^+_{\C} \otimes_\C \bigwedge^{3} \mathfrak{p}_\C^- \rightarrow \tau_{(2,2,-4)}$ denotes the projector to the irreducible factor $\tau_{(2,2,-4)}$ (cf. Lemma \ref{Ktypesex}). We  give an explicit formula for $A$ in Lemma \ref{projectionlemma} below.

Recall that, for any Schwartz--Bruhat function $\Phi_f \in \Sc_0(\Af^2,\overline{\Q})$, we let $\Phi = \Phi_\infty \otimes \Phi_f$, where $\Phi_\infty : \R^2 \to \R$ is given by the Gaussian $(x, y) \mapsto e^{-\pi (x^2 + y^2)}$. Attached to $\Phi$, we have defined in Equation \eqref{eisensteinseriesgl2} the real analytic Eisenstein series $E(g,\Phi,s)$ on $\GL_2$. After applying Proposition \ref{KLF} and Lemma \ref{regulator-currents}, Theorem \ref{integralform} implies the following. 

\begin{theorem}\label{adelicintegral} \label{adelicintegralforgsp6}
Let $\Phi_f \in \Sc_0(\Af^2,\overline{\Q})^{V_1}$ and let $\omega_\Psi$ be as in Theorem \ref{integralform}. We have
\begin{equation} \label{remarkonintegralform}
\langle r_\mathcal{D}({\rm Eis}_{\mathcal{M}}(\Phi_f)), \omega_\Psi \rangle =  \frac{h_{V}}{\mathrm{vol}(V)} \int_{\H(\Q) \Z_\G(\A) \backslash \H(\A)} E(h_1, \Phi,0) (A \cdot\Psi)(h) dh,
\end{equation}
where $h_{V} = 2^{-2}| \Z_\G(\Q) \backslash \Z_\G(\Af) / (\Z_\G(\Af) \cap V)|$.
\end{theorem}

\begin{proof}
Recall that Theorem \ref{integralform} gives \[ \langle r_\mathcal{D}({\rm Eis}_{\mathcal{M}}(\Phi_f)), \omega_\Psi \rangle =  \int_{\Sh_\H(V)} \xi \wedge \iota^* \omega_\Psi, \] where, by Proposition \ref{KLF}, $\xi = {p}_1^*\log|u(\Phi_f)| = {p}_1^* E(g, \Phi, 0)$. Thanks to the equivalence between top differential forms on $X_\H$ and invariant measures on $\H(\R)$ explained above, we pass from integrating over $\Sh_{\H}(V)$ to integrating over $\H(\Q) \Z_\G(\A) \backslash \H(\A)$. More precisely, we have
\begin{eqnarray*}
\int_{\Sh_\H(V)} \xi \wedge \iota^*\omega_\Psi &=& \int_{\H(\Q) \backslash \H(\A) / \Z_\H(\R) K_{\H, \infty} V} E(h_1, \Phi, 0) \omega_\Psi(X_0)(h) dh \\
&=& h_V \int_{\H(\Q) \Z_\G(\A) \backslash \H(\A) / V} E(h_1, \Phi, 0) \omega_\Psi(X_0)(h) dh \\
&=& \frac{h_V}{{\rm vol}(V)} \int_{\H(\Q) \Z_\G(\A) \backslash \H(\A)} E(h_1, \Phi, 0) \omega_\Psi(X_0)(h) dh,
\end{eqnarray*}
where we have used that $|\Z_\H(\R) /(\Z_\G \cap \H)(\R)|=2^{2}$.
Finally, note that $\omega_\Psi(X_0)(h) = (A \cdot\Psi)(h)$ by definition of $\omega_\Psi$. This proves the desired formula.
\end{proof}

We finish this section by giving an explicit formula for the projector $A$ of the theorem.

\begin{lemma} \label{projectionlemma}
Up to renormalizing $X_0$ by an explicit nonzero rational factor, the projections of $X_0$ to $\tau_{(2,2,-4)}$ and $\tau_{(4, -2, -2)}$ are given, respectively, by $A \cdot X_{(2,2,-4)}$ and $A' . X_{(4, -2, -2)}$, where $A = \operatorname{Ad}^2_{X_{e_3-e_2}} \circ \operatorname{Ad}^2_{X_{e_3-e_1}}$, $A' = \operatorname{Ad}^2_{X_{e_2-e_1}} \circ \operatorname{Ad}^2_{X_{e_3-e_1}}$.
\end{lemma}

\begin{proof}
Recall that $X_0$ is a weight $(0,0,0)$-vector (with respect to the action of $\mathfrak{h}_\C$) in $\wedge^3 \mathfrak{p}_\C^+ \otimes \wedge^3 \mathfrak{p}_\C^-$. Thus, we may write $X_0= Y \oplus \alpha x_{(2,2,-4)},$ where $Y$ belongs to the complement of $\tau_{(2,2,-4)}$ in the decomposition of $\wedge^3 \mathfrak{p}_\C^+ \otimes \wedge^3 \mathfrak{p}_\C^-$ as the sum of its weight subspaces (cf. Example \ref{Ktypesex}), the vector $x_{(2,2,-4)}$ is a generator of the one dimensional weight $(0,0,0)$-eigenspace of $\tau_{(2,2,-4)}$, and $\alpha$ is a scalar. We can assume $x_{(2,2,-4)} = \operatorname{Ad}^2_{X_{e_3-e_2}} \circ \operatorname{Ad}^2_{X_{e_3-e_1}}(X_{(2,2,-4)})$, where $X_{(2,2,-4)}=X_{2e_1} \wedge X_{2e_2} \wedge X_{e_1 + e_2} \otimes X_{-e_1-e_3} \wedge X_{-e_2-e_3} \wedge X_{-2e_3 }$ is a highest weight vector for $\tau_{(2,2,-4)}$. Since the the weight $(2,2,-4)$ has multiplicity one in $\wedge^3 \mathfrak{p}_\C^+ \otimes \wedge^3 \mathfrak{p}_\C^-$, we have that $\operatorname{Ad}^2_{X_{e_1-e_3}} \circ \operatorname{Ad}^2_{X_{e_2-e_3}}(Y) = 0$ and hence \[  \operatorname{Ad}^2_{X_{e_1-e_3}} \circ \operatorname{Ad}^2_{X_{e_2-e_3}}(X_0) = \alpha \operatorname{Ad}^2_{X_{e_1-e_3}} \circ \operatorname{Ad}^2_{X_{e_2-e_3}} \circ \operatorname{Ad}^2_{X_{e_3-e_2}} \circ \operatorname{Ad}^2_{X_{e_3-e_1}}(X_{(2,2,-4)}).\]
A direct computation \footnote{The authors have found the Lie algebra package of SageMath \cite{sage} very useful for these computations.} shows that
\begin{itemize}
    \item $\operatorname{Ad}^2_{X_{e_1-e_3}} \circ \operatorname{Ad}^2_{X_{e_2-e_3}}(X_0)=2^6  X_{(2,2,-4)}$,  \item $\operatorname{Ad}^2_{X_{e_1-e_3}} \circ \operatorname{Ad}^2_{X_{e_2-e_3}} \circ \operatorname{Ad}^2_{X_{e_3-e_2}} \circ \operatorname{Ad}^2_{X_{e_3-e_1}}(X_{(2,2,-4)})=2^{10} 3^2 5^2   X_{(2,2,-4)}$.
\end{itemize} 
Therefore, the projection of $X_0$ to $\tau_{(2,2,-4)}$ is $\tfrac{1}{3600}x_{(2,2,-4)}$. The other projection follows (with the same coefficient) by applying the action of complex conjugation. This finishes the proof of the Lemma.
\end{proof}

In what follows we use the main result of \cite{PollackShah} to write the adelic integral calculating the archimedean regulator in terms of a special value of the Spin $L$-function for $\GSp_6$.

\subsection{The Spin \emph{L}-function}

Recall the following.

\begin{definition} For a character $\chi$ of $\qle$, define \[ { L}(\chi , s ) := \begin{cases} (1 - \chi(\ell) \ell^{-s})^{-1} & \text{ if } \chi_{|_{\zle}}=1 \\ 1 & \text{ otherwise.} \end{cases}\]
\end{definition}

Let $\chi_0, \chi_1, \chi_2, \chi_3$ be smooth characters of $\qle$. They define an unramified character $\underline{\chi}$ of the Borel $B_\ell = T_\ell  U_{B,\ell}$ of $\G(\Q_\ell)$, which is trivial on the unipotent radical $U_{B,\ell}$, and on the diagonal torus $T_\ell$ is \[ d = diag(a,b,c, \mu a^{-1} ,\mu b^{-1},\mu c^{-1}) \mapsto \underline{\chi}(d) := \chi_1(a) \chi_2(b) \chi_3(c) \chi_0(\mu). \] 

The modular character of the Borel subgroup $\delta_{B_\ell}: T_\ell \to \C$ is given by \[ diag(a,b,c, \mu a^{-1} ,\mu b^{-1},\mu c^{-1}) \mapsto  \tfrac{|a|^6|b|^4|c|^2}{|\mu|^6}. \]

\begin{definition}
The (normalized) principal series representation $\pi_\ell(\underline{\chi}) = \pi_\ell(\chi_0, \chi_1, \chi_2, \chi_3)$ is the representation of $\G(\Q_\ell)$ whose underlying vector space is the space of functions $f: \G(\Q_\ell) \to \C$ satisfying
\[ f \left( d n g \right) = \tfrac{|a|^3|b|^2|c|}{|\mu|^3} \underline{\chi}(d) f(g), \]
for every $d = diag(a,b,c, \mu a^{-1} ,\mu b^{-1},\mu c^{-1})$ and $u \in U_{B, \l}$, and where the action of $\G(\Q_\ell)$ is given by right-translation.
\end{definition}

\begin{definition} 
Let $\pi_\ell = \pi_\ell(\underline{\chi})$ be an irreducible principal series. Its Spin $L$-factor is defined as \[ L(\pi_\ell, {\rm Spin} , s): = L( \chi_0,s) \prod_{k=1}^3 \prod_{1 \leq i_1 < \cdots < i_k \leq 3} L( \chi_0\chi_{i_1} \cdots \chi_{i_k},s). \]
\end{definition}

Now let $\pi = \pi_\infty \otimes \bigotimes_\ell' \pi_\ell$ be a cuspidal automorphic representation of $\G(\A)$ and let $\Sigma$ be a finite set of primes containing all the bad finite primes for $\pi$.

\begin{definition} 
The partial Spin $L$-function of $\pi$ is defined as \[ L^\Sigma(\pi, {\rm Spin} , s): = \prod_{\ell \not \in \Sigma} L(\pi_\ell, {\rm Spin} , s). \]
\end{definition}

\subsection{The integral representation of the Spin \emph{L}-function}

Let $\pi$ be a cuspidal automorphic representation of $\G(\A)$ with trivial central character.
Denote by $I(\Phi,\Psi,s)$ the integral \[\int_{\H(\Q) \Z_\G(\A) \backslash \H(\A)} E(h_1,\Phi,s) \Psi(h)dh,\] where $\Phi \in \mathcal{S}(\A^2)$ and $\Psi$ is a cusp form in the space of $\pi$. Here, we assume that $\Phi$ and $\Psi$ are factorizable.

\subsubsection{Fourier coefficients of type $(4 \, 2)$}

Let $\mathcal{O}$ be the unipotent orbit of $\G$ associated with the partition $(4 \, 2)$. To such $\mathcal{O}$ one can define a set of Fourier coefficients as follows.
Denote by $h_\mathcal{O}$ the one-dimensional torus \[t \mapsto \operatorname{diag}(t^3,t,t,t^{-3},t^{-1},t^{-1}) \] attached to $\mathcal{O}$ (cf. \cite[p. 82]{nilpotentorbits}).
Given any positive root $\alpha$ (for the action of the diagonal torus of $\G$), there is a nonnegative integer $n$ such that \[ h_\mathcal{O}(t) x_\alpha(u) h_\mathcal{O}(t)^{-1}=x_\alpha(t^n u),\]
where $x_\alpha$ denotes the one-parameter subgroup associated with $\alpha$.
Let $U_2(\mathcal{O})$ denote the subgroup of the unipotent radical $U_B$ of the standard Borel $B$ of $\G$ generated by the $x_\alpha$ such that $n \geq 2$. If $\alpha \ne e_2'-e_3'$, then $n \geq 2$, thus $U_2(\mathcal{O})$ coincides with the unipotent radical $U_P$ of the standard parabolic $P$ with Levi $\GL_1^2 \times \GL_2$, given by \[ \left\{ \left( \begin{smallmatrix} a &  & & \\  & g & &  \\  & & \mu a^{-1} & \\  & {} & & \mu{}^tg^{-1} \end{smallmatrix} \right) \; : \; a,\mu \in \GL_1, \; g \in \GL_2 \right \}. \] 
Let $L = \Q(\sqrt{D})$ be an \'etale quadratic extension of $\Q$ and let $\chi: U_P(\Q) \backslash U_P(\A) \to \C^\times$ be the nondegenerate unitary character associated with $L$ as in \cite[\S 2.1]{PollackShah}.

\begin{definition}\label{def:Fouriercoeff}
Let $\Psi$ be a cusp form in the space of $\pi$. Define the Fourier coefficient \[ \Psi_{\chi,U_P}(g):= \int_{U_P(\Q) \backslash U_P(\A)} \chi^{-1}(u) \Psi(ug)du. \] 
\end{definition}

We say that $\pi$ supports a Fourier coefficients of type $(4 \, 2)$ if $\Psi_{\chi, U_P} \neq 0$ for some choice of $\Psi$ and $L$. Our main results will concern automorphic cuspidal representations $\pi$ supporting a nonzero Fourier coefficient of type $(4 \, 2)$. We would like to point out that the multiplicity of these Fourier coefficients is infinite in general and has been proved to be one only for certain non-tempered cases treated by \cite{GanGurevich}. The following lemma shows that, under some mild assumptions, the cuspidal automorphic representations that we consider in this article always support such a nonzero coefficient.

\begin{lemma} \label{Fouriercoeff}
Let $\pi = \pi_\infty \otimes \pi_f$ be an automorphic cuspidal representation of $\G(\A)$ with trivial central character and such that $\pi_\infty = \pi_{\infty}^{3,3} \oplus \overline{\pi}_\infty^{3,3}$ (cf. Lemma \ref{discrete-series}). Assume that none of the irreducible cuspidal automorphic representations of ${\rm Sp}_6$ inside $\pi$ admits cuspidal theta lift to the split orthogonal group $\mathrm{SO}_{12}(\A)$. Then $\pi$ admits a nonzero Fourier coefficient of type $(4 \, 2)$.
\end{lemma}

\begin{proof}
We first observe that, by \cite[Theorem 2.7]{GinzburgRallisSoudry}, $\pi$ admits a nonzero Fourier coefficient of type $(6)$, $(4 \, 2)$ or $(2\, 2 \, 2)$. The former case corresponds to $\pi$ being generic. We will show that, according to our Hodge type and the assumption of the lemma, $\pi$ cannot have a Fourier coefficient of either type $(6)$ or $(2\, 2 \,2 )$.

We first show that, since $\pi_\infty$ is of Hodge type $(3,3)$, it cannot be generic. Recall that, after \cite[Proposition 6.19]{VoganZuckerman}, one can read the Hodge type of a discrete series representation defined by an element $w \in W_{\Sp_6} / W_{K_\infty}$ (notations are as in \S \ref{discreteseries}) as follows.  The element $w$ induces a reordering of the roots and the Hodge type of the discrete series representation (and hence of the cuspidal form $\pi$) is $(p, q)$, where $p$ (resp. $q$) is the number of positive (resp. negative) simple noncompact roots. Moreover, it is well known (cf. for instance \cite[Proposition 4.1]{MoeglinRenard}) that $\pi$ is generic exactly when $w = w_3, w_6$ in the notation of the proof of Lemma \ref{discrete-series}, which are of Hodge type $(4, 2)$ and $(2, 4)$. We conclude that $\pi$ cannot be generic, i.e. it does not support a Fourier coefficient of type $(6)$.

Now, suppose that $\pi$ admits a nonzero Fourier coefficient of type $(2\, 2 \, 2)$. By \cite[Main Theorem B] {GinzburgGurevich}, $\pi$ lifts to a nontrivial cuspidal representation on $\mathrm{SO}_{12}(\A)$. This contradicts our hypotheses, thus proving the result. 
\end{proof}

\begin{remark}
Observe that \cite[Main Theorem A]{GinzburgGurevich} implies that, if $\pi$ supports a Fourier coefficient of type $(4 \, 2)$, then it admits a nontrivial theta lift to $\mathrm{SO}_{12}(\A)$, but this lift will not be cuspidal. 
\end{remark}

\subsubsection{The unfolding}

Recall that we have assumed $\H = \GL_2 \boxtimes \GL_{2,F}^*$ with $F$ a totally real \'etale quadratic $\Q$-algebra. Let $\chi: U_P(\Q) \backslash U_P(\A) \to \C^\times$ be the character associated with $F$. Then we have the following.

\begin{proposition}[\cite{GanGurevich}, Proposition 7.1]\label{unfolding}
The integral $I(\Phi,\Psi,s)$ unfolds to \[ \int_{U_{B_\H}(\A)\Z_\G(\A) \backslash \H(\A)} f(h_1,\Phi,s) \Psi_{\chi, U_P}(h)dh, \]
where $U_{B_\H}$ is the unipotent radical of the Borel $B_\H$ of $\H$ and $f(h_1,\Phi,s)$ is the normalized section defined in \S \ref{sectionmodularunits}.
\end{proposition}

As explained in \cite{PollackShah}, the Fourier coefficient $\Psi_{\chi,U_P}$ might not factorize in general, thus Proposition \ref{unfolding} does not imply that $I(\Phi,\Psi,s)$ factors into an Euler product; however, in \cite{PollackShah}, the authors define and study local integrals corresponding to this unfolded integral, and use them to relate the global integral $I(\Phi,\Psi,s)$ to values of the Spin $L$-function of $\pi$, as we now recall.

\subsubsection{Connection with values of the Spin $L$-function}\label{connwithvalues}

Recall that we have taken $\pi= \pi_\infty \otimes ( \otimes_p' \pi_p)$ to be a cuspidal automorphic representation of $\G(\A)$ with trivial central character. We further suppose that $\pi$ supports a Fourier coefficient of type $(4 \; 2)$, i.e. that there is a cusp form $\Psi$ in the space of $\pi$ such that $\Psi_{\chi,U_P}$ is not identically zero.

We now recall following \cite{PollackShah} the definition of the local integrals corresponding to $I(\Phi,\Psi,s)$ and their properties. We start with the following definition.  

\begin{definition}
A $(U_P, \chi)$-model for $\pi_p$ is a linear functional $\Lambda: \pi_p \to \C$ such that \[ \Lambda(u \cdot v)=\chi(u) \Lambda(v),  \]
for all $v \in \pi_p$ and $u \in U_P$.
\end{definition}

For a $(U_P, \chi)$-model $\Lambda$ for $\pi_p$, a vector $v$ in the space of $\pi_p$, and $\Phi_p \in \mathcal{S}(\Q_p^2, \C)$, define
\[ I_p(\Phi_p, v , s) := \int_{U_{B_\H}(\qp)\Z_\G(\qp) \backslash \H(\qp)} f(g_1,\Phi_p,s) \Lambda( g \cdot v)  dg, \]
where $f(g,\Phi_p,s)\in  \operatorname{Ind}_{B_{\GL_2}(\qp)}^{\GL_2( \qp)}(\delta_{B_{\GL_2}}^s)$ denotes the function  \[ |\operatorname{det}(g)|^s \int_{\GL_1(\qp)} \Phi_p((0,t)g)|t|^{2s}  dt .\]
One has the following.
\begin{proposition}[\cite{PollackShah} Theorem 1.1, Proposition 5.1]\label{pollackshahgsp6} \leavevmode
\begin{enumerate}
\item If $p$ is a finite odd prime and $\pi_p$ is unramified, let $v_0 \in \pi_p$ be a spherical vector and let $\Phi_p = {\rm char}(\Z_p^2)$; then, for any $(U_P, \chi)$-model $\Lambda$ for $\pi_p$, we have
\[ I_p(\Phi_p, v_0 , s) = \Lambda(v_0)  L(\pi_p,{\rm Spin},s).\]

\item If $\pi_p$ is ramified and $v_0$ is a vector in $\pi_p$, then there exists a vector $v$ in $\pi_p$ and a function $\Phi_p \in \mathcal{S}_0(\Q_p^2, \overline{\Q})$ such that for all $(U_P, \chi)$-models $\Lambda$ \[I_p(\Phi_p,v,s) = \Lambda(v_0). \]
\end{enumerate}
\end{proposition}

\begin{remark}\label{remarkonprop5.1}
As explained in the proof of \cite[Proposition 5.1]{PollackShah}, in the case of a finite bad place $p$, one can choose $\Phi_p$ to be ${\rm char}((0,1)+p^n\Z_p)$, with $n$ a suitable positive integer depending on $v_0$.
\end{remark}
Finally, consider the archimedean integral \begin{equation} \label{archi-integral}
I_\infty(\Phi_\infty,\Psi,s):=\int_{U_{B_\H}(\R)\Z_\G(\R) \backslash \H(\R)} f(h_1,\Phi_\infty,s) \Psi_{\chi, U_P}(h)dh,
\end{equation}
where $\Psi_{\chi, U_P}(h)$ denotes the restriction of the Fourier coefficient to $\H(\R)$ (embedded naturally into $\H(\A)$).
This integral has been studied in \cite{GanGurevich}, where it is shown that it can be made nonzero at arbitrary $s=s_0$ if one has some freedom on the choice of $\Phi_\infty$ and $\Psi_\infty$.

\subsection{The regulator computation}
We now state and prove one of the main results of this manuscript.

Let $\pi=\pi_\infty \otimes (\otimes^{'}_p \pi_p)$ be a cuspidal automorphic representation of $\G(\A)$ with trivial central character, such that \begin{itemize}
    \item[(DS)]  $\pi_\infty$ is a discrete series whose restriction to $\Sp_6(\R)$ is $\pi_\infty^{3,3} \oplus \bar{\pi}_\infty^{3,3}$,
    \item[(FC)] there exists a factorizable cusp form $\Psi=\Psi_\infty \otimes \Psi_f$ in the space of $\pi$ which admits a nonzero Fourier coefficient of type $(4 \, 2)$ associated with a real \'etale quadratic extension $F/\Q$, and such that $\Psi_\infty$ is a highest weight vector of the minimal $K_\infty$-type of $\pi_\infty^{3,3}$.
\end{itemize}
   
   \begin{remark}
    In view of Lemma \ref{Fouriercoeff}, the existence of a Fourier coefficient of type  $(4 \, 2)$ might be partly replaced by asking that $\pi$ lifts to a noncuspidal automorphic representation of the split $\mathbf{SO}_{12}(\A)$. Furthermore, notice that if $\pi$ does not support such a Fourier coefficient, the following results will still be true but all the terms will vanish.
   \end{remark}
   
Let $\H = \GL_2 \boxtimes \GL_{2,F}^*$ and fix a neat open compact subgroup $U=\prod_{p} U_p$ of $\G(\Af)$, for which $\Psi_f \in \pi_f^U$, and $\Phi_f =\otimes_p \Phi_p \in \Sc_0(\Af^2,\Q)^{V_1}$, with $V_1={\rm pr}_1(U \cap \H(\A_f))$, which satisfy the following. Let $\Sigma$ be a finite set of primes containing  $\infty$ and all the bad finite primes for $\pi$;
for every prime $p \not\in \Sigma$, we set $U_p=\G(\zp)$ and $\Phi_p={\rm char}(\Z^2_p)$. 

\begin{theorem} \label{regulator3} Assume (DS), (FC) and that $\Psi$ is invariant under $\G(\zp)$ for every $p \not \in \Sigma$. Then we have
\[ \langle r_\mathcal{D}({\rm Eis}_{\mathcal{M}}(\Phi_f)), \omega_{\Psi} \rangle = C_V \lim_{s \to 0} \left (I_\Sigma(\Phi_\Sigma, A \cdot\Psi,s)  { L}^\Sigma(\pi, {\rm Spin} ,s)\right),\]
where ${L}^\Sigma(\pi, {\rm Spin} ,s) = \prod_{p \not\in \Sigma} {L}(\pi_p, {\rm Spin} ,s) $, and 
\[ I_\Sigma(\Phi_\Sigma, A \cdot\Psi,s):=\int_{U_{B_\H}(\A_\Sigma)\Z_\G(\A_\Sigma) \backslash \H(\A_\Sigma)} f(h_1,\Phi_\Sigma,s) (A \cdot\Psi)_{\chi, U_P}(h)dh,
\]
the operator $A$ is defined in Lemma \ref{projectionlemma} and $C_V = \frac{h_V}{{\rm vol}(V)}$ are as in Theorem \ref{adelicintegralforgsp6}.
\end{theorem}

\begin{proof}
This follows from Theorem \ref{adelicintegralforgsp6} and Proposition \ref{pollackshahgsp6}. For the sake of clarity, we give a sketch of its proof, following \cite{Piatetski-Shapiro-Rallis}.

Recall that, from Theorem \ref{adelicintegralforgsp6} and Proposition \ref{unfolding}, we have
\begin{align*}
   \langle r_\mathcal{D}({\rm Eis}_{\mathcal{M}}(\Phi_f)), \omega_{\Psi} \rangle &= C_{V} \int_{\H(\Q) \Z_\G(\A) \backslash \H(\A)} E(h_1,\Phi,0) \Psi'(h) dh \\
   &= C_{V} \int_{U_{B_\H}(\A)\Z_\G(\A) \backslash \H(\A)} f(h_1,\Phi,0) \Psi'_{\chi, U_P}(h)dh \\
   &=: C_{V} I(\Phi,\Psi',0),
\end{align*}
where we denote $C_V=\frac{h_{V}}{{\rm vol}(V)}$ and $\Psi' = A \cdot\Psi$.
   
We now study the quantity $I(\Phi,\Psi',0)$. Given a finite set of primes $S$ containing $\Sigma$, define 
\begin{equation} \label{semilcoaint} I_S(\Phi_S,\Psi',s):=\int_{U_{B_\H}(\A_S)\Z_\G(\A_S) \backslash \H(\A_S)} f(h_1,\Phi_S,s) \Psi'_{\chi, U_P}(h)dh,
\end{equation}
where $\Phi_S=\prod_{v \in S} \Phi_v$. Then, in the range of convergence \[ I(\Phi,\Psi',s)= \lim_{\Sigma \subseteq S}  I_S(\Phi_S,\Psi',s). \]
Notice that, by \cite[Theorem 2.7]{PollackShah}, if $p \not \in S$,
\[ I_{S \cup \{ p \}}(\Phi_{S \cup \{ p \}},\Psi',s)={L}(\pi_p, {\rm Spin} ,s)   I_S(\Phi_S,\Psi',s).\]  
Indeed, $I_{S \cup \{ p \}}(\Phi_{S \cup \{ p \}},\Psi',s) $ equals to
\[ \int_{U_{B_\H}(\A_{S})\Z_\G(\A_{S}) \backslash \H(\A_{S})}\!\! \! \! \! \! \! f(h_1,\Phi_{S},s) \left( \int_{U_{B_\H}(\qp)\Z_\G(\qp) \backslash \H(\qp)} \!\! \!\!\! \! \! f(h_{1,p},\Phi_p,s) \Psi'_{\chi, U_P}(h_p h)dh_p \right) dh. \]
As $p \not \in \Sigma$, $\pi_p$ is unramified at $p$. Fix a spherical vector $v_0$ for $\pi_p$; as $\Psi'$ is invariant under the action of $\G(\zp)$, there is a $\G(\zp)$-equivariant map $\varphi_p: \pi_p \to \pi$ sending $v_0$ to $\Psi'$. Thus, for a fixed $h \in \G(\A_S)$, the functional $\Lambda_h: \pi_p \to \C$ defined by $\Lambda_h(v):=\varphi_p(v)_{\chi, U_P}(h)$ is clearly a $(U_P,\chi)$-model and $\Lambda_h(h_p \cdot v_0)=\Psi'_{\chi, U_P}(h_p h)$. Proposition \ref{pollackshahgsp6}$(1)$ implies then
\begin{align*}
\int_{U_{B_\H}(\qp)\Z_\G(\qp) \backslash \H(\qp)} &{} \!\! \!\!\! \! \!f(h_{1,p},\Phi_p,s) \Psi'_{\chi, U_P}(h_p h)dh_p \\
&= \int_{U_{B_\H}(\qp)\Z_\G(\qp) \backslash \H(\qp)}\!\! \!\!\! \! \! f(h_{1,p},\Phi_p,s) \Lambda_h(h_p \cdot v_0) dh_p \\
&={L}(\pi_p, {\rm Spin} ,s)  \Lambda_h(v_0) \\
&= {L}(\pi_p, {\rm Spin} , s)  \Psi'_{\chi, U_P}(h), \end{align*}
which implies the desired equality. 

Taking the limit varying the set $S \supseteq \Sigma$, we get \[ I(\Phi,\Psi',0) = \lim_{s\to 0} \left( I_\Sigma(\Phi_\Sigma, \Psi',s) {L}^\Sigma(\pi, {\rm Spin} ,s) \right),\] 
which implies the result.
\end{proof}

The following Corollary studies further the finite integral $I_\Sigma(\Phi_\Sigma, A \cdot\Psi,s)$. For this, we fix the components at finite bad primes of our Schwartz function $\Phi$ and the level. If $p \in \Sigma$ is a finite place, we let  $\Phi_p={\rm char}((0,1)+p^n\Z_p^2)$, and \[U_p =\left\{g \in \G(\zp) \; : \; g \equiv I \; \text{( mod }p^n) \right\}, \] with $n$ the positive integer given by Proposition \ref{pollackshahgsp6}(2) and Remark \ref{remarkonprop5.1}.

\begin{corollary} \label{regulator3coro}
Let $\Psi$ and $\omega_{\Psi}$ be as in Theorem \ref{regulator3}. Then, there exist $\widetilde{\Psi}=\Psi_\infty \otimes \widetilde{\Psi}_f$ in $\pi$, with $\Psi_f$ and $\widetilde{\Psi}_f$ coinciding outside $\Sigma$, such that
\[ \langle r_\mathcal{D}({\rm Eis}_{\mathcal{M}}(\Phi_f)), \omega_{\widetilde{\Psi}} \rangle = C_V  \lim_{s\to 0} \left(I_\infty(\Phi_\infty, A \cdot \Psi, s)  {L}^\Sigma(\pi, {\rm Spin}, s)\right),\]
where 
\[I_\infty(\Phi_\infty,A \cdot \Psi,s) = \int_{U_{B_\H}(\R)\Z_\G(\R) \backslash \H(\R)} f(h_1,\Phi_\infty,s) (A \cdot \Psi)_{\chi, U_P}(h)dh.  \]
\end{corollary}

\begin{proof}

Let $p \in \Sigma$ be a finite place. As $\Psi$ is factorizable and $\pi_p$ irreducible, there exists a $\G(\Q_p)$-equivariant map $\varphi_p: \pi_p \to \pi$, which sends $v_p \mapsto A \cdot\Psi$, where the vector $v_p \in \pi_p$ denotes the local component of $A \cdot\Psi$. By the choice of $\Phi_p$,
Proposition \ref{pollackshahgsp6} $(2)$ and Remark \ref{remarkonprop5.1} produce a vector $\widetilde{v}_p$ in $\pi_p$ which depends on $v_p$, such that for all $(U_P, \chi)$-models $\Lambda$ of $\pi_p$ \[I_p(\Phi_p,\widetilde{v}_p,s) = \Lambda(v_p). \] For a fixed $h \in \G(\A_{\Sigma \smallsetminus \{ p \}})$, define the functional $\Lambda_h: \pi_p \to \C$ as $\Lambda_h(v):=\varphi_p(v)_{\chi, U_P}(h)$. This is a $(U_P,\chi)$-model for $\pi_p$ and $\Lambda_h(h_p \cdot v)=\varphi_p(v)_{\chi, U_P}(h_p h)$.  We have that, for this model $\Lambda_h$,  
\begin{align*}
 \int_{U_{B_\H}(\qp)\Z_\G(\qp) \backslash \H(\qp)}\!\! \!\!\! \! \! f(h_{1,p},\Phi_p,s) \Lambda_h(h_p \cdot \widetilde{v}_p) dh_p =  \Lambda_h(v_p) =  (A \cdot\Psi)_{\chi, U_P}(h). \end{align*}
Now one proceeds similarly to the proof of Theorem \ref{regulator3} to use this equality to show that \[I_\Sigma(\Phi_\Sigma, A \cdot\widetilde{\Psi},s) = I_\infty(\Phi_\infty,A \cdot \Psi,s),  \] where the cusp form $\widetilde{\Psi}$ coincides with $\Psi$ at $\infty$ and away from $\Sigma$, and has component $\widetilde{v}_p$ at each $p \in \Sigma$.
Finally, since $\tilde{\Psi}$ and $\Psi$ coincide almost everywhere as well as at $\infty$, the result follows from Theorem \ref{regulator3}.

\end{proof}

We conclude with two remarks on the nontriviality of the limiting value
\[\lim_{s\to 0} \left(I_\infty(\Phi_\infty, A \cdot \Psi, s)  {L}^\Sigma(\pi, {\rm Spin}, s)\right) \]
and its relation to Beilinson conjectures.

\begin{remark}\label{remarkontaylorexpansion}
The archimedean integral $I_\infty(\Phi_\infty,A \cdot\Psi,s)$ is expected to coincide with the archimedean factor $L_\infty(M(\pi_f), 3 + s)$ defined in Equation \eqref{EquationLinfty} of the completed motivic $L$-function, up to a non-zero constant depending only on the normalization of $\Psi_\infty$. This would imply by Lemma \ref{orderpole} that $I_\infty(\Phi_\infty,A \cdot\Psi,s)$ has a simple pole at $s=0$ under the mild assumption that $\pi_p$ is the Steinberg representation at some finite prime $p$. Recall from Remark \ref{equality-L-functions} that $L^\Sigma(\pi, \mathrm{Spin}, s) = L^\Sigma(M(\pi_f), s + 3)$. Let us write (cf. \S \ref{subsec:archimedeanLfunctionsandDelignecoho}) \[\Lambda(M(\pi_f), s) = L_\infty(M(\pi_f), s) L_\Sigma(M(\pi_f), s) L^\Sigma(M(\pi_f), s),\] where $L_\Sigma(M(\pi_f), s)$ denotes the product of the local $L$-factors over the ramified finite places of $M(\pi_f)$. The functional equation \eqref{EquationFEQ} then gives $\Lambda(M(\pi_f), s + 3) = \epsilon(M(\pi_f), s+3) \Lambda(M(\pi_f), 4 - s)$. Assuming that both $\Lambda(M(\pi_f), 4 - s)$ and $L_\Sigma(M(\pi_f), s + 3)$ have no pole at $s = 0$ then we have that $L^{\Sigma}(M(\pi_f), 3 + s)$ vanishes at $s=0$ with the same order as that of the pole of $L_\infty(M(\pi_f), s+ 3)$ at $s = 0$. Putting everything together, one deduces that $I_\infty(\Phi_\infty, A \cdot \Psi, s)$ has a simple pole at $s = 0$ and that $L^\Sigma(\pi, \mathrm{Spin}, s)$ has a simple zero at $s = 0$. Hence we have (up to the non-zero constant $C_V$ and the normalization of $\Psi_\infty$)
\begin{align*} \langle r_\mathcal{D}({\rm Eis}_{\mathcal{M}}(\Phi_f)), \omega_{\widetilde{\Psi}} \rangle &=  \lim_{s\to 0} \left(I_\infty(\Phi_\infty, A \cdot \Psi, s)  {L}^\Sigma(\pi, {\rm Spin}, s)\right) \\
&= \mathrm{Res}_{s = 0} (L_\infty(M(\pi_f), s+3)) \cdot \frac{d}{ds}|_{s = 0} L^\Sigma(M(\pi_f), s + 3) \\
&=  \frac{1}{2 \pi^6} \frac{d}{ds}|_{s = 0} L^\Sigma(M(\pi_f), s + 3),
\end{align*}
where the last equality follows from using \cite[Lemma 2.8]{CLRG2} to calculate the residue of the archimedean factor. This relates the complex regulator of our motivic classes to the special value at $s = 3$ of the motivic $L$-function. In particular, the dimension of motivic cohomology of $M(\pi_f)(3)$ is greater or equal to the order of vanishing of $L(M(\pi_f)(3), s) = L(M(\pi_f), s + 3)$ at $s= 0$, as predicted by Beilinson conjectures. The analogous expectation about $I_\infty(\Phi_\infty, A \cdot \Psi, s)$ in the case of $\GSp_4$ was confirmed in the articles \cite{Miyazaki}, \cite{Moriyama}.
\end{remark}

\begin{remark} It follows from \cite[Proposition 12.1]{GanGurevich} that the archimedean integral can be made nonzero at arbitrary $s=s_0$ if one has some freedom on the choice of $\Phi_\infty$ and $\Psi_\infty$. This shows that the archimedean integral does not vanish identically. However their result is not enough for our purposes as $\Psi_\infty$ is fixed to be in the minimal $K_\infty$-type of $\pi_\infty^{3,3}$. In similar cases (e.g. \cite{Moriyama}, \cite{PollackG2}), archimedean integrals are calculated once one has an explicit description of the model involved, by solving the system of differential equations deduced from Schmid operators. Since the dimension of Fourier coefficients of type $(4 \, 2)$ is infinite, it seems out of reach to have such explicit descriptions and unfortunately we have not been able to calculate the archimedean integral in the case of interest.  
\end{remark}

\subsection{A remark on the nonvanishing of the regulator}\label{subsub:non-vanishing}

We have the following direct consequence of Theorem \ref{regulator3} regarding the nonvanishing of motivic cohomology. \\

 Let $N$ denote the positive integer defined as the product of prime numbers $\ell$ such that $\pi_\ell$ is ramified.
The fact that $\pi_\infty$ is cohomological implies that there exists a number field $L$ whose ring of integers $\mathcal{O}_L$ contains the eigenvalues of the spherical Hecke algebra $\mathcal{H}^{sph, N}$ away from $N$ and with coefficients in $\mathbb{Z}$ acting on $\bigotimes'_{\ell \nmid N} \pi_\ell^{\G(\Z_\ell)}$. We let $\theta_\pi: \mathcal{H}^{sph, N} \to L$ denote the Hecke character of $\pi$ and we let $\mathfrak{m}_\pi := {\rm ker}(\theta_\pi)$.

\begin{corollary} \label{nonvanish}
We assume that there exists an automorphic representation $\pi$ of $\G(\A)$ and a cusp form $\Psi$ in $\pi$, which satisfy all the running assumptions of Theorem \ref{regulator3} and such that \[\lim_{s \to 0}\left(I_\infty(\Phi_\infty, A \cdot \Psi, s)  {L}^\Sigma(\pi, {\rm Spin}, s) \right) \ne 0. \]
Then, the class ${\rm Eis}_{\mathcal{M}}(\Phi_f) $ is nontrivial and thus the localization $H^{7}_\mathcal{M}(\Sh_\G(U), \overline{\Q}(4))_{\mathfrak{m}_\pi}$ is nonzero.
\end{corollary}
 
By assuming spin functoriality for $\pi$ and by using a generalization of the prime number theorem due to Jacquet and Shalika \cite{JacquetShalika}, we can improve the corollary by establishing that the nonvanishing of our pairing relies on the nonvanishing of the archimedean integral, as we now explain.

Let ${\rm Spin}:{}^L\G \to {}^L\GL_8$ denote the homomorphism between $L$-groups induced by the spin representation ${\rm Spin}: \GSpin_7 \to \GL_8$.
Roughly, Langlands' functoriality would predict the existence of a \emph{spin lift} to $\GL_8$ for $\pi$, i.e. the existence of an automorphic representation $\Pi$ of $\GL_8(\A)$ whose $L$-parameter $\phi_{\Pi_v}$ at each place $v$ is obtained from composing the $L$-parameter $\phi_{\pi_v}$ of $\pi_v$ with ${\rm Spin}$, thus implying that \[ L(\pi, {\rm Spin},s)=L(\Pi,s), \]
where the latter denotes the standard $L$-function associated with $\Pi$.
\begin{remark} \label{nonvanish2}
 In \cite{KretShin}, a \emph{potential} version of spin functoriality is discussed and proved. If we further assume that $\pi_\infty$ is spin-regular (cf. \cite{KretShin}), the result \cite[Theorem C]{KretShin}, which builds upon \cite[Theorem A]{potentialautomorphy}, produces
a cuspidal automorphic representation $\Pi$ of $\GL_8(\A_F)$, over a finite totally real extension $F/\Q$, with the desired properties. For instance, at each finite place $v$ of $F$ above an odd prime $p \not\in \Sigma$ one has $L(\Pi_v,s)=L(\pi_p,{\rm Spin}, s)$.
 \end{remark}

Before stating our final result, we need to fix the components at bad finite primes of the Schwartz function $\Phi$ appropriately. In contrast to Corollary \ref{regulator3coro}, if $p \in \Sigma$ is a finite place, we let $\Phi_p(-x,-y)$ be the Fourier transform of ${\rm char}((0,1)+p^n\Z_p^2)$, with $n$ the positive integer given by Proposition \ref{pollackshahgsp6} (2) and Remark \ref{remarkonprop5.1} so that for any vector $v\in \pi_p$ and $(U_P, \chi)$-model $\Lambda$, there exists $\tilde{v} \in \pi_p$ such that \[I_p(\widehat{\Phi}_p, \tilde{v},s) = \Lambda(v). \]

\begin{corollary} \label{nonvanish1}
Let $\Psi$ and $\omega_{\Psi}$ be as in Theorem \ref{regulator3}.
Suppose that $\pi$ admits a cuspidal spin lift to $\mathbf{GL}_8$ and $I_\infty(\Phi_\infty, A \cdot\Psi, 1) \ne 0$. Then 
$$
\langle r_\mathcal{D}({\rm Eis}_{\mathcal{M}}(\Phi_f)), \omega_{\widetilde{\Psi}} \rangle =  C_V \cdot I_\infty({\Phi}_\infty, A \cdot\Psi, 1) \cdot {L}^\Sigma(\pi, \mathrm{Spin}, 1)
$$
and it is nonzero. In particular, the class ${\rm Eis}_{\mathcal{M}}(\Phi_f) $ is nonzero and thus the localization $H^{7}_\mathcal{M}(\Sh_\G(U), \overline{\Q}(4))_{\mathfrak{m}_\pi}$ is nonzero. 
\end{corollary}
\begin{proof}
Recall that the Eisenstein $E(g,\Phi,s)$ satisfies the  functional equation \[E(g,\Phi,s)= E(g,\widehat{\Phi},1-s), \]
where $\widehat{\Phi}$ denotes the Fourier transform of $\Phi$. 
This implies that \[I(\Phi,\Psi,s)=I(\widehat{\Phi},\Psi,1-s)=I_\Sigma(\widehat{\Phi}_\Sigma,\Psi,1-s)L^\Sigma(\pi,{\rm Spin}, 1-s). \]
Thanks to our choice of $\Phi$, and mimicking the proof of Corollary \ref{regulator3coro}, we know that there exists $\tilde{\Psi}$ in the space of $\pi$, coinciding with $\Psi$ outside of $\Sigma$, such that  \[I_\Sigma(\widehat{\Phi}_\Sigma,A \cdot \tilde{\Psi},s) =  I_\infty(\Phi_\infty, A \cdot \Psi,s),\]
where we have used that $\widehat{\Phi}_\infty=\Phi_\infty$. Let now $\Pi$ be a cuspidal spin lift of $\pi$ to $\mathbf{GL}_8$. As $L^\Sigma(\pi, \mathrm{Spin}, s) = L^\Sigma(\Pi, s)$, then it does not have a pole at $s = 1$. Combining these facts with Theorem \ref{regulator3}, we get that 
\begin{align*}
    \langle r_\mathcal{D}({\rm Eis}_{\mathcal{M}}(\Phi_f)), \omega_{\widetilde{\Psi}} \rangle &= C_V \cdot  \lim_{s \to 0} \left(I_\Sigma(\widehat{\Phi}_\Sigma,A \cdot \widetilde{\Psi},1-s) L^\Sigma(\pi,{\rm Spin}, 1-s)\right) \\ &=  C_V \cdot  \lim_{s \to 0} \left(I_\infty(\Phi_\infty, A \cdot \Psi, 1-s) L^\Sigma(\pi,{\rm Spin}, 1-s)\right) \\
    &= C_V \cdot I_\infty(\Phi_\infty, A \cdot \Psi, 1) L^\Sigma(\pi, \mathrm{Spin}, 1),
\end{align*} 
proving the first equality of the statement.

Finally, we claim that $L^\Sigma(\pi,{\rm Spin}, s)=L^\Sigma(\Pi, s)$ is non-zero at $s = 1$. In \cite{JacquetShalika}, it is shown that $L(\Pi,s) \ne 0$ for any $s$ with ${\rm Re}(s)=1$. Writing \[L^\Sigma(\Pi,s) = L(\Pi,s) \prod_{p \in \Sigma} L(\Pi_p,s)^{-1}, \]
our claim follows from the fact that each Euler factor $L(\Pi_p,s)$ has no pole in the region ${\rm Re}(s)\geq \frac{1}{2}$ (e.g. \cite[p. 317]{RudnickSarnak}). The result follows.
\end{proof}

\bibliographystyle{alpha}
\bibliography{bibliography}
\end{document}